\documentclass[reqno]{amsart}
\usepackage[utf8]{inputenc}
\usepackage{amsmath,amsfonts,amssymb,amsthm,dsfont, bm}
\usepackage{array}
\usepackage{enumitem}

\usepackage{float}
\usepackage[usenames,dvipsnames]{xcolor}
\usepackage[colorlinks=true]{hyperref} 
\usepackage{verbatim}
\usepackage{graphicx,overpic}
\usepackage{mathabx}
\usepackage{scalerel}
\usepackage{booktabs}

\hypersetup{linkcolor=Violet,citecolor=PineGreen}
\newtheorem{thm}{Theorem}[section]

\newtheorem{prop}[thm]{Proposition}

\theoremstyle{definition}
\newtheorem{defn}[thm]{Definition}
\newtheorem{exmpl}[thm]{Example}
\newtheorem{rmk}[thm]{Remark}
\numberwithin{equation}{section}

\newcommand{\A}{\mathcal{A}}

\newcommand{\N}{\mathbb{N}}
\newcommand{\R}{\mathbb{R}}

\newcommand{\ep}{\varepsilon}

\newcommand\note[1]{\marginpar{\flushleft\sffamily\tiny #1}}

\newcommand{\PreserveBackslash}[1]{\let\temp=\\#1\let\\=\temp}
\newcolumntype{C}[1]{>{\PreserveBackslash\centering}p{#1}}
\newcolumntype{R}[1]{>{\PreserveBackslash\raggedleft}p{#1}}
\newcolumntype{L}[1]{>{\PreserveBackslash\raggedright}p{#1}}

\newcommand{\nin}{{n\in\N}}

\newcommand{\dist}{\mathrm{dist}}

\newcommand{\BA}{\mathcal{B(A)}}

\newcommand{\BS}{S_\mathcal{B(A)}}

\newcommand\norm[1]{\left\lVert#1\right\rVert}
\DeclareMathOperator*{\argmax}{argmax}
\graphicspath{{images/}{../images/}}

\makeatletter
\@namedef{subjclassname@2020}{%
  \textup{2020} Mathematics Subject Classification}
\makeatother

\usepackage{subcaption}
\begin{document}
\title[Resilience of dynamical systems]
{Resilience of dynamical systems}
\author[H.~Krakovská]{Hana Krakovská}
\author[C.~Kuehn]{Christian Kuehn}
\author[I.P.~Longo]{Iacopo P. Longo}

\email[Hana Krakovská]{hana.krakovska@meduniwien.ac.at}
\email[Christian Kuehn]{ckuehn@ma.tum.de}
\email[Iacopo P. Longo]{longoi@ma.tum.de}

\address[H.~Krakovská]{Slovak Academy of Sciences, Institute of Measurement Science, Department of Theoretical Methods,
Dúbravská cesta 9, 
841 04 Bratislava 4, 
Slovakia.\vspace{0.1cm}
Medical University of Vienna, Section for Science of Complex Systems, Spitalgasse 23, 1090 Vienna, Austria.}

\address[C. Kuehn, I.P. Longo]{Technische Universit\"at M\"unchen,
Forschungseinheit Dynamics,
Zentrum Mathematik, M8,
Boltzmannstraße 3,
85748 Garching bei M\"unchen, Germany.}

\thanks{HK was supported by the Slovak Grant Agency for Science under the grant No.2/0023/22. CK was partly supported by a Lichtenberg Professorship of the VolkswagenStiftung. CK also acknowledges partial support of the EU within the TiPES project funded the European Unions Horizon 2020 research and innovation programme under grant agreement No. 820970. IPL was partly supported  by European Union’s Horizon 2020 research and innovation programme under the Marie Skłodowska-Curie grant agreement No 754462, by MICIIN/FEDER project RTI2018-096523-B-100, by TUM International Graduate
School of Science and Engineering (IGSSE) and by the University of Valladolid under project PIP-TCESC-2020.
}

\subjclass[2020]{37C20 34D10 34D05 37G35 37N99}
\date{}
\begin{abstract}
Stability is among the most important concepts in dynamical systems. Local stability is well-studied, whereas determining  the ``global stability'' of a nonlinear system is very challenging. Over the last few decades, many different ideas have been developed to address this issue, primarily driven by concrete applications. In particular, several disciplines suggested a web of concepts under the headline ``resilience''. Unfortunately, there are many different variants and explanations of resilience, and often the definitions are left relatively vague, sometimes even deliberately. Yet, to allow for a structural development of a mathematical theory of resilience that can be used across different areas, one has to ensure precise starting definitions and provide a mathematical comparison of different resilience measures. In this work, we provide a systematic review of the most relevant indicators of resilience in the context of continuous dynamical systems, grouped according to their mathematical features. The indicators are also generalized to be applicable to any attractor. These steps are important to ensure a more reliable, quantitatively comparable and reproducible study of resilience in dynamical systems. Furthermore, we also develop a new concept of resilience against certain nonautonomous perturbations to demonstrate how one can naturally extend our framework. All the indicators are finally compared via the analysis of a classic scalar model from population dynamics to show that direct quantitative application-based comparisons are an immediate consequence of a detailed mathematical analysis.
\end{abstract}
\keywords{Resilience, dynamical systems, attractors, stability.}
\maketitle

\section{Introduction}\label{secintro}
In 1973, Holling~\cite{holling1973} prompted the mathematical community in ecology to investigate a new concept of  ``nonlocal'' stability to capture the ``persistence of relationships within a system, [...] the ability to absorb changes of state variables, driving variables, and parameters, and still persist'', and by doing so, complement the classic Lyapunov stability. He named this property \emph{resilience}. The theory of dynamical systems still proved to be the most promising theoretical framework to  treat the question rigorously (see Gruemm~\cite{gruemm1976}). Nevertheless, the newly introduced concept elicited a tremendous, and yet mostly disorganized, research effort  in several areas of applied science. The result is a plethora of different indicators of resilience, in setups which are often not directly comparable  (see Baggio et al.~\cite{baggio2015},  Brand and Jax~\cite{brand2007}, Carpenter et al.~\cite{carpenter2001}, Donohue et al.~\cite{donohue2016}, Grimm and Wissel~\cite{grimm1997}, K\'{e}fi et al.~\cite{kefi2019}, Meyer~\cite{meyer2016}, Van Meerbeek et al.~\cite{vanmeerbek2021} and Walker and Salt~\cite{walker2012}).\par\smallskip 

It is possible to categorize most of the available indicators within two main branches: \emph{engineering resilience}  and \emph{ecological resilience} (Peterson~\cite{peterson1998}).  The former aims at answering the question ``how long will it take for a system that suffers a perturbation to readjust to its original state?'' and can be traced back to Maynard Smith~\cite{smith1968mathematical} and to May~\cite{may2019}. The latter aims at answering the question ``which are the features of the maximal perturbations that a system can endeavour and still return to its original state?'' and it is directly related to  Holling~\cite{holling1973}. The indicators treated in this paper are listed in Table~\ref{indicators_summary} according to this classical subdivision\par\smallskip
\renewcommand{\arraystretch}{1.4}
\begin{table}[h]
	\begin{tabular}{@{} L{4.2cm} C{0.4cm} C{0.7cm} | L{4.2cm} C{0.4cm} C{0.7cm} @{}}   \hline \toprule
	\multicolumn{3}{c}{\textnormal{\textbf{Engineering resilience}}}  & \multicolumn{3}{c}{\textnormal{\textbf{Ecological resilience}}}\\\cmidrule(lr){1-3} \cmidrule(lr){4-6}
		Characteristic return time$^\S$ &	$
		 T_R$ &\ref{local_char_time} &Latitude in width$^\S$& $L_w$ &\ref{def_lw}\\	
		Reactivity$^\S$  &	$R_0$  &\ref{reactivity_def}&	Distance to threshold$^\S$& $DT$ &\ref{def_DT}\\
        Max amplification$^\S$& $\rho_{max}$ &\ref{max_ampl_def} & Precariousness$^{\dagger\ddagger}$ & $P_\A$&\ref{def:precariousness}\\
         Max amplification time$^\S$& $t_{max}$ &\ref{max_ampl_def} &Latitude in volume$^\S$ & $ L_v$ &\ref{latitude_volume} \\
       	 Stochastic invariability$^\dagger$ & $\mathcal{I}_S$ &\ref{def_int_stoch_invar}& Basin stability$^\S$ & $S_\mathcal{B}$&\ref{basin_stab_def}\\ 
         Deterministic invariability$^\dagger$
		&$\mathcal{I}_D$ &\ref{def_int_det_invar}&	Resistance (potential)$^\S$  & $W$ &\ref{potential_def}\\
        	Average return time$^\S$ & $\!\!\!\!T_R^{mean}$ &\ref{return_time_general}&	Flow-kick resilience$^{\S\dagger} $ &--&\ref{def_flowkick_boundary}\\
	 Maximal return time$^\S$ & $\!\!\!\!T_R^{sup}$ &\ref{return_time_general}&	Intensity of attraction$^\dagger$ & $\mathcal{I}$ &\ref{intensity_def}\\
		
		Resistance (Harrison)$^{\dagger\ddagger}$&$R$ &\ref{hison_res}&	Expected escape times$^\dagger$& $\tau$&\ref{defn:exp_escape_times}\\
		Elasticity$^{\dagger\ddagger}$ & $E$ &\ref{def:elasticity}& Distance to bifurcation$^{\ddagger}$&$D_{bif}$&\ref{dist_to_bif_def}\\
		&&&Persistence$^{\dagger\ddagger}$ &$P$&\ref{def:persistence}\\
		&& &Rate-induced tipping$^{\dagger\ddagger}$&$D_{\rm{rate}}$&\ref{dist_to_bif_rate}\\
		\hline
		\toprule
	\end{tabular}
	\caption{Summary of the indicators reviewed in this work according to the ample categories of engineering and ecological resilience (see Peterson~\cite{peterson1998}). The presentation in the paper privileges a subdivision based on the mathematical techniques which they entail. The notation and number of
	definition are displayed next to the name of each indicator. The superscripts indicate the type of perturbation for which each indicator is designed: ($\S$) perturbations of initial conditions, ($\dagger$) time-dependent perturbations, ($\ddagger$) perturbations of parameters.}
	\label{indicators_summary}
\end{table}

In this paper we have three main objectives. Firstly, we carry out an extensive review of indicators of resilience within the framework of continuous dynamical systems theory. While doing so, we provide a systematic organization based on the mathematical techniques which they entail. To our knowledge, such a comprehensive effort does not exist in the literature, nor for dimension, nor in the organizational spirit. As an immediate effect, this paper provides an easier access to a very vast and fragmented literature. Secondly, we aim at increasing  the applicability of this theory by generalizing  these definitions (wherever possible) to general attractors, whereas they are usually designed  only for stable equilibria and sometimes periodic orbits. However,  equilibria and periodic orbits, while playing a fundamental role in the understanding of many phenomena, are far from exhausting the dynamical possibilities in higher dimensional and complex systems. In particular, this applies to the local theory based on the first order approximation of the dynamics near a hyperbolic trajectory. In this context, it is natural to use the notion of exponential dichotomy from the theory of nonautonomous linear systems. Incidentally, we note that this fact encourages the investigation of resilience also in nonautonomous nonlinear dynamical systems, although, for the sake of simplicity, we limit our presentation to the autonomous case. Lastly, we enrich the discussion  by bringing together recent accomplishments in the theory of nonautonomous/rate-induced tipping phenomena, where the time-dependent variation of parameters of some systems has to be generalized beyond classical autonomous bifurcations to the nonautonomous setting. \par\smallskip

The paper is organized as follows. In  Section~\ref{sec:prelimin}, we set the notation and recall some notions on continuous dynamical systems. Each of the following four sections addresses a class of indicators of resilience grouped according to the most relevant  mathematical technique upon which they are based: Section~\ref{local_lI_subsection} contains the local indicators obtained via the linearization in the neighbourhood of a hyperbolic trajectory. Section~\ref{secbasinshape} deals with the indicators focusing on geometric features of the basins of attraction. In Section~\ref{sec:transient}, we present all the indicators related to the transient behaviour of a system during and after a perturbation. The variation of parameters and the consequent indicators are treated in Section~\ref{sec:parameters}. In Section~\ref{sec:examples}, we carry out an analysis of resilience of a model from population dynamics, using and comparing the indicators introduced in the previous sections. We close the paper with a final discussion in Section~\ref{sec:conclusions}.

\section{Notation, assumptions and preliminary definitions}\label{sec:prelimin}
Although the content of this paper applies to any metric space $(X,d)$, for the sake of clarity and accessibility for applications, we will restrict to  the $N$-dimensional Euclidean space $\R^N$ with norm $|\cdot|$. When $N=1$, we will simply write $\R$ and the symbol $\R^+$ will denote the set of nonnegative real numbers. The symbol $\R^{N\times M}$ represents the set of matrices with $N$ rows  and $M$ columns, and given $A\in\R^{N\times M}$, $A^\top$ will denote its transpose. Moreover, the symbol $\|A\|_{op}$ will represent the operator norm (induced by the Euclidean norm) of the matrix $A$. As a rule, a singleton $\{x\}$ in $\R^N$ will be identified with the element $x$ itself and, given $E\subset\R^N$, the symbol $\overline E$ will denote the closure of $E$ with respect to the topology induced by the Euclidean norm.
Moreover, by $B_r(p)$, we denote the open ball of $\R^N$ centered at $p$ and with radius~$r$. If $p$ is the origin of $\R^N$, we simplify the notation to $B_r$. On the other hand,  given any set $E\subset\R^N$, $B_r(E)$ will denote the set of points $x\in\R^N$ such that $x\in B_r(p)$ for some $p\in E$. Furthermore, for any 
$U\subseteq\R^M$ and any $W\subset\R^N$, 
$\mathcal{C}(U,W)$ will denote the space of continuous functions from $U$ to $W$ endowed with the usual supremum norm $\|\cdot\|_{\mathcal{C}(U,W)}$.
\par\smallskip
Unless otherwise stated, we will deal with
a continuous dynamical system on $\R^N$ which we will identify with its continuous (local) \emph{flow}
\begin{equation*}\label{eq:flow}
\phi:\mathcal{U}\subset\R\times\R^N\to\R^N,\quad(t,x_0)\mapsto \phi(t,x_0).    
\end{equation*}
Moreover, given $E\subset\R^N$ and $t>0$, we use $\phi(t,E)$ to denote the set
\[
\phi(t,E)=\big\{y\in\R^N\mid y=\phi(t,x) \text{ for some }x\in E\big\}. 
\]
In particular, we recall that a set $E\subset\R^N$ is called \emph{forward invariant} under the flow $\phi$ if $\phi(t,E)\subset E$ for all $t>0$ and \emph{invariant} if $\phi(t,E)=E$ for all $t\in\R$.

Where necessary, we will assume that the continuous flow $\phi$ is induced by an autonomous ordinary differential equation of the form 
\begin{equation}\label{eq:autonomous}
    \dot x=f(x),\quad x \in \R^N,
\end{equation}
where $f\colon \R^N\to\R^N$ is regular enough so that the initial value problem $x(0)=x_0\in\R^N$ admits a unique solution $x(\cdot,x_0)\in C(I_{x_0},\R^N)$, with $I_{x_0}$ being its maximal interval of definition. It is well-known that~\eqref{eq:autonomous} induces a continuous (local) flow on $\R^N$ via the relation
\begin{equation*}
 \phi(t,x_0)=x(t,x_0).    
\end{equation*}
\par\smallskip

We will also assume that the considered dynamical system has a local attractor $\A\subset \R^N$ in the sense of the definition below (which is taken from Kloeden and Rasmussen~\cite{kloeden2011}).
\begin{defn}[Local attractor and basin of attraction]\label{attractor_def}
    A compact set $\A\subset \R^N$ invariant under a flow $\phi$ on $\R^N$ is called a \emph{local attractor} if there exists $\eta>0$ such that for all $x_0\in B_\eta (\A)$, $x(\cdot,x_0)$ is defined for all $t>0$ and  
    \[
    \lim_{t\to\infty}\dist\big(\phi(t,B_\eta(\A)),\A\big)=0,
    \]
where for all $A,B\subset\R^N$, $\dist(A,B)$ denotes the Hausdorff semi-distance between $A$ and $B$. We call the  \textit{basin of attraction} of $\A$ the set
\[
\BA=\{x_0 \in \R^N:\emptyset \neq \omega(x_0)\subset \A\}.
\]
where, given $E\subset\R^N$, $\omega(E)$ represents the \emph{omega limit set of $E$} under the flow $\phi$, i.e.~ the set
\[
\omega(E):=\bigcap_{t\ge0}\overline{\bigcup_{s\ge t}\phi(s,E)}.
\]
\end{defn}

\begin{rmk}\label{rmk:different_attractors}
     Several alternative definitions of local attractor are available in the literature. 
    For example, the above definition is sometimes completed by the condition of minimality:  ``there is no proper subset of $\A$ which satisfies the properties in Definition~\ref{attractor_def}''. In order to understand the implication of such condition, consider the two-dimensional differential system in polar coordinates $(r,\varphi) \in [0,\infty)\times[0,2\pi)$ where $0$ is identified with $2\pi$,
    \begin{equation}\label{eq:28/11-12:54}
    \dot{r}=r(r-1)(r-3), \qquad
	\dot{\varphi}=1.
    \end{equation}
   \begin{figure}[h]
       \centering
        \includegraphics[scale=0.25,trim={0 0cm 0 0cm},clip]{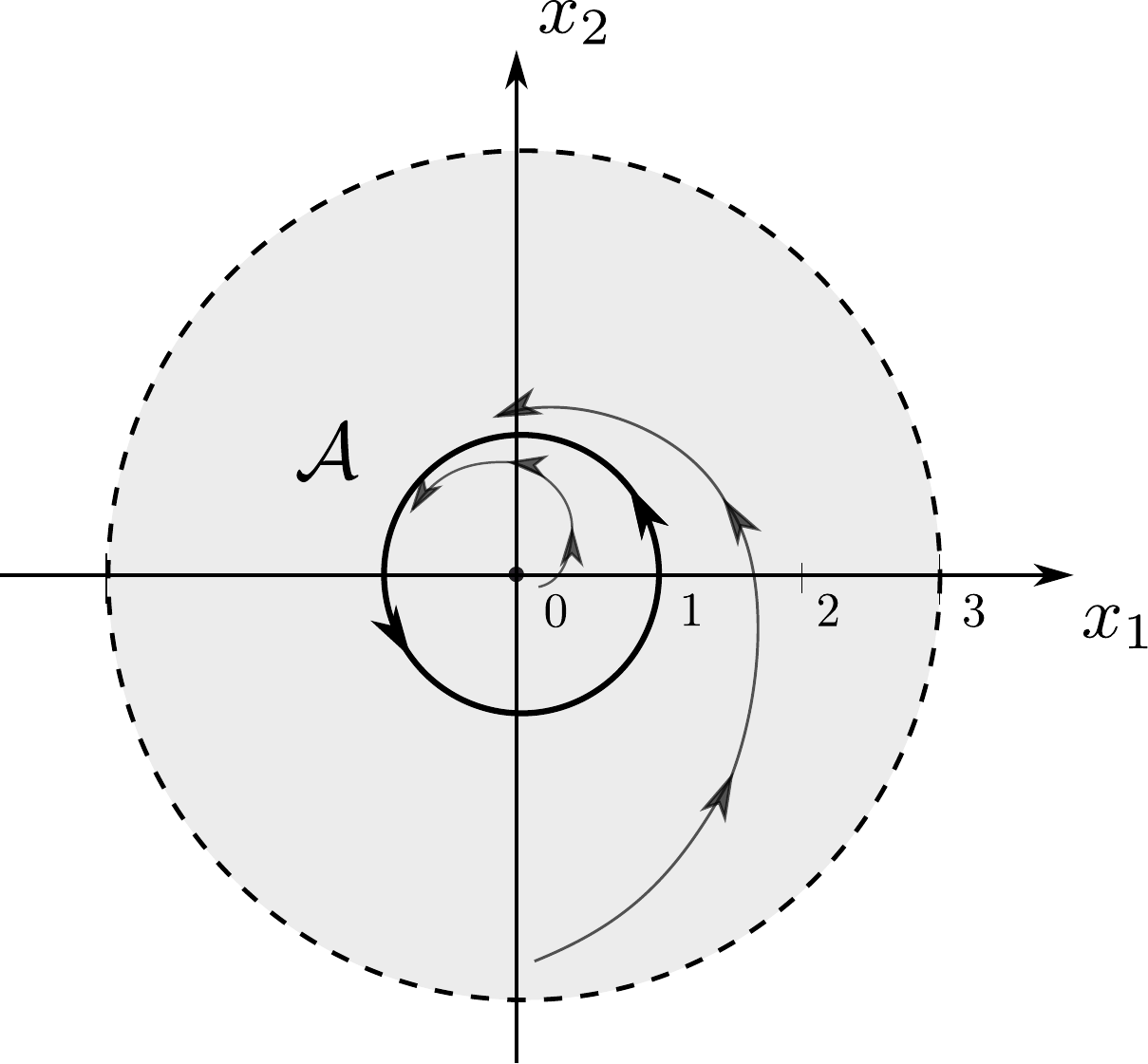}
    \caption{Sketch of the phase space for~\eqref{eq:28/11-12:54}. According to Definition~\ref{attractor_def}, one can choose either the whole closed ball of radius one, the origin and the periodic orbit of radius one, or just the latter as a local attractor for the induced dynamical system. Under the minimality condition in  Remark~\ref{rmk:different_attractors}, only the periodic orbit of radius one qualifies for being a local attractor and the origin belongs to the boundary of the basin of attraction.}
       \label{fig:attractor}
   \end{figure}
   It is easy to show that~\eqref{eq:28/11-12:54} has an unstable equilibrium at the origin and two hyperbolic periodic orbits of radius $r=1$ (stable) and radius $r=3$ (unstable), respectively (see  Figure~\ref{fig:attractor} for a sketch of the phase space).  
    If one employs the minimality condition, the local attractor of the system consists of the points in the periodic orbit of radius $r=1$ and its basin of attraction is the open ball of radius $r=3$ with the exception of the origin. In spite of belonging to the boundary of the basin of attraction, the origin is merely an isolated point. Every solution starting in a sufficiently small ball centered at the origin, except for the origin itself, will eventually converge to the periodic orbit of radius $r=1$. If the minimality condition is not taken into account, we are allowed to choose as a local attractor the subset of $\R^2$ made of the points in the periodic orbit of radius $r=1$ and also the origin. Its basin of attraction is now the whole open ball of radius $r=3$. It is worth noting that an alternative approach is given by the Milnor's definition of local attractor~\cite{milnor1985}, where the asymptotic behaviour of solutions starting in negligible sets (with respect to the Lebesgue measure in this context) are disregarded. In this case, the local attractor would be made of the points in the sole periodic orbit of radius $r=1$. In order to avoid measure theoretic arguments as much as possible, and in view of the nonlocal nature of the concept of ecological resilience, in this work we privilege the flexibility provided by Definition~\ref{attractor_def}. Note also that, although in this example the boundary of the basin is an unstable periodic orbit, this is not always the case (see Figure~\ref{fig:Duffing}). 
    \end{rmk}

\begin{figure}
       \centering
        \begin{overpic}[width=0.65\textwidth]{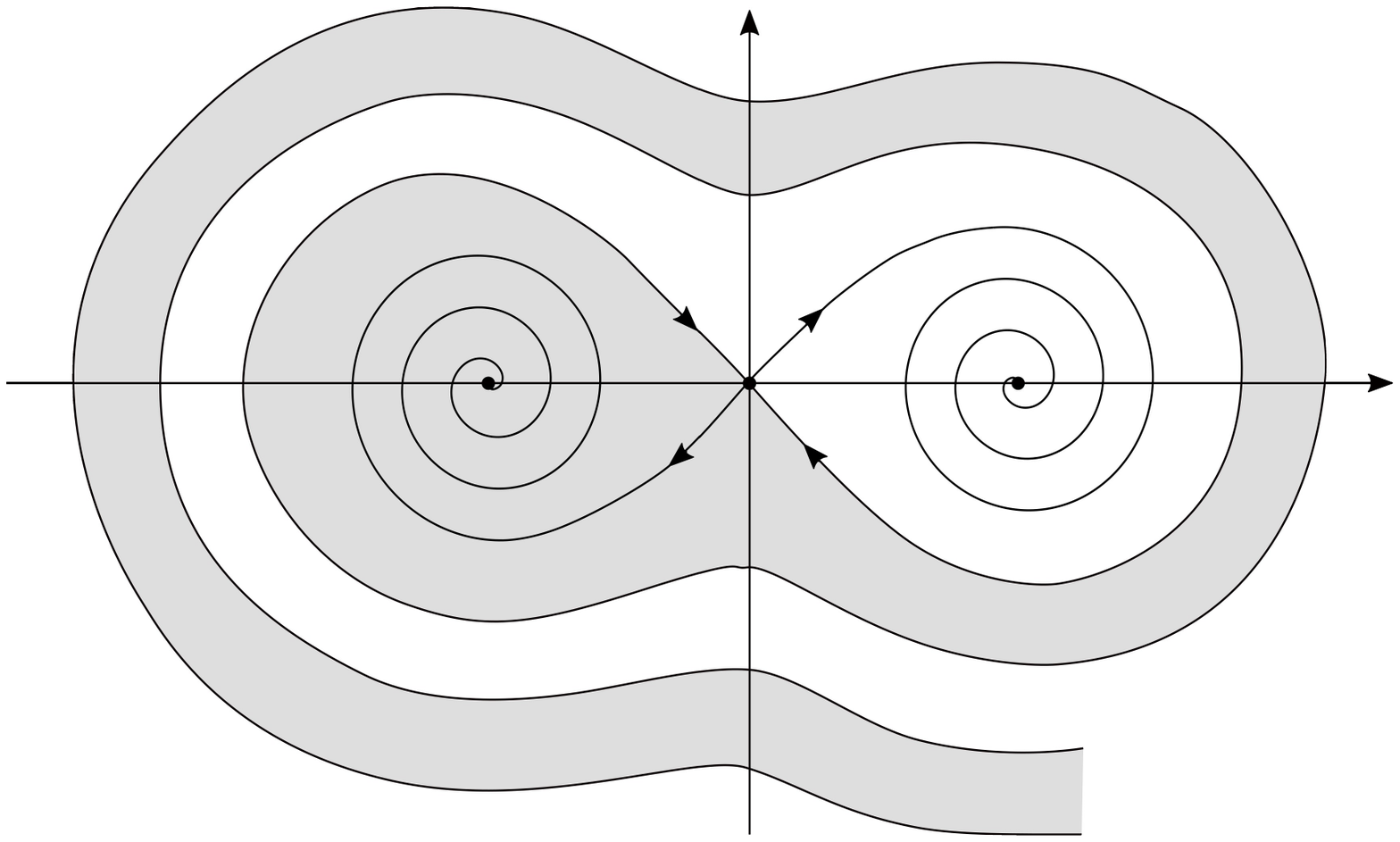}
        \put(97,40){$x$}
        \put(55,63){$y$}
        \end{overpic}
    \caption{The unforced Duffing oscillator is a classic example given by the equations $\dot x=y$, $\dot y=x-x^3-\delta y$, where $\delta\ge0$. The system has three equilibria, namely $(0,0)$ and $(\pm1,0)$. Easy calculations show that for $\delta>0$, the origin is a saddle and the fixed points at $(\pm1,0)$ are asymptotically stable. In particular, the boundary of the basin of attraction of each one of the latter is given by the stable manifold of the saddle at the origin. 
    }
       \label{fig:Duffing}
   \end{figure}

\begin{defn}[Global Attractor]\label{def:global_attractor}
A compact set $\A\subset \R^N$ invariant under the flow $\phi$ induced by~\eqref{eq:autonomous} is called a \emph{global attractor} if $\A$ attracts each bounded subset of $\R^N$.
\end{defn}
It is easy to show that the global attractor for $\phi$ is the minimal compact set that attracts each bounded subset of $\R^N$, and also that it is the maximal closed and bounded invariant set (see Lemma 1.6 in~\cite{carvalho2012attractors}). As a matter of fact, the global attractor is composed of all the points of $\R^N$ which belong to a bounded global solution for~\eqref{eq:autonomous} (see Theorem 1.7 in~\cite{carvalho2012attractors}).\par\smallskip

We also recall that, under the considered definitions and assumptions, the basin of attraction of a local attractor is an open set. The proof is a consequence of simple topological arguments and it is therefore omitted.
\begin{prop}\label{basin_open_prop}
	Assume the evolution operator $\phi$ is continuous. The basin of attraction $\BA$ of a local attractor $\A$ is an open set.
\end{prop}

 Throughout the document we will sometimes need to use nonautonomous differential equations, i.e. problems of the type 
\[
\dot x=f(t,x),\quad t\in\R,x\in\R^N,
\]
where $f\colon\R\times \R^N\to\R^N$ is regular enough so that the initial value problem $x(t_0)=x_0\in\R^N$, $t_0\in\R$, admits a unique solution $x(\cdot,t_0,x_0)\in C(I_{t_0, x_0},\R^N)$, with $I_{t_0,x_0}$ being its maximal interval of definition and $t_0\in I_{t_0, x_0}$. In order to carry out a dynamical analysis of a nonautonomous differential equation one can either construct a two-parameter semigroup---also known as process---or a skew-product flow, that is an autonomous flow on an extended phase space where the base is a functional space parametrized on time and the fiber is $\R^N$. For a concise introduction to these and further techniques in nonautonomos dynamical systems theory we point the reader to the the first chapter of the book by Kloeden and Pötzsche~\cite{kloeden2013nonautonomous}  (the following chapters contain also several applications of the theory to life sciences); for a more extended presentation we recommend for example the book by Kloeden and Rasmussen~\cite{kloeden2011}. Both techniques are frequently used in more theoretical aspects of nonautonomous dynamics. Yet, the use of these methods is relatively technical, so we decided to focus on the key notions necessary. In particular, we now recall two important notions from nonautonomous dynamical systems theory which will play an important role throughout the work: the one of exponential dichotomy and of hyperbolic solution.
\par\smallskip

An exponential dichotomy consists of a splitting of the phase space of a nonautonomous linear differential equation into solutions that decay exponentially fast to zero either as $t\to\infty$ or as $t\to -\infty$~\cite{coppel2006dichotomies,kloeden2013nonautonomous,kloeden2011}. For autonomous systems such splitting is obtained through the real parts of the eigenvalues (when they are nonzero) of the matrix defining the system and the associated eigenspaces. However, classical examples show that the eigenvalues are generally of no use when the matrix depends on time~\cite{coppel2006dichotomies}. 

\begin{defn}[Exponential dichotomy]\label{def:exp_dichotomy}
Given a locally integrable function $ A:\R\to\R^{N\times N}$, the linear system 
\begin{equation}\label{eq:linear_non_autonomous}
 \dot y= A(t)y   
\end{equation}
is said to have an \emph{exponential dichotomy} on an interval $I\subset\R$ if there are a projection $P$ (i.e.~a matrix $P\in\R^{N\times N}$ satisfying $P^2=P$), and constants $\alpha>0$, $K\ge1$  such that 
\[
\begin{split}
    |Y(t)PY^{-1}(s)|\le Ke^{-\alpha(t-s)},\quad\text{for all }t,s\in I, \, t\ge s,&\text{ and}\\
    |Y(t)(Id-P)Y^{-1}(s)|\le Ke^{\alpha(t-s)},\quad\text{for all }t,s\in I, \,t\le s,&
\end{split}
\]
where $Y:\R\to\R^{N\times N}$ is a fundamental matrix solution of~\eqref{eq:linear_non_autonomous}, and $\textnormal{Id}$ is the identity matrix on $\R^{N\times N}$. 
\end{defn}
The exponential dichotomy on $\R$ fulfils the role that eigenvalues with nonzero real part play in the study of stability of hyperbolic equilibria or periodic orbits. However, it also allows to treat nontrivial time-dependent solutions (if they exist) which have the equivalent role of determining the asymptotic behaviour of solutions in their vicinity.
For the sake of consistency and to avoid introducing further notation, we will present the notion of hyperbolic solutions only for autonomous ordinary differential equations although it is defined for general nonautonomous systems~\cite{kloeden2011}.
\begin{defn}[Hyperbolic solution]\label{def:hyperb_sol}
A globally defined solution $\widetilde x:\R\to\R^N$ of a nonautonomous differential equation $\dot x =f(t,x)$, with $f\colon\R\times\R^N\to\R^N$, $(t,x)\mapsto f(t,x)$ continuously differentiable in $x$ for almost every $t\in\R$, is called {\em hyperbolic\/} if the variational
equation $\dot y= {\textnormal{D}}f\big(t,\widetilde x(t)\big)y$ has an
\emph{exponential dichotomy} on $\R$. In particular, we will call an hyperbolic solution $\widetilde x$ \emph{locally attractive} if the exponential dichotomy has projector the identity $P={\textnormal{Id}}$.
\end{defn}
\begin{rmk}
The recalled notion of hyperbolicity generalizes the one for equilibria and periodic orbits to arbitrary time-dependent solutions. One of the important reasons for presenting it is that hyperbolicity in the sense of Definition~\ref{def:hyperb_sol} is a robust property, in that a differential equation with a hyperbolic solution can be perturbed (within a certain class of perturbations) and the hyperbolic solution persists~\cite{potzsche2011nonautonomous}. For example classic hyperbolic equilibria can be perturbed into hyperbolic nonstationary trajectories. This fact is of key importance in the study of resilience which in fact aims to capture the persistence of certain properties of attractivity as we shall see in due time.
The reader who is interested in a deeper understanding of the extent of such generalization, can refer for example to~\cite[Section 1.4]{johnson2016nonautonomous} and the references cited therein (see for example the combined role of~\cite[Proposition 1.56]{johnson2016nonautonomous} and~\cite[Theorem~III.2.4]{hale1980ordinary}).
On the other hand, the reader who deals with attractors made of equilibria and/or periodic orbits can intend the term hyperbolic in the classic sense.
\end{rmk}

\section{Local Indicators}\label{local_lI_subsection}
Linear stability analysis is one of the classic tools in the study of dynamical systems. It allows to infer the asymptotic dynamics of a system in the surrounding of a reference trajectory by looking at its linear approximation and the dominant Lyapunov exponent of a fundamental matrix solution. If the dominant Lyapunov exponent's real part is nonzero, the sign qualifies the reference trajectory as stable or unstable. Its absolute value measures the asymptotic speed of convergence or divergence after a small perturbation. When the reference trajectory is a stationary state or a periodic orbit, this reduces to the calculation of the eigenvalues of the Jacobian of the initial vector field evaluated on such orbits. \par\smallskip

Such ideas have been used to determine the resilience of an attractor via the rate of convergence of the nearby solutions. Due to the inherent local nature of linear stability analysis, these indicators of resilience overlook the topological structure of the phase space away from the considered attractor and are not designed to determine the highest possible perturbation which a system can absorb before tipping to a different state.\par\smallskip

The section contains five subsections and six indicators. In subsection~\ref{local_char_time_susbec} we present the classic \emph{characteristic return time} which relates the resilience of a system in the nearby of an attractor to the asymptotic rate of convergence of the solutions. Subsections~\ref{reactivity_subsection} and~\ref{amplification_subsec} address respectively the question of local resilience in the short-term horizon and in the transient after a perturbation of the initial conditions. The former contains the indicator  \emph{reactivity} whereas the latter the indicators \emph{maximal amplification} and \emph{maximal amplification time}. Subsection~\ref{subsec_intrinsic_variability} addresses the question of local resilience against time-dependent (random or deterministic) perturbations and contains the indicators \emph{intrinsic stochastic variability} and \emph{intrinsic deterministic variability}. Finally, subsection~\ref{local_discussion_subsection} contains a short discussion of the relations among the previously introduced indicators.
\par\smallskip

Besides the assumptions in Section~\ref{sec:prelimin}, we shall also consider the following assumption.
\begin{enumerate}[label=\upshape(\textbf{H1}),leftmargin=*]
	\item \label{system_l_ind}
	The function $f:\R^N\to\R^N$ in~\eqref{eq:autonomous} is assumed to be continuously differentiable and, for every $x\in\R^N$, ${\textnormal{D}}f(x)$ will denote the Jacobian of $f$ calculated at $x$. Moreover, assume that $\A=\{\widetilde x(t)\mid t\in\R\}$, and $\widetilde x$ is a locally attractive hyperbolic solution for $\dot x =f(x)$.
	\par\smallskip
	
\end{enumerate}
\subsection{Characteristic Return Time}\label{local_char_time_susbec}
The notion of characteristic return time in the context of resilience for ecological systems already dates back to May~\cite{may2019}.
A very commonly used version is due to Beddington et al.~\cite{beddington1976} for discrete dynamical systems and then to Pimm and Lawton~\cite{pimm1977} for the continuous case. It is indistinctly used under  different names, e.g., return time~\cite{pimm1977}, characteristic return time~\cite{pimm1984}, engineering resilience~\cite{gunderson2000, holling1996,standish2014}, resilience~\cite{neubert1997}.
A qualitative description of the underlying idea  is provided by 
Pimm~\cite{pimm1984} as ``how fast the variables return towards their equilibrium following a perturbation", or, more specifically, as the ``time taken for a perturbation to return to $1/e$ of its initial value". The definition is motivated by the fact that a trajectory starting in the nearby of a locally stable equilibrium $x^*$ will approach it in a time which is proportional to the reciprocal of the eigenvalue with largest real part for the system linearization at $x^*$. The presentation below is however given for the more general case of a locally attractive hyperbolic solution (see Definition~\ref{def:hyperb_sol}).

\begin{defn}[Characteristic return time]
	\label{local_char_time}
	Consider $f:\R^N\to\R^N$ satisfying assumption~\ref{system_l_ind}. The \textit{characteristic return time} $T_R$ of the system $\dot x=f(x)$ for the attractor $\A$ is defined as 
	\[
	T_R(\A)= \frac{1}{\widehat \alpha},
	\] 
    where  
	\[
    \widehat \alpha:=\inf\big\{\alpha>0\ \mid\  |Y(t)Y^{-1}(s)|\le Ke^{-\alpha(t-s)},\quad\text{for all }t\ge s\big\},
    \]
    and $Y:\R\to\R^{N\times N}$ is a fundamental matrix solution of $\dot y ={\textnormal{D}}f\big(\widetilde x(t)\big)y$. 
\end{defn}	
\begin{rmk}\label{rmk:char_return_equil}
    If the considered local attractor $\A= x^*$ is an attractive hyperbolic fixed point, then $\widehat \alpha$ coincides with the opposite of the real part of the dominant eigenvalue of ${\textnormal{D}}f(x^*)$, i.e.~$\widehat \alpha=-Re\big(\lambda_{dom}({\textnormal{D}}f(x^*))\big)$ .
\end{rmk}

The definition of characteristic return time motivated the introduction of the following indicator of resilience for a stable hyperbolic equilibrium as the rate of decay 
	\[EV(\A)=\widehat \alpha.\]
This indicator has been widely used and studied for both continuous and discrete systems in the case that $\A=x^*$ is hyperbolic and attracting (see for example Arnoldi et al.~\cite{arnoldi2016}, DeAngelis et al.~\cite{deangelis1989,deangelis1989b}, Harwell and Ragsdale~\cite{harwell1979}, Pimm and Lawton~\cite{pimm1978}, Rooney et al.~\cite{rooney2006}, Van Nes and Scheffer~\cite{van2007}, Vincent and Anderson~\cite{vincent1979}).

\subsubsection*{Invariance with respect to change of coordinates} 
The characteristic return time is invariant with respect to change of basis $z=Qy$, with $Q$ nonsingular. Indeed, considering the principal matrix solutions $U(t,s)$ and $V(t,s)$  of $\dot y=A(t)y$ and $\dot z=QA(t)Q^{-1}y$ respectively, and $y_0\in\R^N$, we have that
$Q^{-1}V(t,s)Qy_0=U(t,s)y_0$. Therefore, if  $\dot y=A(t)y$ has an exponential dichotomy on $\R$ with projector the identity and constants $\alpha>0$ and $K\ge 1$, then we have that 
\[
|V(t,s)|=|QU(t,s)Q^{-1}|\le \widetilde Ke^{-\alpha(t-s)}.
\]
For autonomous systems, this fact becomes even more evident because $A$ and $QAQ^{-1}$ have the same set of eigenvalues.

\subsection{Reactivity}\label{reactivity_subsection}

Neubert and Caswell~\cite{neubert1997} proposed and studied different indicators with the aim of capturing the transient behavior of a trajectory starting in the neighbourhood of a stable equilibrium as it may substantially differ from the asymptotic one. Specifically, the \textit{reactivity} corresponds to the maximum instantaneous rate at which an asymptotically stable linear homogeneous system responds if the initial condition is taken outside the origin. The reactivity of a nonlinear system $\dot x = f(x)$ in the neighbourhood of a stable hyperbolic equilibrium $x^*$ is obtained through its linearization at $x^*$.
\begin{defn}[Reactivity] 
	\label{reactivity_def}
	Let $\dot y= A(t)y$, with $A:\R\to\R^{N\times N}$ locally integrable, be an asymptotically stable linear homogeneous system, and denote by  $y(\cdot,t_0,y_0)$ its unique solution satisfying $y(t_0,t_0,y_0)=y_0$. We shall call the \textit{reactivity} of the system at time $t_0\in\R$, the quantity 
	\begin{equation}\label{reactivity_eq}
	R_{t_0}=\max_{|y_0|\neq 0} \left( \frac{1}{|y(t,t_0,y_0)|}\frac{d|y(t,t_0,y_0)|}{dt} \right)\bigg|_{t=t_0}.
	\end{equation}
	The system $\dot y= A(t)y$ is called \emph{reactive} if there is $t_0\in\R$ such that $R_{t_0}>0$ and nonreactive otherwise.
	A nonlinear system $\dot x =f(x)$ satisfying Assumption~\ref{system_l_ind} is called reactive if there exists a neighbourhood of a locally attractive hyperbolic solution $\widetilde x$ such that $\dot y ={\textnormal{D}}f\big(\widetilde x(t)\big)y$ is reactive.  
\end{defn}

If a system is reactive in a neighbourhood of a locally attractive hyperbolic solution $\widetilde x$, some trajectories starting in a neighborhood of $\widetilde x$  may initially move away from $\widetilde x$, before  converging to it. In other words, the finite-time Lyapunov exponents for $\widetilde x$ can be positive. 
An uninformed guess might relate the short-term behaviour of solutions to the real part of the least stable eigenvalue of $A$. This is true only when $A$ has a set of orthogonal eigenvectors. In such a case, however, a monotonic decay towards zero characterizes all the solutions since the eigenvalues of ${\textnormal{D}}f(x^*)$ determine both the asymptotic and the transient behaviour of the system. In other words, a short-time amplification is a possible effect of nonorthogonality of the eigenvector basis---also called non-normality of $A$. 
Neubert and Caswell~\cite{neubert1997} unveil a relation between the reactivity of a linear homogeneous system $\dot y= Ay$, $A\in\R^{N\times N}$, and the dominant eigenvalue of the symmetric part of the Toeplitz decomposition of $A$ (recall that every real symmetric matrix is Hermitian and therefore its eigenvalues are real).  The argument works also for nonautonomous linear homogeneous problems. Note that 
\begin{equation}\label{eq:reactivity_derivation}
\begin{split}
\frac{d|y(t,t_0,y_0)|}{dt} &=\frac{d}{dt}\sqrt{y(t,t_0,y_0)^\top y(t,t_0,y_0)}\\
&=\frac{\dot y(t,t_0,y_0)^\top y(t,t_0,y_0)+y(t,t_0,y_0)^\top \dot y(t,t_0,y_0)}{2|y(t,t_0,y_0)|} \\
& =\frac{ y(t,t_0,y_0)^\top \big(A(t)^\top +A(t)\big)y(t,t_0,y_0)}{2|y(t,t_0,y_0)|}.
\end{split}
\end{equation}
Therefore, from \eqref{reactivity_eq} we obtain that 
\[
R_{t_0}=\max_{|y_0|\neq 0}\frac{ y_0^\top \big(A(t_0)^\top +A(t_0)\big)y_0}{2|y(t,t_0,y_0)|^2}.
\]
The term on the right-hand side of the previous formula is also known as a Rayleigh quotient and its maximum is attained at the largest eigenvalue of the matrix $\big(A(t_0)^\top +A(t_0)\big)/2$~\cite{horn2012matrix}, that is, 
\[R_{t_0}=\lambda_{dom}\big(H(A(t_0)\big)\in\R,\qquad\text{where } H(A)=\frac{A(t_0)^\top +A(t_0) }{2}.\]

In the context of autonomous non-normal linear operators, reactivity is also known as the numerical abscissa of $A$~\cite{embree2005spectra}.

\subsubsection*{Invariance with respect to change of coordinates} In general, the reactivity is not preserved under a change of basis $z=Qy$, with $Q$ nonsingular. If $Q$ is orthogonal, i.e.~($Q^{-1}=Q^\top $), the eigenvalues of the symmetric parts of $A(t_0)$ and $QA(t_0)Q^{-1}$ are the same and reactivity is conserved.

\begin{exmpl}\label{reac_nonreac_example}
	In this example we compare three planar asymptotically stable linear systems $\dot{x}=A_i x,$ where $i=\{1,2,3\}:$ 
	\[
	A_1=	\begin{pmatrix}
			-2 & 0\\
			5 & -1	
			\end{pmatrix},\ \
			A_2=
			\begin{pmatrix}
			-2 & 1\\
			 \sqrt{26} + \sqrt{50} & - \sqrt{26} - \sqrt{50} - 1	
			\end{pmatrix},\ \
			A_3=
			\begin{pmatrix}
			-2 & 0\\
			1 & -1	
			\end{pmatrix}.
	\]

	\noindent All three systems have the same value of the dominant eigenvalue and consequently the same characteristic return time given as $T_R(0)=1$. The systems $A_1$ and $A_2$ are reactive with a positive value of reactivity given by $R_0=(\sqrt{26} - 3)/2\approx 1.05.$ The system $A_3$ is nonreactive with $R_0=(\sqrt{2} - 3)/2\approx -0.79.$  In the first row of Figure~\ref{non_reac_amplif_fig}, the plots show the time evolution of the magnitudes of trajectories starting from different perturbed initial conditions around the origin. 
	Whereas the trajectories of the reactive systems $A_1$ and $A_2$ can exhibit a transient growth in their magnitude, this is not possible for the nonreactive system $A_3$. Furthermore, this example shows that despite having the same characteristic return time and reactivity indicator values, the transient behaviour of the trajectories in systems $A_1$ and $A_2$ may vary substantially. This motivates the introduction of the so-called \emph{amplification envelope} (see Subsection~\ref{amplification_subsec}). The amplification envelopes of three considered systems are plotted in the second row of Figure~\ref{non_reac_amplif_fig}. 
\begin{figure}[h!]		
\captionsetup[subfigure]{labelformat=empty}
	\centering
	\begin{subfigure}{0.31\textwidth}
		\includegraphics[scale=0.295,trim={3.2cm 9.5cm 0cm 9.5cm},clip]{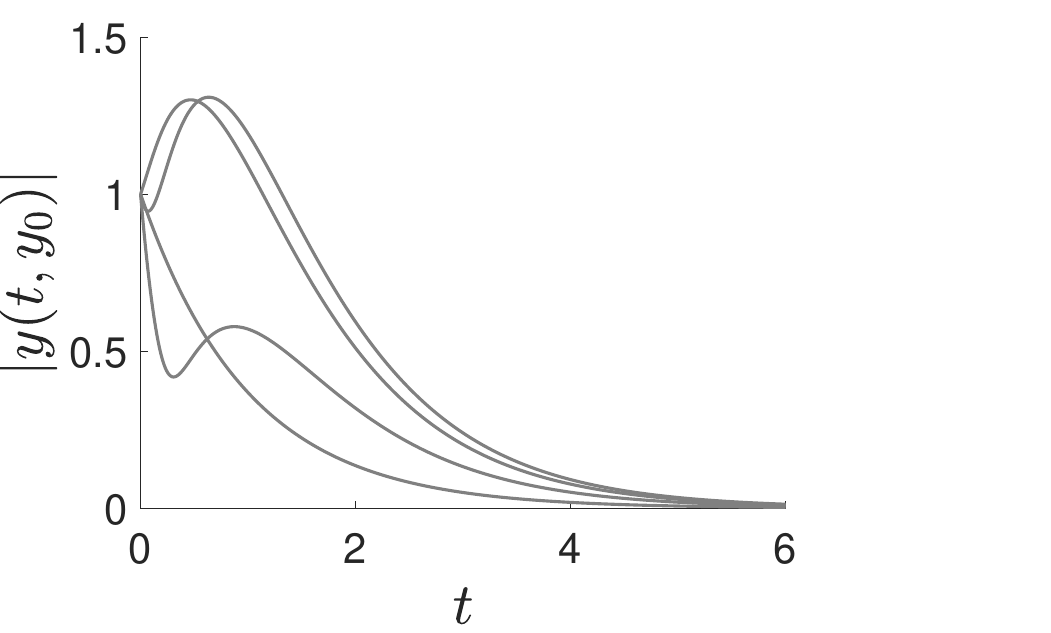}
		\caption{$A_1$}
	\end{subfigure}
	\begin{subfigure}{0.31\textwidth}
		\includegraphics[scale=0.295,trim={2.5cm 9.5cm 0cm 9.5cm},clip]{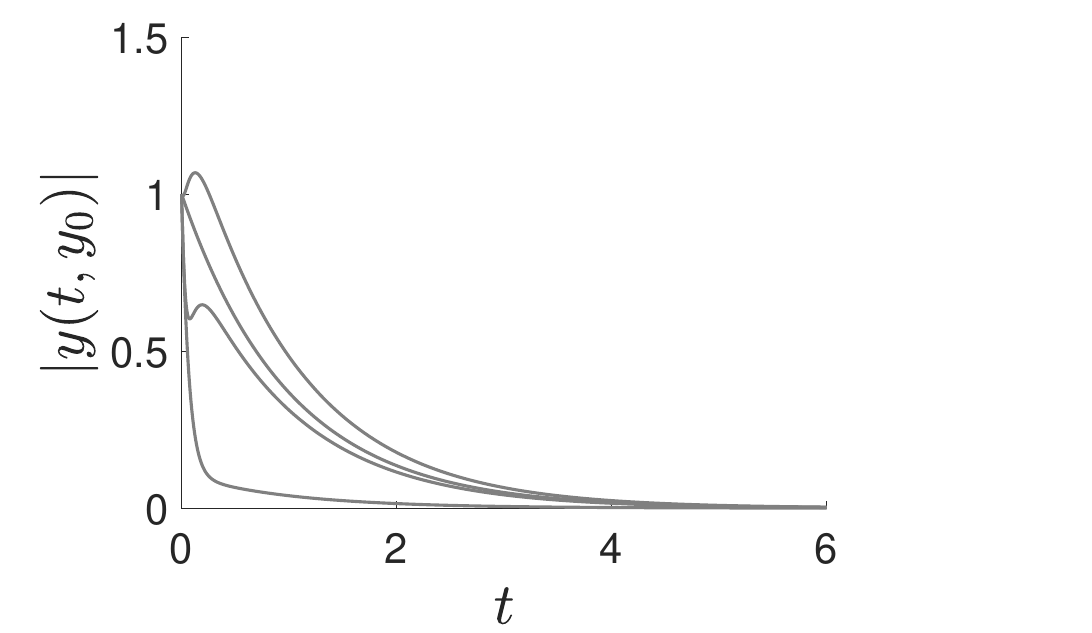}
		\caption{$A_2$}
	\end{subfigure}
		\begin{subfigure}{0.31\textwidth}
		\includegraphics[scale=0.295,trim={2.65cm 9.5cm 0cm 9.5cm},clip]{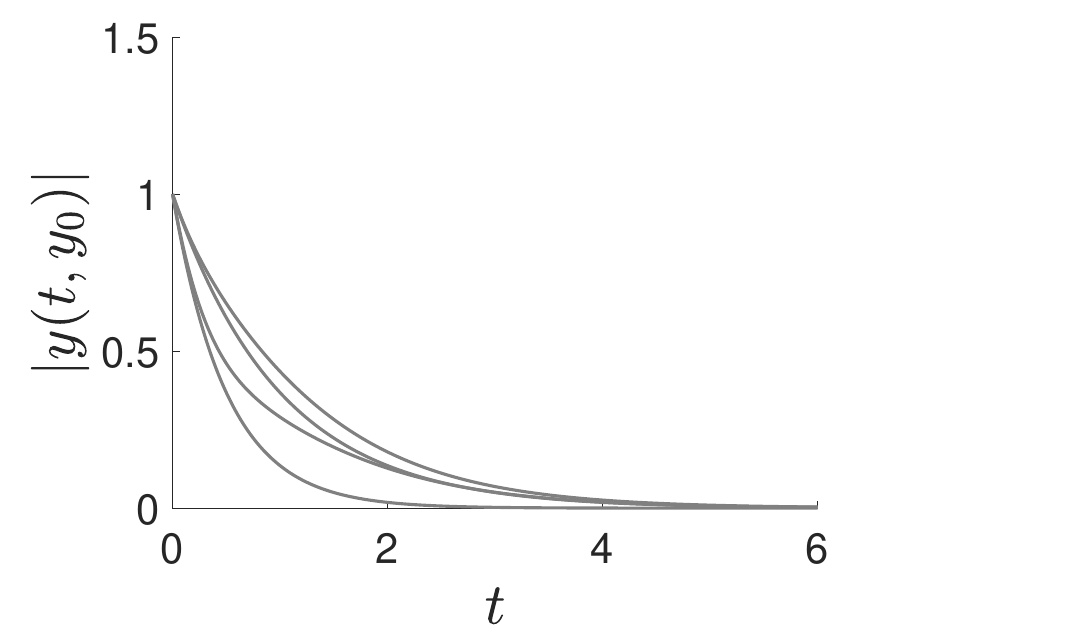}
		\caption{$A_3$}
	\end{subfigure}
	\begin{subfigure}{0.325\textwidth}
		\includegraphics[trim={0cm 0cm 0cm 0.5cm},clip,scale=0.295]{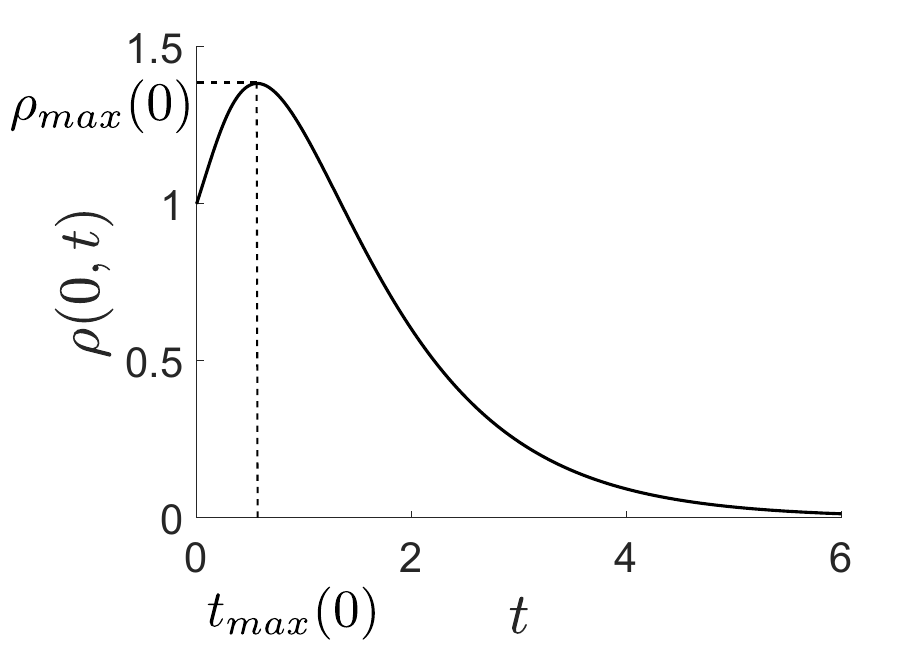}
		\caption{$A_1$}
	\end{subfigure}
	\begin{subfigure}{0.326\textwidth}
		\includegraphics[trim={0cm 0cm 0cm 0.5cm},clip,scale=0.295]{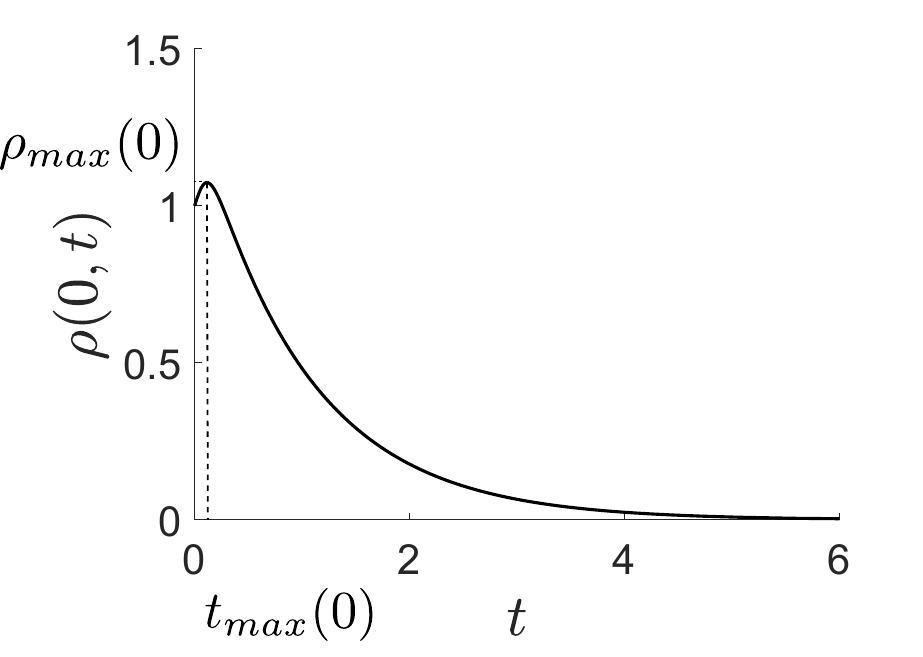}
		\caption{$A_2$}
	\end{subfigure}
		\begin{subfigure}{0.325\textwidth}
		\includegraphics[trim={0cm 0cm 0cm 0cm},clip,scale=0.295]{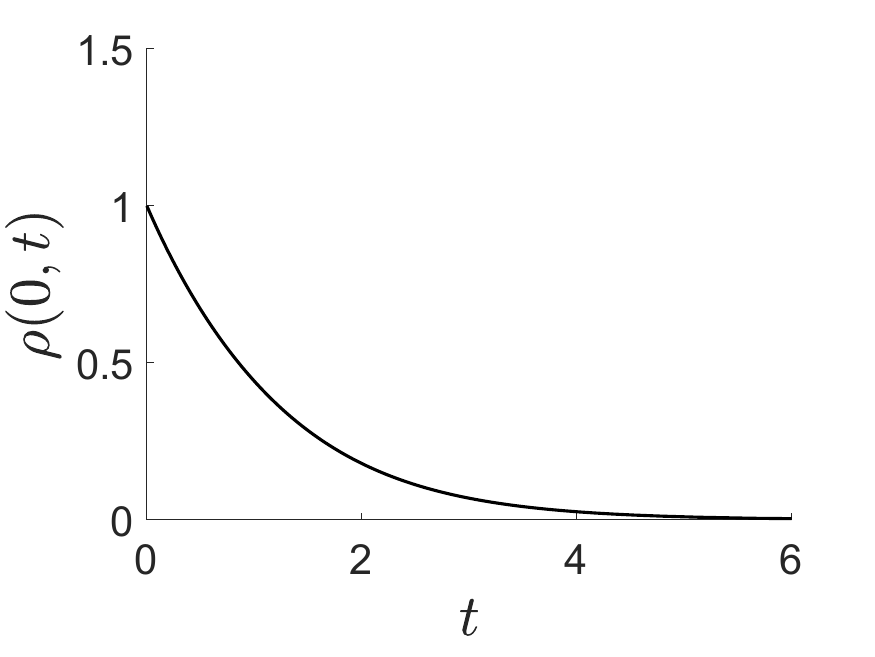}
		\caption{$A_3$}
	\end{subfigure}
	\caption{The transient behaviour of three systems $A_1,A_2,A_3$ from Example~\ref{reac_nonreac_example} is investigated. In the first row, the time evolution of the magnitude of four trajectories starting from different initial conditions around the origin is depicted. In the second row, we see the respective amplification envelopes of each system. Even though the systems have the same characteristic return time, and additionally, the systems $A_1$ and $A_2$ have also the same value of reactivity, the indicator amplification envelope and its characteristics $\rho_{max}$ and $t_{max}$ are able to capture different transient behaviour.}
		\label{non_reac_amplif_fig}
	\end{figure}
\end{exmpl}
\subsection{Maximal amplification and maximal amplification time}\label{amplification_subsec}
Example~\ref{reac_nonreac_example} shows that reactive systems with same reactivity and characteristic return time do not necessarily share the same transient behaviour. 
This fact led Neubert and Caswell~\cite{neubert1997} to the definition of \textit{amplification envelope}, which records the maximal deviation from the attractor for trajectories starting at  perturbed initial conditions of fixed norm.

\begin{defn}[Amplification envelope]
	\label{ampl_envelope}
	The \textit{amplification envelope} for an asymptotically stable linear system $\dot y= A(t)y$,  with $A:\R\to\R^{N\times N}$ locally integrable, is defined as the continuous function 
	\begin{equation}\label{eq:30/10-09:50}
	\begin{aligned}
	\rho:\R^+\times \R\to \R^+,\quad (t, t_0)\mapsto\rho(t,t_0)=\sup_{ |y_0|\neq 0} \frac{|y(t+t_0,t_0,y_0)|}{|y_0|},
	\end{aligned}
	\end{equation}
	where $y(\cdot,t_0,y_0)$ is the unique solution of $\dot y= A(t)y$ satisfying $y(t_0,t_0,y_0)=y_0$. Therefore,  $\rho(t,t_0)$ is in fact equal to the operator norm (induced by the Euclidean norm) of the principal matrix solution $Y(t+t_0,t_0)$ of $\dot y =A(t)y$ for the initial time $t_0$, that is 
	\[
	\rho(t,t_0)=\big\|Y(t+t_0,t_0)\big\|_{op}.
	\]
	For a nonlinear system $\dot x =f(x)$ such that $f$ satisfies Assumption~\ref{system_l_ind}, the \textit{amplification envelope} in a neighbourhood of a hyperbolic solution $\widetilde x$ is defined as the amplification envelope of the linear system $\dot y = {\textnormal{D}}f(\widetilde x(t))y$.
\end{defn}

Note in particular that if $A(t)=A\in\R^{N\times N}$ for all $t\in\R$, the amplification envelope depends only on the variable $t\in \R^+$, i.e.
\[
	\rho(t,t_0)=\rho(t)=\big\|e^{At}\big\|_{op}.
\]
In such a case, a constant $\mathcal{K}(A)$, known as the Kreiss constant of $A$ and defined as  $\mathcal{K}(A)=\sup_{ Re(z)>0} Re(z)\|(zI-A)^{-1}\|_{op}$, can be used to identify upper and lower estimates for the maximum amplification over time, using Kreiss Matrix Theorem~\cite{trefethen1997pseudospectra}. In particular, 
\[
\mathcal{K}(A)\le \max_{t>0}\rho(t)\le eN\mathcal{K}(A).
\]
The exact computation of the Kreiss constant is itself the object of study~\cite{apkarian2020optimizing}. A method to address the time-dependent case on a finite time interval requires the use of an augmented Lagrangian whose first variations are set to zero to obtain a set of algebraic-differential equations that can be solved forward and backward in time iteratively~\cite{schmid2007nonmodal}.

From the amplification envelope we can derive the two following indicators of resilience. The maximal amplification corresponds to the maximal magnification over time for trajectories starting in the nearby of an attractor, whereas the maximal amplification time records the
occurrence of the maximal amplification.
\begin{defn}[Maximal amplification and maximal amplification time]
\label{max_ampl_def}
	Consider the setting of Definition~\ref{ampl_envelope}. The \textit{maximal amplification} 	and \textit{maximal amplification time} are respectively defined as:
	\[\rho_{max}(t_0)=\max_{t \geq 0}\rho(t,t_0),\qquad t_{max}(t_0)=\argmax_{t \geq 0}\rho(t,t_0).\]
\end{defn}

\begin{rmk}
	The characteristic return time (see Definition~\ref{local_char_time}) and the Reactivity (see Definition~\ref{reactivity_def}) of a linear system 
	are also calculable from the amplification envelope. Indeed, for every fixed $t_0\in\R$, the first corresponds to the slope of $ln(\rho(\widetilde x,t,t_0))$ as $t \to +\infty$, while the second is equal to the slope of $ln(\rho(\widetilde x,t,t_0))$ as $t \to t_0$.  
\end{rmk}

\subsubsection*{Invariance with respect to change of coordinates} 
The construction of the amplification envelope implies that it is invariant with respect to diffeomorphisms of the phase space which preserve the Euclidean distance. This includes in particular linear changes of coordinates  $z=Qy$, with $Q$ nonsingular and orthogonal.

\subsection{Intrinsic stochastic and deterministic invariability}\label{subsec_intrinsic_variability}
The amplification envelope provides an effective tool to describe the local transient behaviour of a system close to an attractor if it is affected by an ``isolated and impulsive" perturbation. Arnoldi et al.~\cite{arnoldi2016} investigate the same transient behaviour when the system is subjected to a time-dependent (random or deterministic) forcing.  In the spirit of this paper, we generalize the presentation of~\cite{arnoldi2016} to hyperbolic solutions (see Definition~\ref{def:hyperb_sol}) up to the point where this makes sense. \par\smallskip

Consider a continuous and bounded function $A:\R\to\R^{N\times N}$ such that the linear homogeneous system 
\begin{equation}\label{eq:linear-nonautonom}
\dot y= A(t)y
\end{equation}
has an exponential dichotomy on $\R$ with projector the identity (see Definition~\ref{def:exp_dichotomy}).  This means that,  denoted by $U(t,s)$ the principal matrix solution of~\eqref{eq:linear-nonautonom} at time $s\in\R$, one has that there are constants $\alpha>0$ and $K\ge1$  such that 
\begin{equation}\label{eq:exp_dich_arnoldi}
|U(t,s)|\le Ke^{-\alpha(t-s)},\quad\text{for all }t\ge s.
\end{equation}
Firstly, consider the stochastic differential problem with additive white noise 
\begin{equation}\label{eq:Arnoldi_stochastic}
dy=A(t)y\,dt+S\,dW(t)
\end{equation}
where $S\in\R^{N\times M}$, and $ W= (W_1,\dots W_M)^\top $ is a vector of $M$ independent Brownian motions. As $t\to\infty$ the distribution of each strong solution of~\eqref{eq:Arnoldi_stochastic} converges to a stationary Gaussian distribution centered at the origin. Moreover, the covariance matrix of the system at time $t\in\R$ and initial data at $t_0\in\R$ is given by 
\begin{equation}\label{eq:covariance_matrix}
 C(t,t_0)=\int_{t_0}^{t}U(t,s)\Sigma U(t,s)^\top \,ds,
\end{equation}
where $\Sigma:=SS^\top $. 
Notice also that the differentiation of $C(t,t_0)$ with respect to $t$ shows that it solves the matrix differential equation
\begin{equation}\label{eq:covariance_differentiation}
\frac{d}{dt}C= A(t)C+CA(t)^\top+\Sigma.
\end{equation}
In particular, $C(t,-\infty)$ is the only bounded solution of~\eqref{eq:covariance_differentiation} over the whole real line, and~\eqref{eq:exp_dich_arnoldi} and~\eqref{eq:covariance_matrix} together imply that ~\eqref{eq:covariance_differentiation} has a local attractor (see for example~\cite[Theorem 1.23]{kloeden2011}) that must thus coincide with $C_*(\Sigma)=\{(t,C(t,-\infty)\mid t\in\R\}$. $C_*(\Sigma)$ shall be called the  \emph{stationary covariance} of the system. \par\smallskip

In order to construct an indicator of intrinsic resilience against stochastic perturbations, one looks at the largest stationary response among the possible perturbations of given ``magnitude".

\begin{defn}[Intrinsic stochastic invariability]\label{def_int_stoch_invar}
Consider a continuous dynamical system induced by an ordinary differential equation $\dot x =f(x)$, $x\in\R^N$, satisfying~\ref{system_l_ind} and the assumptions in {\rm Section~\ref{sec:prelimin}}. Moreover, assume that $\widetilde x$ is a locally attractive hyperbolic solution and consider the local attractor $\A=\{\widetilde x(t)\mid t\in\R\}$. Fixed a norm $\|\cdot\|$ on the matrix space $\R^{N\times N}$ (e.g.~the operatorial norm, the Froebenius norm), the \emph{stochastic variability} $\mathcal{V}_S(\A)$ and the \emph{intrinsic stochastic invariability} of the attractor $\A$ with respect to $\|\cdot\|$ are respectively defined as 
\[
\mathcal{V}_S(\A)=\sup_{\substack{\Sigma\ge0,\,\|\Sigma\|=1\\ t\in\R}}\|C(t,-\infty)\|_{op}\qquad\text{and}\qquad \mathcal{I}_S(\A)=\frac{1}{2\mathcal{V}_S(\A)},
\]
where  $C_*(\Sigma)$ is the stationary covariance of~\eqref{eq:Arnoldi_stochastic} with $A(t)={\textnormal{D}}f\big(\widetilde x(t)\big)$.
\end{defn}

\begin{rmk}\label{rmk:stoch_invariab_equilibria}
    If $\A=x^*$, then $A={\textnormal{D}}f(x^*)$, and~\eqref{eq:covariance_differentiation} does not depend on time. In this case, one can show that the linear operator $\widehat A(C):=AC+CA^\top$ has $N^2$ eigenvalues of the form $\lambda_j+\lambda_i$, where $\lambda_i,\lambda_j$ are eigenvalues of $A$, and the stationary covariance matrix $C_*(\Sigma)$ is given by the unique solution of the Lyapunov equation $AC+CA^\top+\Sigma=0$~\cite[Lemma 5.1.2]{berglund2006noise}. Therefore, one has that 
    \[
    \mathcal{V}_S(\A)=\sup_{\Sigma\ge0,\,\|\Sigma\|=1}\|-\widehat A^{-1}(\Sigma)\|.
    \]
\end{rmk}
On the other hand, one can consider a deterministic forcing of~\eqref{eq:linear-nonautonom}, that is,
\begin{equation}\label{eq:arnoldi_deterministic}
\dot y=A(t)y+g(t),
\end{equation}
where $g:\R\to\R^N$ is a bounded function. Denoted by $Y(t)$ a fundamental matrix solution of~\eqref{eq:linear-nonautonom}, note that 
\[
x(t,g)=Y(t)\int_{-\infty}^tY^{-1}(s)g(s)\,ds,\quad t\in\R,
\]
is a solution  of~\eqref{eq:arnoldi_deterministic}, which can be regarded as the \emph{stationary system response}, and its \emph{mean square deviation} from the origin (which is the global attractor of the homogeneous problem) is 
\begin{equation}\label{eq:mean_square_deviation}
m(g):=\lim_{T\to\infty}\frac{1}{T}\int_0^T|x(s,g)|^2\,ds.
\end{equation}
Arnoldi et al.~\cite{arnoldi2016} suggest to use the largest mean square deviation of $x(t,g)$ from the origin over the possible perturbations $g$ of given norm, as an indicator of resilience of the attractor. 

\begin{defn}[Intrinsic deterministic invariability]\label{def_int_det_invar}
    Consider a continuous dynamical system induced by an ordinary differential equation $\dot x =f(x)$, $x\in\R^N$, satisfying~\ref{system_l_ind} and the assumptions in {\rm Section~\ref{sec:prelimin}}. Moreover, assume that $\widetilde x$ is a locally attractive hyperbolic solution and consider the local attractor $\A=\{\widetilde x(t)\mid t\in\R\}$. The \emph{deterministic variability} $\mathcal{V}_D(\A)$ and the \emph{intrinsic deterministic invariability} of the attractor $\A$ are respectively defined as 
\begin{equation}\label{eq:intrinsic_deterministic_invariability}
\mathcal{V}_D(\A)=\sup_{\|g\|_\infty=1}2\sqrt{m(g)},\qquad\text{and}\qquad \mathcal{I}_D(\A)=\frac{1}{\mathcal{V}_D(\A)},
\end{equation}
where  $m(g)$ is the mean square deviation of the stationary response of~\eqref{eq:arnoldi_deterministic} with $A(t)={\textnormal{D}}f\big(\widetilde x(t)\big)$.
\end{defn}
\begin{rmk}
    If $\A=x^*$, then $A={\textnormal{D}}f(x^*)$, and one can use standard frequency analysis to improve the information given by~\eqref{eq:intrinsic_deterministic_invariability}. This is particularly true when one limits the possible deterministic perturbation to the so-called ``wide-sense stationary signals". Indeed, since any deterministic signal can be developed into a sum of harmonic terms,
or Fourier modes, and due to the fact that in the linear approximation, the system response to
a general perturbation is equal to the sum of the system response to the single-frequency
modes, a convexity argument yields that the perturbation generating the
largest system response is a single-frequency mode~\cite{arnoldi2016}.
Therefore, when $g$ is a single-frequency periodic forcing one obtains that 
    \[
    \mathcal{V}_D(\A)=\sup_{\omega\in\R}\|\textnormal{i}\omega-A^{-1}\|,
    \]
    where $\omega$ is the forcing frequency of $g$, $\textnormal{i}$ is the imaginary unit, and $\|\cdot\|$ is the induced matrix operator norm from the norm on $\mathbb{R}^N$.
\end{rmk}
When $\A=x^*$, an important result in~\cite{arnoldi2016} establishes a chain of order for some of the indicators of local resilience presented in this section.
\begin{prop}\label{prop:linear_ind_chain}
If $\A=x^*$ is a hyperbolic stable equilibrium, the following chain of inequalities holds true:
\[
-R_0\le \mathcal{I}_S(\A)\le \mathcal{I}_D(\A)\le EV(\A).
\]
\end{prop}
 \subsubsection*{Invariance with respect to change of coordinates} From the definitions of covariance matrix in \eqref{eq:covariance_matrix}, and of largest mean-square deviation in \eqref{eq:mean_square_deviation}, one can easily show that  intrinsic stochastic invariability and the intrinsic deterministic invariability are invariant with respect to diffeomorphisms of the phase space which preserve the Euclidean distance. This includes in particular linear changes of coordinates  $z=Qy$, with $Q$ nonsingular and orthogonal.

\subsection{Discussion}\label{local_discussion_subsection}

\begin{itemize}[leftmargin=*]
    \item Local indicators are intrinsically robust against ``small" perturbations. The reason lies in the property of persistence of the exponential dichotomy which goes under the name of \emph{roughness} (see Coppel~\cite{coppel2006dichotomies}). Note that this is a further reason justifying the notion of exponential dichotomy. Also in the case in which our original system admits only hyperbolic equilibria a sufficiently small but general time-dependent forcing perturbs the equilibria into hyperbolic trajectories whose properties of stability persist but, in general, cannot be analyzed via the sign of the eigenvalues of the variational problem anymore. In other words, the Lyapunov spectrum of an hyperbolic equilibrium is perturbed into the dichotomy spectrum~\cite{sacker1978spectral}.  
    
    \item Lundstr\"om~\cite{lundstrom2018} provides calculable definitions to approximate the recovery rate (reciprocal of the characteristic return time) and the slowest return time while at the same time estimating the distance to threshold (see Definition~\ref{def_DT}) and the volume of the basin of attraction (see Definition~\ref{latitude_volume}).

    \item Reactivity and maximal amplification are important concepts in fluid mechanics where they are used in the context of creation and evolution of local instabilities in a flow. The collection of results in the area goes under the name of nonmodal stability theory. For a concise introduction to the central techniques of nonmodal stability theory we point the interested reader to Schmid~\cite{schmid2007nonmodal}.

    \item Ives and Carpenter~\cite{ives2007} propose the idea that the dynamics in the nearby of the boundary of a basin of  attraction can be important to infer a measure of nonlocal stability for the attractor. Recalling the unforced Duffing oscillator (see Figure~\ref{fig:Duffing}) one can intuitively see that if a system lingers close to the boundary after a first perturbation, it is possibly more susceptible to tip once a new instantaneous perturbation takes place (see also Definition~\ref{def:precariousness}). The linearization of the model along the boundary is therefore suggested as a valuable tool. In practical terms, it is important to point out that identifying the boundary may be as difficult, if not more difficult, compared to identifying the attractor itself.  The study of the boundary and its distance from the attractor are the subject of Section~\ref{secbasinshape}. It is worth noting that all the definitions and properties contained therein are purely geometrical and do not involve any linearization.

\end{itemize}

\section{Basin Shape Indicators}\label{secbasinshape}
This section contains the core indicators of the classic research in ecological resilience which dates back to the groundbreaking work by Holling~\cite{holling1973}. The underlying assumption is that the considered phenomenon, and its mathematical model, admits multiple stable states and thus the phase space of the induced dynamical system can be partitioned into basins of attraction. Consequently, one aims to estimate the minimal perturbation of initial conditions which can drive a system lying on an attractor,  outside of its basin of attraction. The study of the geometrical features of the basins becomes the central focus, aiming to relate such characteristics to inherent properties of  stability of the respective attractors. \par\smallskip

This section contains three subsections and five indicators. In subsection~\ref{latitudew_subsec}, we present \emph{latitude in width}, \emph{distance to threshold} and \emph{precariousness}. In subsection~\ref{latitude_vol_susbec}, we present \emph{latitude in volume} and \emph{basin stability}. In subsection~\ref{subsec:shape} we briefly discuss the relations between the indicators presented above.\par\smallskip

All the indicators in this section are defined for a continuous dynamical system on $\R^N$ with local flow \[
\phi:\mathcal{U}\subset\R\times\R^N\to\R^N,\quad(t,x_0)\mapsto \phi(t,x_0).
\]

\subsection{Latitude in width,  distance to threshold and precariousness} \label{latitudew_subsec}
The notion of ``width" of a basin of attraction as an indicator of resilience dates back directly to the seminal work by Holling~\cite{holling1973}. Loosely speaking, the width of a basin of attraction corresponds to the length of the ``minimal" segment crossing through the attractor and intersecting (at its extrema) the boundary of the basin of attraction, if it applies  (see illustrative sketch in Figure~\ref{latitude_DT_picture}).
Although a precise mathematical formulation seems hard to find in the literature, the geometrical representations used in many works (for example Peterson et al.~\cite{peterson1998} and Walker et al.~\cite{walker2004}) permit to precisely formalize the underlying idea.

\begin{defn}[Latitude in width]\label{LW_def}
	Consider the (possibly empty) set 
	\begin{equation}\label{def:S_latitude}
	S=  \big\{(y,z)\in\R^{2N}\,\big|\,y, z \in \partial \BA, \text{ and }\alpha y+ (1-\alpha)z \in  \A\text{ for some } \alpha \in (0,1)  \big\}.
	\end{equation}
	The \textit{latitude in width} of an attractor $\A\subset\R^N$ for a continuous flow $\phi:\R\times\R^N\to\R^N$ is defined as
	\begin{equation*}
	L_w(\A)=\begin{cases}
	\infty,&\text{if }S=\emptyset\\
	\inf_{(y, z) \in S}|y-z|,&\text{otherwise}.
	\end{cases}
	\end{equation*}
	\label{def_lw}
\end{defn}
If $L_w(\A)<\infty$, the inferior in the previous formula is in fact a minimum as shown in the next result.
\begin{prop}
	\label{lw_min_inf}
	Under the notation and assumptions of  {\rm Definition~\ref{def_lw}} and if $L_w < +\infty$, we have:
	\[L_w(\A) = \inf_{(y, z) \in S} |y-z|  = \min_{(y, z) \in S} |y-z|.\]
\end{prop}
\begin{proof}
	 By definition $\A$ is a compact subset of $\R^N$. Therefore, there is $\rho>L_w(\A)>0$ such that for all $a_1,a_2 \in \A$, $|a_1-a_2|<\rho$. Moreover, notice that $\partial \BA$ is nonempty since $L_w(\A)<\rho<\infty$. By the compactness of $\A$, there are $a_1,\dots, a_n\in\A$ such that 
	 \[
	 \A\subset \bigcup_{i=1}^n \overline{B}_{\rho}(a_i)=:P.
	 \]
	 In particular, notice that if $y,z\in S$ and  $|y-z|<\rho$, then $(y,z)\in P\times P$, which is a compact subset of $\R^{2N}$. Therefore, by the continuity of the norm in $\R^N$ and Weierstrass Theorem, we immediately obtain the result.
\end{proof}

\begin{figure}[h]
	\centering
	\begin{subfigure}{.33\textwidth}
		\centering
		\includegraphics[height=3.5cm]{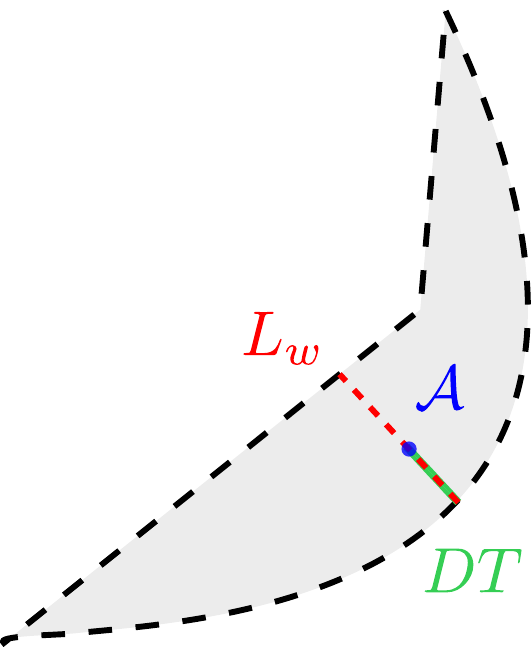}
		\caption{}
		\label{classic_lw}
	\end{subfigure}%
	\begin{subfigure}{.33\textwidth}
		\centering
		\includegraphics[height=3.5cm,trim={2cm 10cm 3.1cm 4cm},clip]{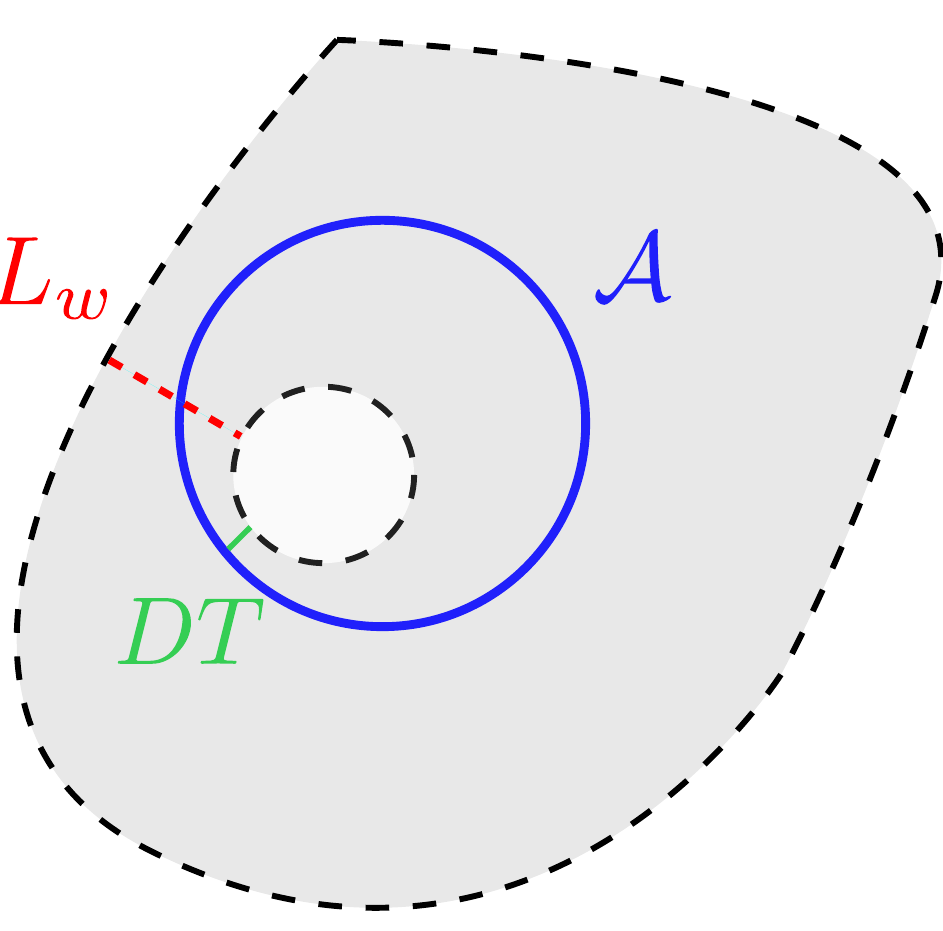}
		\caption{}
		\label{limit_cycle_lw}
	\end{subfigure}
		\begin{subfigure}{.33\textwidth}
		\centering
		\includegraphics[height=3.5cm,scale=5]{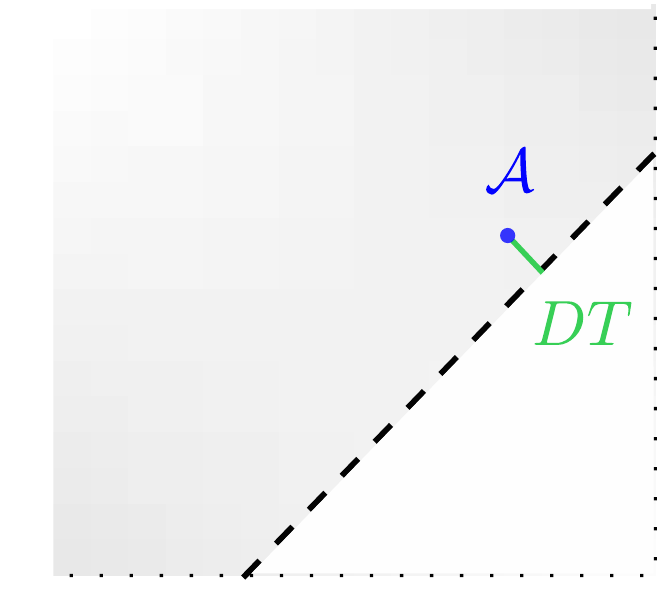}
		\caption{}
		\label{infty_lw}
	\end{subfigure}
	\caption{Three scenarios of a planar dynamical system are depicted. The grey regions represent the basins of attraction for the different attractors in blue. The corresponding boundaries are depicted as a black dashed line.
		(A) latitude in width $L_w$ for this equilibrium is given as the length of the red dashed line. Distance to threshold is given as the length of the green line. (B) system with a limit cycle attractor. $L_w$ is given as the length of the red dashed line and distance to threshold $DT$ as the length of the green line. Note that the lines have different locations. (C) example of an attractor with a basin of attraction that has infinite latitude in width indicator ${L_w= + \infty}.$ Note the distance to threshold given as the length of the green line is finite.}
	\label{latitude_DT_picture}
\end{figure}
Already in Peterson et al.~\cite{peterson1998}, while still employing the idea of width of the basin of attraction (under the name of ecological resilience), a focus was aimed at understanding the smallest perturbation of initial conditions able to drive the system away from the original attractor (or equivalently the largest perturbation of initial conditions for which the systems returns to the attractor). This approach, qualitatively, yet unequivocally presented by Beisner et al.~\cite{beisner2003}, emphasizes the importance of the points of minimal distance between an attractor and the boundary of its basin of attraction (in some cases called threshold or separatrix). The definition has the advantage of revealing that an attractor with infinite latitude in width (see Definition~\ref{def_lw}) is not necessarily ``very resilient" (see Figure~\ref{fig:inf_lw_nonresilient}).
This idea has been used consistently ever since for example in~Lundstr\"{o}m and Adainp\"{a}\"{a}~\cite{lundstrom2007}, Lundstr\"{o}m~\cite{lundstrom2018}, Kerswell et al.~\cite{kerswell2014}, Klinshov et al.~\cite{klinshov2015}, Mitra et al.~\cite{mitra2015}, and Fassoni and Yang~\cite{fassoni2017} where it is called \textit{precariousness} (on the subject see also Definition~\ref{def:precariousness} and Proposition~\ref{prop:DT_prec} below).

\begin{defn}[Distance to threshold]\label{def_DT}
	 Consider the (possibly empty) set 
	\begin{equation}\label{eq:01/12-12:01}
	S=\{(a,y)\in\R^{2N}\mid a \in \A, y \in \partial \BA \}.
	\end{equation}
	The \textit{distance to threshold} for an attractor $\A\subset\R^N$ for a continuous flow $\phi:\R\times\R^N\to\R^N$  is defined as 
	\begin{equation*}
	DT(\A) =\begin{cases}
	\infty,&\text{if }S=\emptyset\\
	\inf_{(a, y) \in S}|a-y|,&\text{otherwise}.
	\end{cases}
	\end{equation*}
\end{defn}

\begin{rmk}\label{rmk:DT_min}
    Reasoning as for the proof of Proposition~\ref{lw_min_inf} one can easily show that if $DT(\A)<\infty$, then it is attained at a point of minimum. Moreover, notice that $DT(\A)=+\infty$ only if $\partial \BA=\varnothing$ which means that $\A$ is the global attractor (see Definition~\ref{def:global_attractor}) for the system.
\end{rmk}

\begin{figure}
    \centering
    \begin{overpic}[trim={1.5cm 10cm 1.5cm 10cm},clip, width=\textwidth]{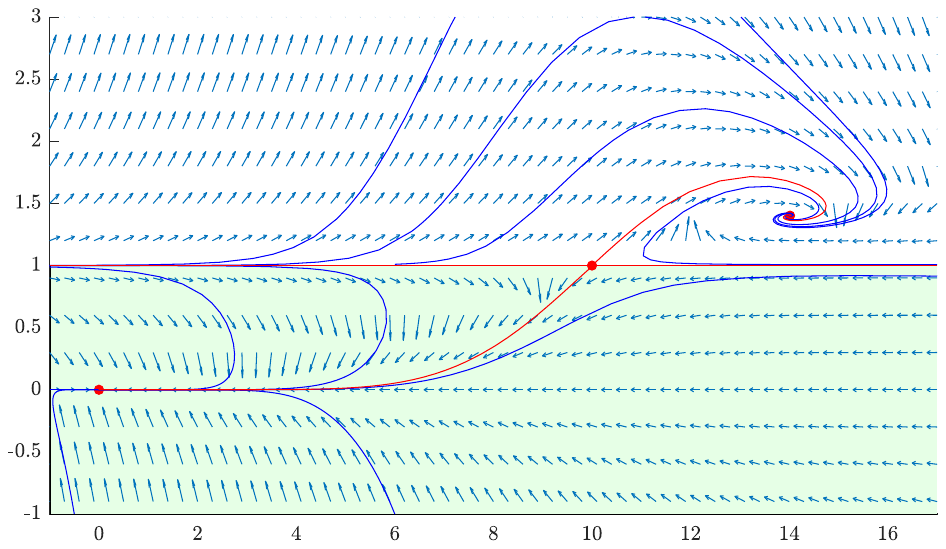}
    \put(50,0){$x_1$}
    \put(4,27){$x_2$}
    \put(57,29){$X_s$}
    \put(11,17){$X_0$}
    \put(74,31.5){$X_1$}
    \end{overpic}
    \caption{Representation of the dynamics induced by the planar system $\dot{x}_1=-x_1+10x_2$, 
$\dot{x}_2=x_2(10\exp(\frac{-x_1^2}{100})-x_2)(x_2-1)$. The system has three equilibria $X_0$, $X_s$ and $X_1$. The stable and unstable manifolds of the saddle node $X_s$ are depicted in red. The stable manifold of $X_s$ is the separatrix between the basins of attraction of $X_0$ and $X_1$, respectively painted in green and white. A qualitative behaviour of the system can be deduced via the vector field (blue arrows) and a few trajectories in solid blue.} Both stable equilibria have an infinite latitude in width, while their distance to threshold can be made as a small as wished through a suitable scaling. The resilience of this system has been thoroughly analyzed by Kerswell et al.~\cite{kerswell2014}. 
    \label{fig:inf_lw_nonresilient}
\end{figure}

When dealing with a concrete model of a real phenomenon, it is possible to tune the previous definitions according to the available information on the phenomenon. In particular, if one can anticipate the set of possible perturbed initial conditions, the calculation of the latitude in width and the distance to threshold can be restricted to a suitable subset $C$ of the phase space called \emph{region of interest}. For example, such an adjustment on the distance to threshold has been explicitly presented in~\cite{meyer2016} where the set $S$ in~\eqref{eq:01/12-12:01} is substituted by 
\begin{equation}\label{set_SC}
    S_C=\{(a,y)\in\R^{2N}\mid a \in \A, y \in \partial \BA\cap C\},
\end{equation}
for some $C\subset \R^N$. 
Likewise, the latitude in width can be modified with respect to the region of interest. Then, for some fixed $C\subset \R^N$, the set $S$~\eqref{def:S_latitude} is changed for
\[
    S_C=\{(y,z): (y,z) \in S \text{ and } y \in C\}.
\]
 
\begin{exmpl}
	\label{miniDT}
	\quad
		\begin{figure}[h!]
	\begin{center}
		\includegraphics[trim={0cm 0cm 0cm 0cm},clip,scale=0.9]{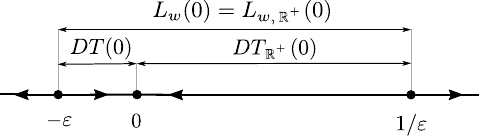}
	\end{center}
	\caption{Sketch of the phase space for the dynamical system induced by~\eqref{systemik} and of the distance to threshold and latitude in width of the attractor $\A=\{0\}$. Additionally, we can see that when we specify the anticipated perturbations as $\R^+,$ the distance to threshold changes, whereas, the latitude in width remains the same (for details, see Example~\ref{miniDT}).}
		\label{fig:dt_lw_1dim_roi}
	\end{figure}
	\begin{equation}\label{systemik}
	\begin{aligned}
	\dot{x}&= x(x-1/\varepsilon)(x+\varepsilon) 	\quad \quad \text{where }\quad  x \in \R, \varepsilon \in (0,1]
	\end{aligned}
	\end{equation}
	The dynamical system induced by~\eqref{systemik} has a stable equilibrium at $x=0$, which we consider as the relevant attractor, and two unstable fixed points at $-\varepsilon$ and $1/\varepsilon.$ See Figure~\ref{fig:dt_lw_1dim_roi} for a sketch of the phase space.  Notice that when $\varepsilon$ becomes smaller, the distance to threshold $DT(0)=\varepsilon$ decreases, but the latitude in width $L_w(0)=1/\varepsilon+\varepsilon$ increases. 
    \noindent If the set of anticipated perturbed initial conditions $C=\R^+$ is considered, the distance to threshold becomes $DT_{\R^+}(0)=1/\varepsilon,$ since we have the set $S_C=\{(0,1/\varepsilon)\}$ (see~\eqref{set_SC}), while the latitude in width $L_{w,\R^+}(0)$ remains the same.
	\noindent 
	On the other hand, if we fix $C=(0,\varepsilon),$ the set $S_C$ becomes empty for both distance to threshold and latitude in width and therefore  $DT_{(0,\varepsilon)}(0)=L_{w,(0,\varepsilon)}(0)=+\infty.$
\end{exmpl}

\begin{exmpl}\label{exmpl:dt1_dt2_attractor} Let us consider the example presented in Remark~\ref{rmk:different_attractors} (see also Figure~\ref{fig:DT1_2}). Depending on the different choice of the local attractor we may obtain different values of the distance to threshold $DT.$ If we consider only the limit cycle of radius $1$ as the attractor $\A_1,$ the unstable equilibrium at the origin is part of the basin's boundary $\partial \mathcal{B}(\mathcal{A}_1).$ Therefore, the distance to threshold $DT(\A_1)=1$. However, if we consider as a local attractor the set of points in the limit cycle and the origin, $\A_2=\A_1\cup \{(0,0)\},$ or the whole closed ball of radius one $\A_3=B_1$, we obtain $DT(\A_2)=DT(\A_3)=2.$ It is obvious that in terms of actual resilience of the system's state against a perturbation the origin plays a negligible role.
	\begin{figure}
		\centering
		\includegraphics[trim={0 0cm 0 0cm},clip,scale=0.25]{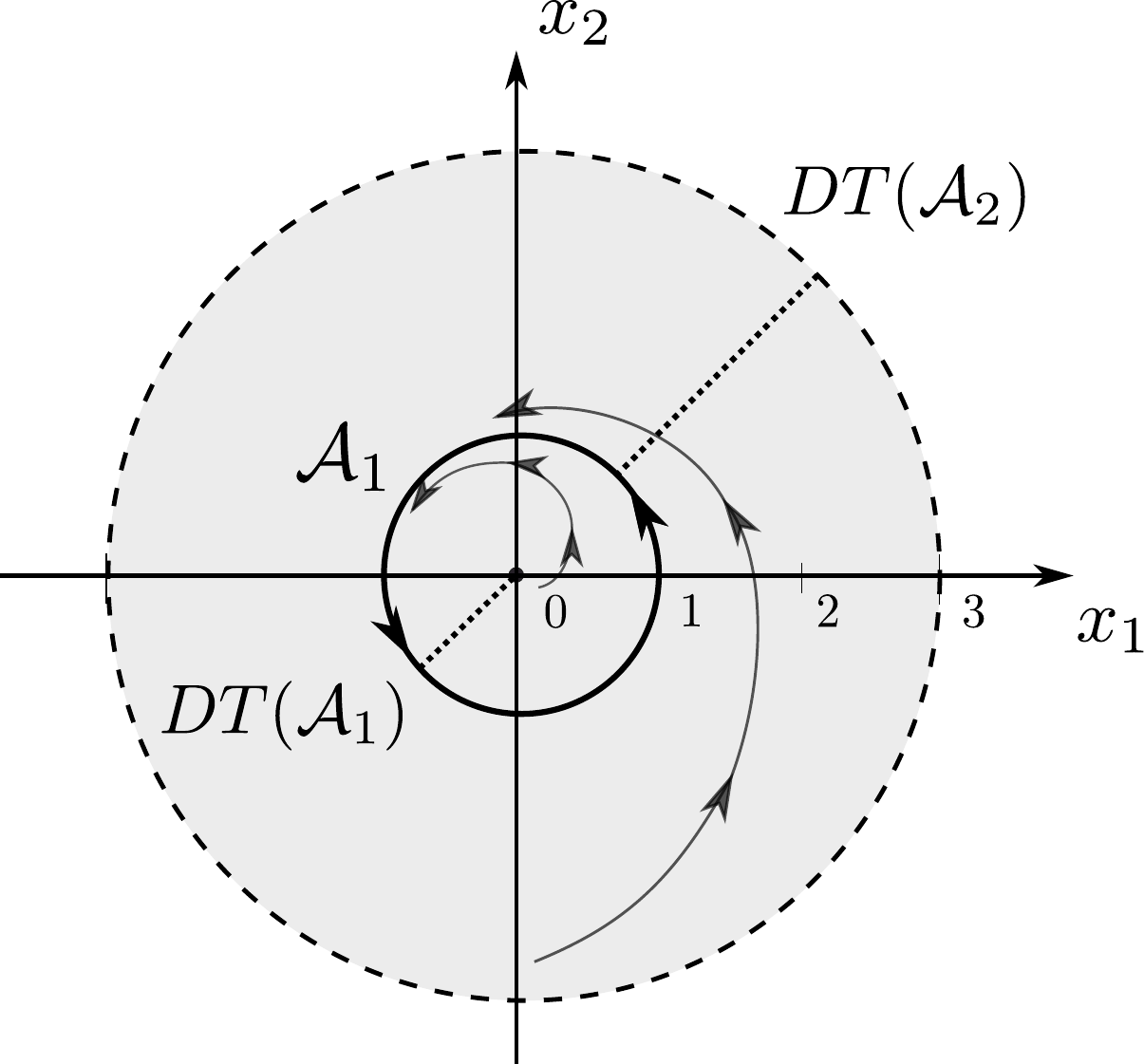}
		\caption{
		Sketch of the phase space for the dynamical system induced by~\eqref{eq:28/11-12:54}, and representation of the distance to threshold depending on the choice of the local attractor (see  Remark~\ref{rmk:different_attractors}).
		For the local attractor $\A_1$ consisting of the points in the limit cycle of radius $1$, $DT(\A_1)=1$. For the local attractor $\A_2=\A_1\cup \{(0,0)\}$, $DT(\A_2)=2.$  The choice of the attractor $\A_2$ admits to consider only distances to the ``significant" parts of the boundary.}\label{fig:DT1_2}
	\end{figure}
	
\end{exmpl}

Instead of focusing on an attractor and its basin, in practice it is convenient to look at the current state of the system within the basin. The idea of studying the minimal distance from the boundary of the basin has qualitatively presented by Walker et al.~\cite{walker2004}, where it is described as only one of four components of ecological resilience together with the latitude in width (see Definition~\ref{LW_def}), the resistance of a gradient system (see Definition~\ref{potential_def}) and the cross-scale interaction among the previous components and eventual subsystems. We push the idea in~\cite{walker2004} a bit forward and propose a mathematical definition which probes also the points outside the basin of attraction. This is not only a reasonable generalization, as it also allows to quantify the difficulty for a point outside the basin to enter it after a given perturbation, but  it also allows us to later establish a connection with a recently presented indicator (see Definition~\ref{def_flowkick_boundary}).
\begin{defn}[Precariousness]\label{def:precariousness}
The function
$P_\A:\R^N\to\R\cup\{\infty\}$ defined by
\[
x_0\mapsto P_\A(x_0)=\begin{cases}
\infty&\text{if }\partial\BA=\varnothing,\\
\displaystyle\inf_{y\in\partial\BA}|x_0-y|&\text{if }\partial\BA\neq\varnothing\text{ and } x_0\in\BA\\
\displaystyle-\inf_{y\in\partial\BA}|x_0-y|&\text{if }\partial\BA\neq\varnothing\text{ and } x_0\notin\BA\\
\end{cases} 
\]
is called the \emph{precariousness} of the point $x_0\in\R^N$ with respect to the basin of attraction $\BA$.
\end{defn}

The following simple proposition establishes a relation between the distance to threshold and the asymptotic value of precariousness.

\begin{prop}\label{prop:DT_prec}
Consider a continuous dynamical system $\phi$ on $\R^N$. If $\partial\BA$ is nonempty, then 
\[
\lim_{t\to\infty}P_\A\big(\phi(t,x_0)\big)\ge DT(\A),\quad\text{for all }x_0\in\BA,
\]
and the equality holds if every trajectory in $\A$ is dense.
\end{prop}
\begin{proof}
The first inequality is a direct consequence of the definition of basin of attraction (see Definition~\ref{attractor_def}) and of distance to threshold (see Definition~\ref{def_DT}) i.e., for any $x_0\in\BA$,
\[
\begin{split}
    \lim_{t\to\infty}P_\A\big(\phi(t,x_0)\big)&=\lim_{t\to\infty} \inf_{y\in\partial \BA}|\phi(t,x_0)-y|
    =\lim_{t\to\infty} \inf_{\substack{x\in\omega(x_0),\\y\in\partial \BA}}|\phi(t,x_0)\pm x-y|\\
    &\ge\inf_{\substack{x\in\omega(x_0),\\y\in\partial \BA}}|x-y|\ge \inf_{\substack{a\in\A,\\y\in\partial \BA}}|a-y|=
DT(\A).
\end{split}
\]

 If all trajectories in $\A$ are dense, the equality is a direct consequence of the fact that $\omega(x_0)\subset\A$ is invariant for the flow and thus it is dense in $\A$.
\end{proof}
 \subsubsection*{Invariance with respect to change of coordinates}The indicators latitude in width, distance to threshold and precariousness are invariant with respect to diffeomorphisms of the phase space which preserve the Euclidean distance. This includes in particular linear changes of coordinates  $z=Qy$, with $Q$ nonsingular and orthogonal.
\subsection{Latitude in volume and Basin stability}\label{latitude_vol_susbec}
For one-dimensional dynamical systems, the latitude in width of a basin of attraction (see Definition~\ref{LW_def}) coincides, in fact, with its Lebesgue measure $\mu(\BA)$ (recall that a basin of attraction is always an open set and therefore Lebesgue-measurable; see Proposition~\ref{basin_open_prop}). The notion acquires special interest when the study can be restricted to a measurable set $C\subset\R^N$ of finite measure since one can interpret the ``relative volume" of the basin of attraction $\BA$ (or the portion of $\BA$ which lies in $C$) in terms of the probability that a perturbed initial condition still belongs to $\BA$.

According to Grümm~\cite[Page 4]{gruemm1976},
Holling and collaborators are the first to use this approach.
Grümm~\cite[Page 7]{gruemm1976}, proposes to measure an attractor's resilience as the volume of its basin of attraction.  
Wiley et al.~\cite{wiley2006}, and Fassoni and Yang~\cite{fassoni2017} suggest to measure the volume in a relative sense to some bounded subset, as follows.
\begin{defn}[Latitude in volume] \label{latitude_volume}
	Let $C\subset\R^N$ be Lebesgue-measurable and such that $0<\mu(C)<\infty$.  
	We define the \textit{latitude in volume} of the attractor $\A$ with respect to the \emph{region of interest} $C$ as: 
	\[L_v(\A,C) = \frac{\mu(\BA\cap C)}{\mu(C)}.\]
	$L_v$ returns a nondimensional value in $[0,1]$. Particularly, $L_v(\A,C)=0$ if and only if all the trajectories starting in $C$ (except at most those starting in a negligible subset of $C$) do not converge to $\A$ as $t\to\infty$. On the other hand, $L_v(\A,C)=1$ if and only if all the trajectories starting in $C$ (except at most those starting in a negligible subset of $C$) converge to $\A$ as $t\to\infty$.
\end{defn}

\begin{rmk}[Region of interest]
	\label{region_of_interest}
	The ``region of interest" $C$ implicitly reveals the characteristics of the models to whom the indicator of latitude in volume should be applied. The system at hand should be at most subjected to isolated impulsive perturbations of bounded magnitude $\delta>0$, where the term isolated means that the interval of time between two occurrences should be longer than the transient required by the system to reach an ``indistinguishable proximity" to the attractor $\A$ from any point in the basin of attraction. Then, the subset $C$ is defined as the set of possible perturbed initial conditions from the attractor~\cite{menck2013}, i.e.~$B_\delta(\A)$. If additional information on the probability distribution of such perturbations is available, the indicator can be refined (see Definition~\ref{basin_stab_def}).
	On the other hand, this indicator assumes additional meaning in the case of multi-stable systems ~\cite{wiley2006, fassoni2017}. In this case the basin's volume can be interpreted as a measure of the ``attractor's relevance" in the phase space, where it competes against other attractors and their basins over the phase space.  Thus, the region of interest should include all the (relevant) attractors $\A_i$, $i=1,\dots,\nu$  (assuming that they are included in a set of bounded measure) as well as the possible perturbed initial conditions of magnitude $\delta$ from them, i.e. $\bigcup_{i=1}^\nu B_\delta(\A_i)\subset C$. Even so, there is no consensus in the literature as how to choose the region $C$ itself. For example, in~\cite{fassoni2017}  the smallest $N$-dimensional interval $I$ containing $\bigcup_{i=1}^\nu B_\delta(\A_i)$ is considered (see the conceptual sketch in Figure~\ref{far_volume2}). This approach privileges the immediate calculation of the measure of $I$.
\end{rmk}
\begin{figure}[h!]
	\centering
	\begin{subfigure}{0.49\textwidth}
		\includegraphics[height=4cm,trim={0cm -1cm 0cm -3cm},clip]{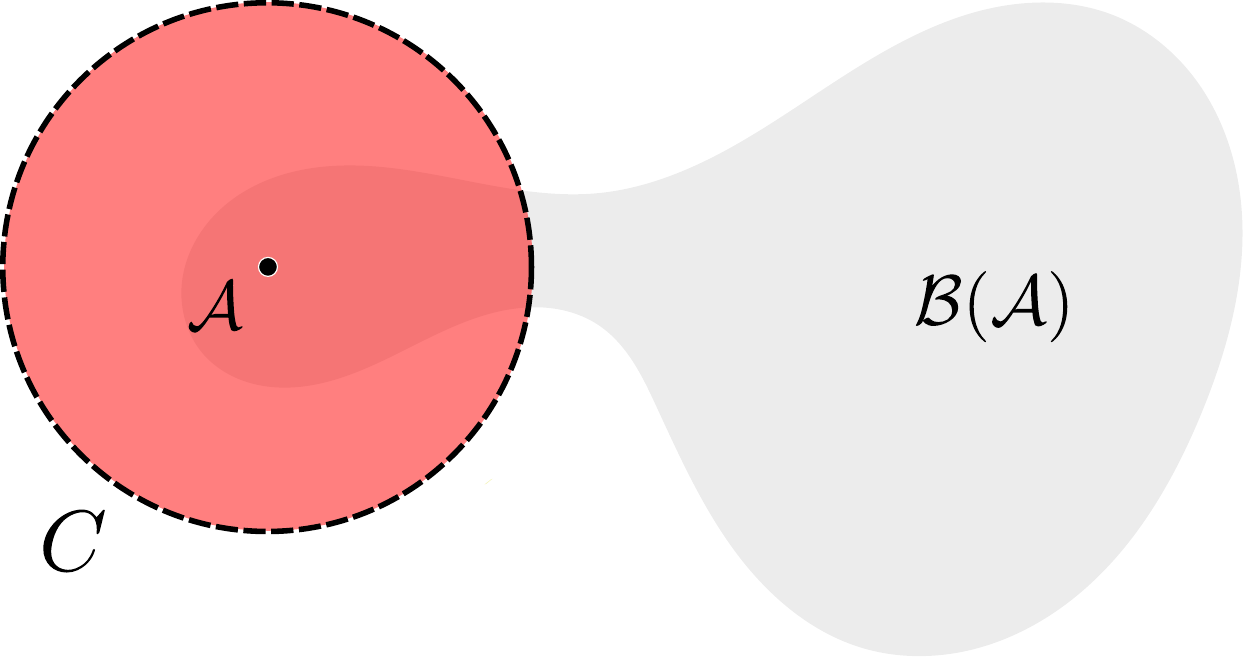}
		\caption{}
		\label{far_volume}
	\end{subfigure}
	\begin{subfigure}{.49\textwidth}
		\centering
		\includegraphics[trim={0cm 0cm 0cm 0cm},clip,height=4cm]{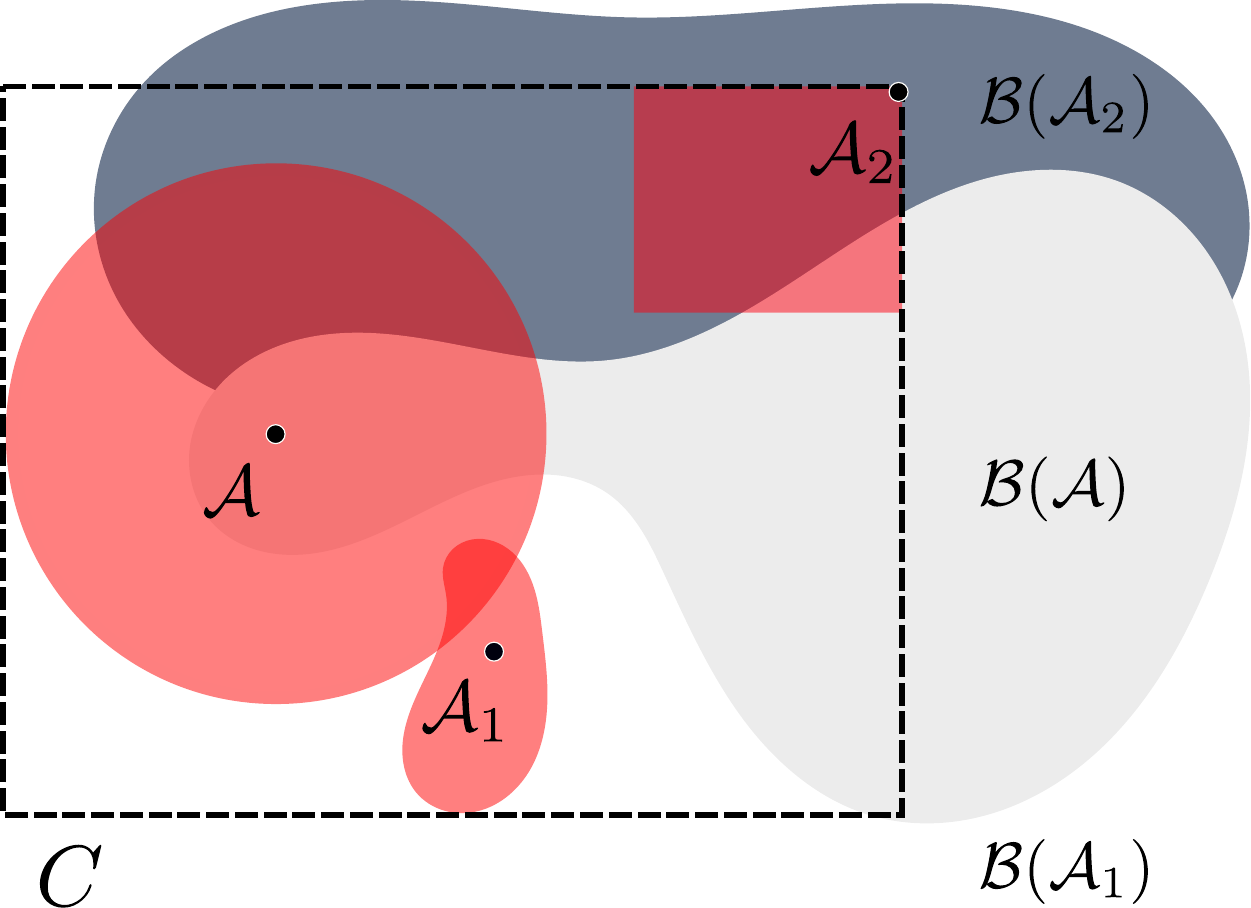}
		\caption{}
		\label{far_volume2}
	\end{subfigure}
	\caption{Depiction of two different approaches in choosing the region of interest $C$ depending on the available information. The relevant attractor is sketched as a black dot and identified by the symbol $\A$, whilst its basin of attraction corresponds to the gray area denoted by the symbols $\BA$. The region of interest $C$ is shown as the area encircled by a black dashed line and denoted by $C$.
		(A) The latitude in volume of the attractor $\A$ is calculated with respect to the region of interest $C$ corresponding to the set of expected perturbations of the initial conditions (red region). 		
		(B) The latitude in volume of the attractor $\A$ is calculated  with respect to the  $n-$dimensional interval $C$ containing also the further attractors $\A_1,\A_2$ as well as the set of each local attractor's expected perturbed initial conditions (red regions).}
	\label{volumes_fig}
\end{figure}

\label{basin_stability_subsec}
Menck et al.~\cite{menck2013,menck2014} propose a variation of the latitude in volume (see Definition~\ref{latitude_volume}) where a weight over the phase space is considered, according to an \textit{a priori} chosen probability density $\rho$ of the expected perturbations of initial conditions from the attractor. This indicator and its numerical calculation has been investigated in several works including for example Lundstr\"{o}m~\cite{lundstrom2018}, Lundstr\"{o}m and Adainp\"{a}\"{a}~\cite{lundstrom2007},  Evans and Swartz~\cite{evans2000}, and Mitra et al.~\cite{mitra2017}.

\begin{defn}[Basin stability]\label{basin_stab_def}
	Let $\rho: \R^N\to \R^+$ be a Lebesgue integrable function such that $\int_{\R^N} \rho(x)\, dx =1$. The \textit{basin stability} $\BS$ of an attractor $\A$ with respect to $\rho$ is defined as: 	 
	\begin{equation}\label{basin_stab_eq}
	S_\mathcal{B(A)}(\rho)= \int_{\R^N}  \chi_{\mathcal{B(A)}}(x)\rho(x) \, dx,
	\end{equation}
	where $\chi_Z(x)$ is the characteristic function of  $Z \subset \R^N$.  $S_\mathcal{B(A)}$ returns a nondimensional value in $[0,1]$.
\end{defn}

\begin{rmk}
Given $C\subset\R^N$ Lebesgue-measurable and such that $0<\mu(C)<\infty$ and considered 
    \[
	\rho_C(x) = 
	\begin{cases}
	1/\mu(C) & \text{for } x \in C, \\
	0 & \text{else,}\\
	\end{cases}
	\]
	one immediately has that $S_\mathcal{B(A)}(\rho_C)=L_v(\A,C)$.
\end{rmk}

\subsubsection*{Invariance with respect to change of coordinates} Both, basin stability  and latitude in volume  are not invariant with respect to the change of coordinates.
Nevertheless, measure-preserving diffeomorphisms of $\R^N$ allow to preserve the value of the indicators. However note, that then also the region of interest $C$ (or function $\rho,$ in the case of basin stability) have to be transformed accordingly.

\subsection{Discussion} \label{subsec:shape}

\begin{itemize}[leftmargin=*]
\item Basin shape indicators, while being a useful intuitive tool, should be treated very carefully. Indeed, often a real-world system is inherently ``open", meaning that it is constantly subjected to a time-dependent forcing (deterministic and/or random). Hence, the boundary of the basin of attraction of the unperturbed problem may not have a dynamical meaning and the indicator can become misleading. A concrete example of this fact can be recognized in the phenomenon of tipping induced by nonautonomous terms as treated in Section~\ref{sec:parameters}. Furthermore, even if this does not apply and the system is subject to only instantaneous isolated perturbations, the estimation of a basin's boundary is usually a difficult task. For example Kerswell et al.~\cite{kerswell2014} use a variational method employing the maximal amplification (see Definition~\ref{max_ampl_def}) to estimate the distance to threshold.
On the other hand, basins of attraction with fractal-like boundaries (or intermingled and riddled basins, which do not arise for trapped attractors but are possible for the Milnor's definition) are discussed in the work of Schultz et al.~\cite{schultz2017}.

\item The following result establishes a direct relation between latitude in width, latitude in volume and distance to threshold. It is intuitively clear that the latter delivers the safest information (see Figure~\ref{fig:flower_basin}). Nevertheless, we include a brief proof as it seems to not be anywhere in the literature. 
\begin{figure}
		\includegraphics[scale=0.42]{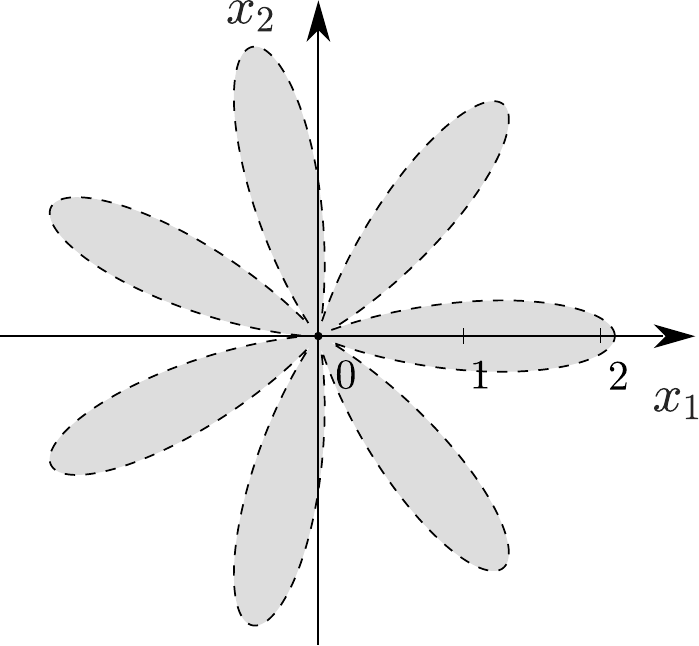}
		\caption{Depiction of the phase space of the system induced by $\dot{r}=r(r-cos(7\varphi )-(1+\varepsilon))$, $\dot{\varphi}=0$, where $(r,\varphi) \in [0,\infty) \times [0,2\pi)$ and $ \varepsilon>0$. The problem has an asymptotically stable equilibrium at the origin. The basin of attraction is in a shape resembling a ``flower", and has a relatively big volume. However, the distance to threshold is only $DT=\ep.$}
		\label{fig:flower_basin}
	\end{figure}

\begin{prop}\label{prop:DT_width_volume}
	 Consider a continuous dynamical system $\phi$ on $\R^N$. Then, for a given attractor $\A$ and for any  $ C\subset\R^N$ measurable with $0<\mu(C)<\infty$ we have that
	\[
	2 DT (\A)\leq L_w(\A),\quad\text{and}\quad  L_v(\A, C)\le L_v\big(\A, B_{DT}(\A)\big)=1.
	\]  
\end{prop}
\begin{proof}
	We let $S_{w}$ and $S_{DT}$ denote the sets introduced in Definition~\ref{def_lw} and~\ref{def_DT}, respectively. Firstly, notice that if $DT(\A)= +\infty$, then $S_{DT}= \emptyset$. Therefore,  the basin's boundary $\partial B(A) = \emptyset$ and thus also $S_{w}= \emptyset$ and $L_w=+\infty$. On the other hand, if $DT(\A)<\infty$ and $L_w(\A)=+\infty$ the aimed inequality is trivially true.	Assume now that $DT(A)<+\infty$ and $L_w(A)<+\infty$. Due to Proposition~\ref{lw_min_inf} and Remark~\ref{rmk:DT_min} there is $(y,z)\in S_{w}$ such that $L_w(\A)=|y-z|$ and $(u,v)\in S_{DT}$ such that $DT(\A)=|u-v|$. Let $\alpha^*\in(0,1)$ be such that $x= \alpha^* y+ (1-\alpha^*)z \in  \A $. Then, we have that $|u-v|\le \min\{|y-x|,|x-z|\}$. Therefore, 
	\[2DT(\A)=2|u-v|\le |y-x|+|x-z|=|y-z|=L_w(\A),\]
	which is the first aimed inequality. In regards to the second inequality, an argument by contradiction allows to easily show that  $B_{DT}(\A) \subset \BA$. Thus, we have
	\[
	L_v(\A,B_{DT}(\A)) = \frac{\mu(\BA\cap B_{DT}(\A))}{\mu(B_{DT}(\A))}= \frac{\mu(B_{DT}(\A))}{\mu(B_{DT}(\A))}=1.
	\]
	On the other hand, \[1=\max\{	L_v(\A,C)\mid C\subset\R^N,\text{ measurable with } 0<\mu(C)<\infty\}\]
	which immediately gives the second inequality.
\end{proof}

\end{itemize}

\section{nonlinear transient dynamics}\label{sec:transient}
The indicators of resilience introduced in the previous sections focus either on the properties of the linearized flow, or on the geometrical features of the basin of attraction. However, they do not take into account the nonlinear transient dynamics within the basin, which may play a crucial role. In order to better exemplifies this idea let us briefly present an example from Meyer and McGehee~\cite{meyer2020intensity}. Consider the differential equations
\[
\dot x=f(x)=\begin{cases}
x&\text{if }x<0,\\
\pi^{-1}\sin(\pi x)&\text{if }0\le x\le 1,\\
1-x&\text{if }x\ge 1,\end{cases}\quad\text{and}\quad\dot x=g(x)=-x(x-1).
\]
The dynamical systems induced by both equations have an unstable equilibrium at $x=0$ and a stable equilibrium at $x=1$ . Therefore the distance to threshold for $\A=1$ is the same for both systems, namely $DT(\A)=1$. Moreover, the characteristic return times are also the same $T_R(\A)=1$. Nonetheless, the two functions $f$ and $g$ attain different maxima, respectively $\mu=\pi^{-1}$ and $\widehat \mu=1/4$. For a given uncertainty $\widehat \mu<\ep<\mu$ on both vector fields, it can happen that the induced dynamical systems cease to be topologically equivalent (see Figure~\ref{Meyer_example}). In particular the dynamical system induced by $\dot x=g(x)-\ep$ does not have any bounded solution anymore.\par\smallskip

\begin{figure}[h]
	\centering
	\begin{overpic}[width=\textwidth,trim={0cm 14cm 0 0cm},clip]{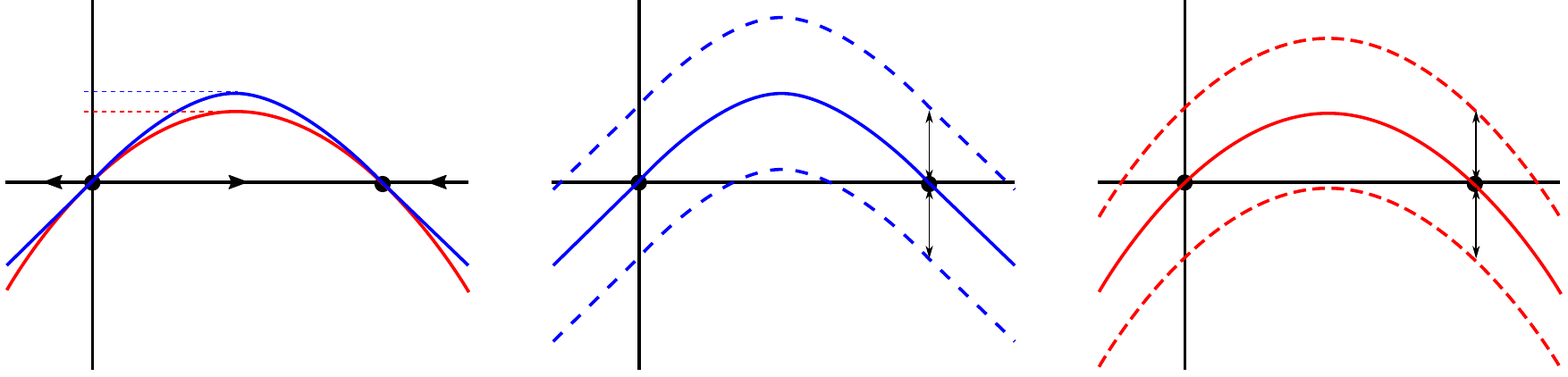}
	\put(0,22){(a)}
	\put(35,22){(b)}
	\put(70,22){(c)}
	\put(3,18){{\color{blue}$\mu$}}
	\put(3,15){{\color{red}$\widehat\mu$}}
	\put(12,19){{\color{blue}$f(x)$}}
	\put(47,19){{\color{blue}$f(x)$}}
	\put(12,13.5){{\color{red}$g(x)$}}
	\put(82,18){{\color{red}$g(x)$}}
	\put(27,13.5){$x$}
	\put(23.5,8.5){$1$}
	\put(59.7,8.5){$\ep$}
	\put(59.7,13){$\ep$}
	\put(94.5,8.5){$\ep$}
	\put(94.5,13){$\ep$}
	\end{overpic}
	\caption{(a) Two continuous vector fields $f$ and $g$ induce ordinary differential equations sharing the same distance to threshold and characteristic return time for the attractor $\A=1$. However, a common uncertainty $\ep>0$ on both problems may lead to dynamical systems which are not topologically equivalent, since  the first maintains always a nonempty set of bounded solutions (b), whereas it may occur that the latter has no bounded solutions (c). The example is thoroughly described by  Meyer and McGehee~\cite{meyer2020intensity}. }
	\label{Meyer_example}
\end{figure}

This section contains five indicators. \emph{Return time} in subsection~\ref{global_ret_time_subsec}, \emph{resistance of a gradient system} in subsection~\ref{potential_subsec}, \emph{resilience boundary} in subsection~\ref{resil_boundary_subsec}, \emph{intensity of attraction} in subsection~\ref{intensity_subsec} and \emph{expected escape times} in subsection~\ref{sec:escape_times}.\par\smallskip

Except for the return time, which is defined for a general continuous flow on $\R^N$, all the other indicators are defined for dynamical systems induced by ordinary differential equations of the type $\dot x =f(x)$, $x\in\R^N$, for which we assume that the assumptions in Section~\ref{sec:prelimin} are satisfied.

\subsection{Return Time}\label{global_ret_time_subsec}
Between the '70s and the '90s of the last century, the notion of return time  as an indicator of resilience (firstly introduced by O'Neill~\cite{oneill1976}) gathered attention and stimulated a fair amount of research (e.g.~Cottingham and Carpenter~\cite{cottingham1994}, DeAngelis et al.~\cite{deangelis1989,deangelis1989b}, Harte~\cite{harte1979ecosystem}). The reason lies in the fact that this definition was able to capture the transient behaviour of the system after an initial perturbation in contrast with the more classical characteristic return time (see Definition~\ref{local_char_time}). Although there are several variations in the literature (mostly differing only for the constant of normalization), they all are similar to the definition below~\cite{neubert1997,oneill1976}.

\begin{defn}[Return time]
	\label{return_time_general}
	Consider a continuous dynamical system $\phi$ on $\R^N$, and assume that  $\A$ is an attractor for $\phi$ and $\BA$ its basin of attraction. The \textit{return time} for the system from $x_p \in \BA\setminus \A$ to $\A$ is defined as:
	\begin{equation}\label{global_ret_time}
	\overline T_R(\A, x_p) = \frac{1}{\dist(x_p, \A)}\int_{0}^{\infty} \dist(\phi(t,x_p),\A)\, dt,
	\end{equation}
	where $\dist(A,B)$ denotes the Hausdorff semi-distance between the sets $A,B\subset\R^N$.  We shall use the symbol $\overline T_R$ instead of the more classical $T_R$ in order to avoid confusion with the characteristic return time presented in Subsection~\ref{local_char_time_susbec}. Moreover, for a fixed measurable set $C\subseteq \BA\setminus \A$, the maximal return time and average return time from $C$ to $\A$ are respectively defined as 
	\[
	T_R^{sup}(\A,C)=\sup_{x_p \in C} \overline T_R(\A, x_p),\quad\text{and}\quad T_R^{mean}(\A,C)=\frac{1}{\mu(C)}\int_{C} \overline T_R(\A,x)\ d\mu(x).
	\]
\end{defn}

The definition of return time is a valuable theoretical tool, but the indefinite integral in~\eqref{global_ret_time} is a clear obstacle in practical applications. Therefore, more calculable indicators, still accounting for the transient dynamics (see Neubert and Caswell~\cite{neubert1997}), have been developed and were frequently preferred in applications.\par\smallskip 

The following result shows (as expected) that the integral in~\eqref{global_ret_time} is finite when the attractor consists of a hyperbolic solution (see Definition~\ref{def:hyperb_sol}).

\begin{prop}\label{prop:finite-return-time}
	Consider a continuous dynamical system on $\R^N$ induced by an ordinary differential equation $\dot x =f(x)$, with $f$  continuously differentiable. Furthermore, assume that the attractor $\A$ is made of the points of a locally attractive hyperbolic solution $\widetilde x$ (see {\rm Definition~\ref{def:hyperb_sol}}). Then,
	\[
	\overline T_R(\A, x_p)<+\infty,\qquad\text{for all }x_p \in \BA.
	\] 	
\end{prop}
\begin{proof}
    Due to the First Approximation Theorem (see~\cite[Theorem~III.2.4]{hale1980ordinary}), there is $\delta>0$ such that if $\overline x\in\BA$ satisfies $|\overline x-\widetilde x(s)|\le \delta$ for some $s\in\R$, then for all $t\ge s$, $|\phi(t,\overline x)-\widetilde x(t)|<Ke^{-\beta(t-s)}$ for some $K>1$ and $\beta>0$. On the other hand, by definition, for every $x_p\in\BA\setminus \A$ there is $s>0$ such that $\dist(\phi(s,x_p),\A)<\delta$. Then the result is a consequence of the linearity of the integral in~\eqref{return_time_general}.
\end{proof}

\subsubsection*{Invariance with respect to change of coordinates} The maximal and average return times are invariant with respect to diffeomorphisms of $\R^N$ which preserve the Euclidean distance. This includes in particular linear changes of coordinates  $z=Qy$, with $Q$ nonsingular and orthogonal. Note that the region of interest needs to be reparametrized in a consistent way.

\subsection{Resistance of a gradient system}\label{potential_subsec} 
Stability landscapes are a commonly used tool in mathematical ecology as they provide immediate intuitive information on the dynamical features of equilibria (when they exist) of low-dimensional systems. Limitations and warnings of cautious employment of this instrument have been pointed out by many (see~\cite{meyer2016} for example). In mathematical terms a stability landscape can be seen as the potential $V$ of a gradient system (and related to the notion of Lyapunov function~\cite{conley1988gradient}). This means that the phenomenon can be represented via an ordinary differential equation of the type $\dot x= - \nabla V(x)$, where $V$ is a twice continuously differentiable real valued function.
Then, one can consider the  minimal work required to bring the system from an attractor $\A$ to a perturbed initial condition outside the basin of attraction $\BA$ as an indicator of resilience (see for example Lundstr\"{o}m~\cite{lundstrom2018}).

\begin{defn}[Resistance  of a gradient system]\label{potential_def}
Consider a continuous dynamical system induced by an ordinary differential equation  $\dot x= - \nabla V(x)$, where ${V:\R^N\to\R}$ is a twice continuously differentiable. The \emph{resistance (potential) of a gradient system} with respect to the attractor $\A$ is 
\[
W(\A)=\inf_{x_0\notin\BA,\,a\in\A}\big(V(x_0)-V(a)\big).
\]
\end{defn} 

\begin{rmk}
	Walker et al.~\cite{walker2004} call \emph{resistance of an attractor} $\A$ the quantity defined by 
	\[
	R_W(\A)=\frac{W(\A)}{L_w(\A)},
	\]
	where $L_w(\A)$ is the latitude in width of $\A$. As pointed out also by Meyer~\cite{meyer2016}, care should be taken in the interpretation of this indicator. In particular, if the timescales of disturbance and internal dynamics are considerably different, the slope of the potential function at the state of the system does not necessarily relate to the capability of the system to recover.
\end{rmk}
 \subsubsection*{Invariance with respect to change of coordinates} 
If the potential function $V$ is conserved, a diffeomorphism $\sigma:R^N\to\mathbb{R}^N$ $y\mapsto \sigma(y)$ of the phase-space does not change the indicator since $|V(x_0)-V(x)|= |V(\sigma(y_0)-V(\sigma(y)|$ if $x_0=\sigma(y_0)$ and $x=\sigma(y)$.

\subsection{Resilience Boundary}\label{resil_boundary_subsec}
Meyer et al.~\cite{meyer2018} suggest to measure the resilience of an attractor by introducing and applying the so-called ``flow-kick framework", a sequence of repeated ``impulsive" disturbances.  Given a continuous dynamical system induced by an ordinary differential equation $\dot x =f(x)$, $x\in\R^N$, assume that solutions are defined for all $t\in\R^+$. For fixed $\tau>0$, $\kappa\in\R^N$, and $a\in\A$ a \emph{flow-kick trajectory} $\widetilde x_{\tau,\kappa}(\cdot, a)$ starting in $a\in\R^N$  is constructed by considering a partition of the positive real line in intervals $[j\tau,(j+1)\tau)$, $j\in\N$, and a sequence of functions $\widetilde x_j:[j\tau,(j+1)\tau]\to\R^N$ such that $\widetilde x_j$ solves the Cauchy problem 
\[
\dot x=f(x),\qquad \widetilde x_j(j\tau)=\widetilde x_{j-1}(j\tau)+\kappa,\quad\text{ and }\quad\widetilde x_0(0)=a\in\A.
\]
One may interpret $\kappa$ as a \emph{kick}, while $\tau$ represents the \emph{flow time}, and a pair $(\tau,\kappa)\in(0,\infty)\times\R^N$ is called a \emph{disturbance pattern}. Notice that except for the trivial case where $\kappa$ is the origin in $\R^N$, a flow-kick trajectory is not a solution of the original problem.

\begin{defn}[Flow-kick resilience set and resilience boundary]\label{def_flowkick_boundary}
Under the previously stated assumptions and notation, a disturbance pattern $(\tau,\kappa)\in(0,\infty)\times\R^N$ is said to be in the \emph{flow-kick resilience set} if 
\[
\inf_{t>0}P_\A\big(\widetilde x_{\tau,\kappa}(t, a)\big)>0\quad\text{for all }a\in\A,
\]
where $P_\A(y)$ is the precariousness of $y\in\R^N$ (see {\rm Definition~\ref{def:precariousness}}). Moreover, the boundary of the flow-kick resilience set in $(0,\infty)\times\R^N$ is called the \emph{(flow-kick) resilience boundary} for the basin of attraction $\BA$.
\end{defn}
If $N>1$ the resilience boundary should be considered as a precautionary threshold. Indeed, it allows to identify the flow-kick trajectories which leave $\BA$ at some point, but does not allow to distinguish them from those who eventually return to $\BA$~\cite{meyer2018}.
On the other hand, if $N=1$ and $\widetilde x_{\tau,\kappa}(\overline t, a)\notin\BA$ for some $a\in\A$ and disturbance pattern $(\tau,\kappa)$, the same is true for all $t\ge\overline t$. 
In fact, under appropriate assumptions, the scalar case allows to associate a quantitative indicator based on the area of the region in between the flow-kick resilience set and the line $|\kappa|=DT(\A)$; see also Figure~\ref{fig:resilience_boundary}, where $DT(\A)$ denotes the distance to threshold for $\A$. This indicator has the merit of establishing a relation with distance to threshold (Definition~\ref{def_DT}) and distance to bifurcation (Definition~\ref{dist_to_bif_def}).

\begin{figure}[h]
\centering
\includegraphics[width=0.8\textwidth,trim={1cm 1cm 0cm 0cm},clip]{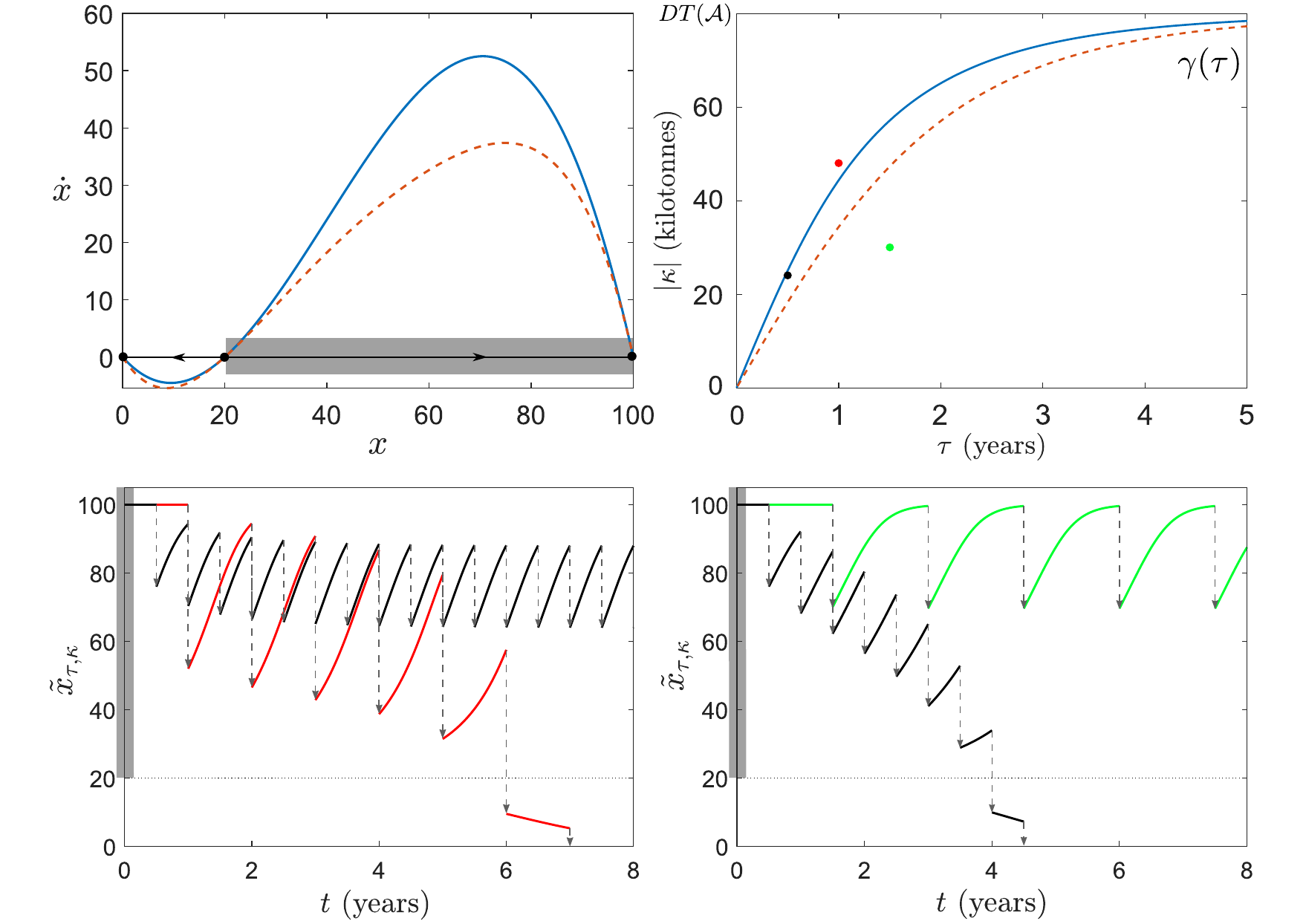}
	\caption{Example of flow-kick resilience and resilience boundary for scalar models in population dynamics. The upper left-hand side panel portrays  two populations. Population 1 modelled by  $\dot{x}=x\left(1-\frac{x}{100kt} \right)\left(\frac{x}{20kt}-1 \right),$ (solid blue line) where $kt$ stands for kilo-tonnes, and population 2 by $\dot{x}=x\left(1-\frac{x}{100kt} \right)\left(\frac{x}{20 kt}-1 \right)(0.0002x^2-0.024x+1.4)$ (dashed red line). Both systems have an attracting fixed point at $\A=\{100kt\}$ whose  basin of attraction is marked in grey. The two systems have same distance to threshold  and characteristic return time. Their resilience boundaries are shown in the upper right-hand side panel where also three kick flow patterns are highlighted. In the lower panels, two pairs of flow-kick trajectories are shown for population 1 on the left-hand side and population 2 on the right-hand side. In black the flow-kick trajectory relative to the disturbance pattern $(6\text{ months},-24kt)$ , in red the one relative to $(12\text{ months},-48kt)$ and in green the one relative to $(18\text{ months},-30kt)$. The disturbance pattern $(6\text{ months},-24kt)$ lies in the resilience set for population 1, but outside the resilience set of population 2. We refer the reader to~\cite{meyer2018} for further details.}
	\label{fig:resilience_boundary}
\end{figure}

 \subsubsection*{Invariance with respect to change of coordinates} Diffeomorphisms of $\R^N$ which preserve the Euclidean distance do not affect this indicator, provided that also the vector $\kappa\in\R^N$ of the disturbance pattern is reparametrized accordingly. This is due to the fact that such changes of variable  do not affect the indicator of precariousness (Definition~\ref{def:precariousness}). In particular linear changes of coordinates  $z=Qy$, with $Q$ nonsingular and orthogonal conserve the value of this indicator.

\subsection{Intensity of Attraction}
\label{intensity_subsec}
The intensity of attraction aims to measure the strength of the attraction of an attractor. It was first developed by McGehee~\cite{mcgehee1988} in the context of continuous maps acting on a compact metric space, and it has been recently extended by Meyer~\cite{meyer2019} and Meyer and McGehee~\cite{meyer2020intensity} to the continuous-time dynamics induced from an ordinary differential equation on $\R^N$. For the sake of consistency with the set-up of this work we shall present the latter.

Consider a continuous dynamical system induced by an ordinary differential equation $\dot x =f(x)$, with $f:U\subset\R^N\to\R^N$, $U$ open, and $f$ bounded and globally Lipschitz continuous on $U$. 
Moreover, consider an interval $I\subset \R$ containing zero and an essentially bounded function $g\in L^\infty([0,T],\R^N)$, that is,  $g:[0,T] \to \R^n$ such that $\inf\{c\geq 0: \mu(\{t \in [0,T]:|g(t)|>c\})=0\}<\infty$ 
and the perturbed system
\begin{equation}\label{meyer_PP}
\dot{x}=f(x)+g(t).
\end{equation}
Equation~\eqref{meyer_PP} is well-defined (it is in fact a special type of so-called Carathéodory differential equations) and it satisfies the conditions of existence and uniqueness for the associated Cauchy problems~\cite{coddington1955}. In fact, it is also assumed that for any $x_0\in\R^N$ the solution $x_g(\cdot,x_0)$ of~\eqref{meyer_PP} with $x(0)=x_0$ can be continuously extended to $[0,T]$. Notice also that the function
\[
\phi:[0,T]\times\R^N\times L^\infty\to \R^N,\quad (t,x_0,g)\mapsto \phi(t,x_0,g)=x_g(t,x_0)
\]
vary continuously in $[0,T]\times\R^N\times L^\infty$~\cite{meyer2020intensity}. As a matter of fact, under the given assumptions, the problem also induces a continuous skew-product flow~\cite{longo2018topologies,longo2017topologies,longo2019weak}. The fundamental idea behind the concept of intensity is to conceive each bounded set of perturbing functions $g\in L^\infty$ as a class of controls, and to investigate which parts of the phase space can be reached through them. In this way one is able to relate the transient dynamics of a system to the minimal class of controls, which are able to drive the points of an attractor away from it for all future time. In order to state the mathematical definition, let us recall some notation. Given $S \subset \R^n$,  we  denote by $ \phi(t,S,g)$ the set $\bigcup_{x_0 \in S}\phi(t,x_0,g)$ and call the \textit{reachable set} from  $S$ via controls of norm $r>0$ within the interval of time $I$, the set
	\begin{equation}
	P_r(S,[0,T])=\bigcup_{t \in [0,T]} \bigcup_{g \in F_r([0,T])} \phi(t,S,g),
	\end{equation}
where $F_r([0,T]) =\{g\in L^\infty ([0,T],\R^n)\colon \norm{g}_\infty<r\}$.

\begin{defn}[Intensity of attraction]
\label{intensity_def}
	The \textit{intensity of attraction} $\mathcal{I}(\A,[0,T])$ of an attractor $\A$ is defined as: 
	\begin{equation}
	\mathcal{I}(\A,[0,T]) = \sup\{r \geq 0 \mid P_r(\A,[0,T])\subset K \subset \BA \text{ for some compact }  K\subset\R^N \}.
	\end{equation}
\end{defn}
Interestingly, the intensity of attraction provides information on the persistence of an attractor not only against a time-dependent perturbation. Meyer et al.~prove that for any vector field, which differs in norm from $f$ less than its intensity of attraction, a compact positively invariant subset of $\R^N$ exists which contains the attractors of both systems (see~\cite[Theorem 6.12]{meyer2020intensity}). 
\subsubsection*{Invariance with respect to change of coordinates} Linear changes of coordinates  $z=Qy$, with $Q$ nonsingular and orthogonal conserve the value of this indicator.

\subsection{Expected escape times}\label{sec:escape_times} 
It is a well-known fact that the presence of a stochastic perturbation can have a non-negligible impact on the dynamics of a system and possibly trigger a critical transition~\cite{kuehn2011mathematical}. 
Dennis et al.~\cite{dennis2016allee} who, in turn, refer to Gardiner~\cite{gardiner1985handbook} for the main idea, suggest to judge the resilience of an attractor for a system subjected to white noise via the notion of mean first passage time, also known as expected escape time. 

\begin{defn}[Expected escape times]\label{defn:exp_escape_times}
Consider a scalar stochastic differential equation 
\begin{equation}\label{eq:scalar_stochastic}
 dX =f(X)\,dt+\sqrt{\nu(X)}\,dW(t),   
\end{equation}
where $\nu:\R\to\R^+$ is continuous, and $ W$ is a standard Brownian motion. If~\eqref{eq:scalar_stochastic} has a stationary distribution with probability density $p(x)$, then the area under $p(x)$ between $x_0$ and $x_1$ gives the long-term proportion of time that the process $X$ will spend in the interval $(x_0,x_1)$. The \emph{expected escape times} correspond to the mean first passage times for a process starting at $x_0$ to increase (resp. decrease) to $x$,
\[
\begin{split}
\tau_1(x)&=2\int_{x_0}^x\frac{\int_0^yp(z)\,dz}{\nu(x)p(y)}\,dy\quad\text{for }x_0<x, \\ 
\tau_2(x)&=2\int^{x_0}_x\frac{\int_y^\infty p(z)\,dz}{\nu(x)p(y)}\,dy\quad\text{for }x_0>x.
\end{split}
\]
\end{defn}
The expected escape times represent the mean amount of time necessary for the process $X$ to attain the value $x$ starting from $x_0<x$ and $x_0>x$, respectively. As explained in~\cite{dennis2016allee}, one or both these integrals can be infinite. For example if $x=0$ acts like an absorbing value and the system lies close to it, there is a positive probability that the first escape time takes an infinite value and so $p(x)$ is not integrable near $0$. An analogous reasoning holds for the second integral if $\infty$ acts as an absorbing boundary. In this case the right tail of $p$ is not integrable. 

\subsubsection*{Invariance with respect to change of coordinates} The expected escape times are presented for scalar stochastic equations. A linear change of coordinates in this context would mean a scaling of the only available dependent variable. For example consider $Y=aX$, with $a>0$. Then one obtains the stochastic differential equation $dY =af(Y/a)\,dt+\sqrt{a^2\nu(Y/a)}\,dW(t)$. Note that if $p(z)$ is the stationary distribution for $X$, then $\widetilde p(z)=p(z/a)/a$ is the stationary distribution for $Y$. Then, the expected escape times for the process $Y$ starting at $y_0$ are 
\[
\begin{split}
\widetilde \tau_1(y)&=2\int_{y_0}^y\frac{\int_0^\mu \widetilde p(z)\,dz}{a^2\nu(y/a)\widetilde p(\mu)}\,d\mu\quad\text{for }y_0<y, \\ 
\widetilde\tau_2(y)&=2\int^{y_0}_y\frac{\int_\mu^\infty \widetilde p(z)\,dz}{a^2\nu(y/a)\widetilde p(\mu)}\,d\mu\quad\text{for }y_0>y.
\end{split}
\]
For example, let us consider the first integral; then we have that for $y_0<y$ and using $\widetilde p(z)=p(z/a)/a$, and suitable changes of variables,
\[
\begin{split}
\widetilde \tau_1(y)&=2\int_{y_0}^y\frac{\frac{1}{a}\int_0^\mu  p(z/a)\,dz}{a\nu(y/a) p(\mu/a)}\,d\mu
=  2\int_{y_0/a}^{y/a}\frac{\int_0^{a\eta}  p(z/a)\,dz}{a\nu(y/a) p(\eta)}\,d\eta=2\int_{y_0/a}^{y/a}\frac{\int_0^{\eta}  p(\zeta)\,d\zeta}{\nu(y/a) p(\eta)}\,d\eta\\
\end{split}
\]
which is equal to $\tau_1(y/a)$ for the process $X$ starting at $y_0/a$. An analogous result holds for $\tau_2$. General diffeomorphic reparametrizations of $\R$ do not in general preserve the value of the indicators.

\subsection{Discussion} 

\begin{itemize}[leftmargin=*]
\item Although the indicators based on the transient dynamics add an additional layer to the analysis of resilience, their mutual comparison as well as their comparison with other classes of resilience indicators is a task that remains largely unaccomplished. One of the few quantitative exceptions is given by the comparison between the characteristic return time (Defintion~\ref{local_char_time}) and the return time (Definition~\ref{return_time_general}) showcased by Cottingham and Carpenter~\cite{cottingham1994} for a model of summer phosphorus cycling in north temperate lakes. According to their simulations, the characteristic return time identifies a   planktivore-dominated food web as more resilient than a piscivore-dominated one, whereas the opposite result is obtained through the return time. This single example shows that no qualitative agreement between these indicators can be generally assumed. As a matter of fact, the return time seems to be considered as a more reliable indicator since the estimation given by the reciprocal of the characteristic return time completely neglects the transient deviation after the perturbation which can be considerably large~\cite{deangelis1989}.

\item In 1979, Harrison~\cite{harrison1979} introduced the concept of \emph{stress period} for a differential equation, that is, a compact interval of time in which the problem is perturbed and one can measure the ability of the system to resist displacement, and the rate of recovery to the original condition once the perturbation is over. In the original formulation (presented in subsection~\ref{harrison_res_elas_subsec}) the perturbation consists only of a time dependent piece-wise-constant change of a parameter. However, it can be easily generalized as exposed in Remark~\ref{rmk:harrison_generalized} and included among the indicators of resilience focusing on transient dynamics.

\item  Stochastic tools have come to play an increasing role in the analysis of persistence of certain dynamical features. In this review, we limited the presentation to the expected escape times in Definition~\ref{defn:exp_escape_times} which explicitly appear in the resilience literature. However, it seems clear to us that other analytical tools from stochastic analysis will eventually appear in the context of resilience. For example, the identification of the most likely escape trajectory has been treated by Forgoston \& O.~Moore~\cite{forgoston2018primer} by   identifying a heteroclinic connection for the associated twice-the-dimension Euler-Lagrange problem, and it might be regarded as a stochastic analogous of the resistance of gradient systems. More recently, Slyman and Jones~\cite{slyman2022rate} have applied a similar technique in the context of rate-induced tipping, a critical transition that we shall treat in Subection~\ref{subsec:R-tipping}. Similarly, the role of propagation of uncertainties in the occurrence of a bifurcation can be regarded as a transient indicator which extend the distance to bifurcation indicator in Definition~\ref{dist_to_bif_def} (see for example Kuehn and Lux~\cite{kuehn2021uncertainty} or Lux et al.~\cite{lux2021uniting}).

\end{itemize}

\section{Variation of parameters}
\label{sec:parameters}

In addition to the previous notions of resilience, one can also view parameters as static phase space variables, which are changed infinitely slowly. Due to their special role, one can then ask, how one may define resilience in the parameter space. Therefore, the term resilience acquires a wider significance as it does not limit to the capability of a system to endure a perturbation and revert to a desirable state but to the overall persistence of the dynamical relations which determine the qualitative behaviour of all the trajectories. Let us set some common notation for this section. We shall consider a family of continuous dynamical systems induced by a parametric family of ordinary differential equations
\begin{equation}\label{eq:parametric-ode}
\dot x =f(x,\lambda),\quad x\in\R^N,\,\lambda\in\Lambda\subset\R^M,
\end{equation}
with $f\in C(\R^N\times\R^M,\R^N)$ satisfying the assumptions in {\rm Section~\ref{sec:prelimin}} for all $\lambda\in\Lambda$. When we fix a parameter $\lambda_0\in\R^M$, we will refer to the respective equation by the symbol~\eqref{eq:parametric-ode}$_{\lambda_0}$ and  denote by $\phi_{\lambda_0}$  the local flow on $\R^N$ induced by~\eqref{eq:parametric-ode}$_{\lambda_0}$.
\par\smallskip

This section contains three subsections. In subsection~\ref{dist_to_bif_subsec}, we present the indicator \emph{distance to bifurcation}. In subsection~\ref{harrison_res_elas_subsec}, we present three indicators: \emph{resistance}, \emph{elasticity} and \emph{persistence}. Lastly, the relation between resilience and tipping points in nonautonomous systems is explored in subsection~\ref{subsec:R-tipping}. It should be noted that in subsections~\ref{harrison_res_elas_subsec} and~\ref{subsec:R-tipping}, the considered variation of parameters is ``dynamic" meaning that in place of $\lambda\in\Lambda$ in~\eqref{eq:parametric-ode}, we shall use a function $\gamma:\R\to\Lambda$, $t\mapsto\gamma(t)$.

\subsection{Distance to Bifurcation}\label{dist_to_bif_subsec}

Bifurcation theory is a classical branch of nonlinear dynamical systems which studies the lack of topological equivalence in phase space upon (smooth) change of one or more of the system's parameters. The classic theory for autonomous dynamical systems encompasses the bifurcations of local steady states, periodic orbits,  as well as homoclinic and heteroclinic orbits~\cite{kuznetsov2013elements}. For its characteristics, a bifurcation point entails a profound change in the dynamics of the system, which may include the change of stability or the disappearance of a considered attractor. In practical terms, a bifurcation point can sometimes be undesirable, especially when the system exhibits hysteresis or irreversibility~\cite{carpenter1999}.  It is therefore natural that the distance to a bifurcation point has been suggested as a measure of resilience~\cite{dobson1993,ives2007}.

\begin{defn}[Distance to bifurcation]\label{dist_to_bif_def}
Let $\A$ be an attractor for the dynamical system induced by~\eqref{eq:parametric-ode}$_{\lambda_0}$, and  denote by $\Lambda_{bif}(\A)\subset\R^M$ the set of bifurcation points for $\A$. The \textit{distance to bifurcation}  for $\A$ is defined as
\[
D_{\rm{bif}}(\A)=\inf_{\lambda^* \in \Lambda_{bif}} |\lambda_0-\lambda^*|. 
\]
\end{defn}

Notably, distance to bifurcation can be interpreted as a distance to threshold (see Definition~\ref{def_DT}) in parameter space. 

 \subsubsection*{Invariance with respect to change of coordinates} No general class of diffeomorphisms of the phase space leaves this indicator unaltered.

\subsection{Resistance, elasticity and persistence}\label{harrison_res_elas_subsec}

Harrison~\cite{harrison1979} examines the transient behaviour of a system which undergoes an instantaneous change of parameters and returns to the original values,  just as instantly, after an interval of time. The system's response during and after the perturbation are suggested as indicators of the system's capacity to withstand a change and to recover. Harrison names these features respectively \textit{resistance} and resilience (although, to avoid confusion, we will use Orians's nomenclature~\cite{orians1975} and call Harrison's resilience \textit{elasticity}). Using Webster's words~\cite{webster1975}, the first describes ``the ability to resist displacement", whereas the latter ``the rate of recovery to the original condition". With the notation set at the beginning of the section we have the following definitions:

\begin{defn}[Resistance]\label{hison_res}
	 Given an attractor $\A$ for~\eqref{eq:parametric-ode}$_{\lambda_0}$, and fixed $T>0$ such that for all $\lambda\in\Lambda$, and $a\in\A$, $\phi_\lambda(\cdot,a)$ is defined at least in $[0,T]$, the \emph{resistance} of  system~\eqref{eq:parametric-ode}$_{\lambda_0}$ at  $\A$ with respect to $\Lambda$ and the \emph{stress period} $[0,T]$ is defined as:
	\[R(\A,\Lambda, T)=\sup_{\substack{t\in [0,T],\, a\in\A_{\lambda_0}\\ \lambda\in\Lambda}}|\phi_{\lambda}(t,a)-\phi_{\lambda_0}(t,a)|.\]
\end{defn}

\begin{defn}[Elasticity]\label{def:elasticity}
	  Given an attractor $\A$ for~\eqref{eq:parametric-ode}$_{\lambda_0}$, and fixed $T>0$ such that $\phi_{\lambda}(T,\A)\subset\mathcal{B}(\A)$ for all $\lambda\in\Lambda$, consider the continuous function $\Phi_{\lambda_0}:[T,\infty)\times\A\times\Lambda
	  \to\R^N$ defined by
	  \[
	  \Phi_{\lambda_0}(t,a, \lambda):=|\phi_{\lambda_0}\big(t-T,\phi_{\lambda}(T,a)\big)-\phi_{\lambda_0}(t,a)|.
	  \]
	 The \textit{elasticity} of the system~\eqref{eq:parametric-ode}$_{\lambda_0}$ at  $\A$ with respect to $\Lambda$ and the \emph{stress period} $[0,T]$ is defined as:
	\[E(\A,\Lambda, T)=\sup_{\substack{t\in [T,\infty),\, a\in\A_{\lambda_0}\\ \lambda\in\Lambda}}\left(\frac{1}{\Phi_{\lambda_0}(t,a, \lambda)}\frac{d\, \Phi_{\lambda_0}(t,a, \lambda)}{dt}\right).
	\]
\end{defn}

\begin{figure}
		\centering
		\includegraphics[trim={0.5cm 9.5cm 2cm 10cm},clip,scale=0.44]{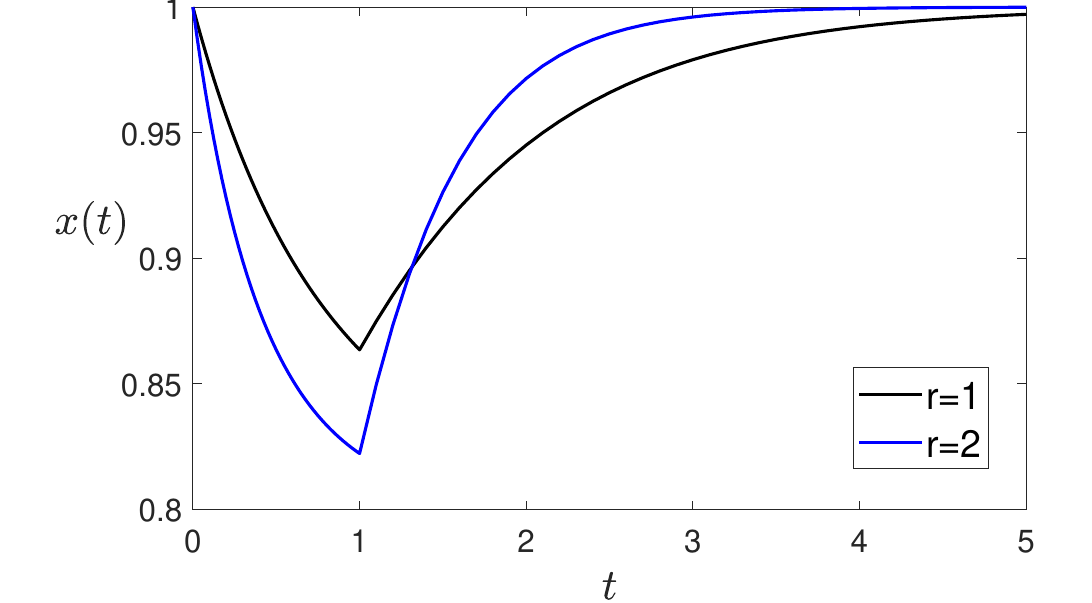}
		\caption{Effect of a stress period of Harrison's type on two populations following the logistic model $\dot{x}=rx\left(1-x/K\right)$  with carrying capacity $K=1$, and growth rates respectively equal to $r_1=1$ and  to $r_2=2$. A stress period $[0,1]$ is applied to both models during which the carrying capacity is diminished of the $20\%$ (i.e.~$\Lambda=\{0.8\})$. The resistance of the attractor $x^*=1$ is lower for the first species than the second. Conversely, the elasticity of the attractor $x^*=1$ is higher for the first species than the second.}
		\label{log_resil_resist}
	\end{figure}

In Figure~\ref{log_resil_resist} it is possible to appreciate the complementary information provided by resistance and elasticity in a simple scalar differential model from population dynamics.

\begin{rmk}
The definitions of resistance and elasticity presented above admit a natural generalization. Instead of using the original trajectory inside the attractor as a reference, in practice it may be sufficient to record the displacement from the attractor or the rate of convergence to the attractor. This would lead to the following weaker definitions
\[
R^w(\A,\Lambda, T)=\sup_{\substack{t\in [0,T],\, a,a'\in\A_{\lambda_0}\\ \lambda\in\Lambda}}|\phi_{\lambda}(t,a)-a'|,
\]
and 
\[
	  \Phi^w_{\lambda_0}(t,a, \lambda):=\inf_{a'\in\A}|\phi_{\lambda_0}\big(t-T,\phi_{\lambda}(T,a)\big)-a'|,
	  \]
	\[E^w(\A,\Lambda, T)=\sup_{\substack{t\in [T,\infty),\, a\in\A_{\lambda_0}\\ \lambda\in\Lambda}}\left(\frac{1}{\Phi^w_{\lambda_0}(t,a, \lambda)}\frac{d\, \Phi^w_{\lambda_0}(t,a, \lambda)}{dt}\right).
	\]
\end{rmk}

Ratajczak et al.~\cite{ratajczak2017} use a setup similar to Harrison's to address the problem of the interaction between the duration and the ``intensity" of a stress period in determining a tipping of the system from the original basin of attraction to a different one (if it applies). In accordance with Harrison's nomenclature, we shall call this property persistence.

\begin{defn}[Persistence]\label{def:persistence}
The \emph{persistence} of an attractor $\A$ for~\eqref{eq:parametric-ode}$_{\lambda_0}$ with respect to a \emph{stress intensity} $\lambda\in\Lambda$ is given by
\[
P_\lambda(\A)=\sup\{T>0\mid \phi_\lambda(t,\A)\subset\mathcal{B}(\A)\text{ for all }0\le t\le T\}.
\]
The \emph{persistence} of an attractor $\A$ for~\eqref{eq:parametric-ode}$_{\lambda_0}$ with respect to a \emph{stress duration} $T>0$ is given by
\[
P_T(\A)=\sup\{\rho>0\mid  \phi_\lambda(t,\A)\subset\mathcal{B}(\A)\text{ for all }0\le t\le T\text{ and } |\lambda-\lambda_0|\le\rho\}.
\]
\end{defn}

\begin{rmk}\label{rmk:harrison_generalized}
It is clear that Harrison's formulations of resistance and elasticity as well as of persistence generalize to any perturbation on a finite time interval rather than just a piece-wise constant change of parameters. Consider~\eqref{eq:parametric-ode}$_{\lambda_0}$ for some $\lambda_0\in\R^M$, $T>0$, and a set $\Lambda$ of functions from $\R\times\R^N$ into $\R^M$ such that if $\lambda\in\Lambda$, then for every compact set $K\subset\R^N$, $\sup_{x\in K}\lambda(t,x)$ is locally integrable and $\lambda(t,x)=\lambda_0$ for all $x\in\R^N$ and $t\notin[0,T]$. Hence, one can consider the problem 
\[
\dot x =f\big(x,\lambda(t,x)\big), 
\]
and still define resistance, elasticity and persistence as above. 
\end{rmk}

 \subsubsection*{Invariance with respect to change of coordinates} Both, resistance and elasticity are invariant with respect to diffeomorphisms of the phase space which preserve the Euclidean distance. This includes in particular linear changes of coordinates  $z=Qy$, with $Q$ nonsingular and orthogonal. The persistence with respect to stress intensity is susceptible to changes of time-scale. Note laso that no specific class of diffeomorphisms of $\R^N$ preserves the persistence with respect to the stress duration.

\subsection{Resilience against rate-induced tipping}\label{subsec:R-tipping}
Sizable changes in the output of a system upon a negligible variation of the input are referred to as \emph{critical transitions} or \emph{tipping points}. Motivated by current challenges in nature and society~\cite{gladwell2006tipping,scheffer2020critical}, the study of the several mechanisms leading to a critical transition has experienced significant mathematical interest as well~\cite{kuehn2011mathematical,kuehn2013mathematical}. In recent years for example, the phenomenon of rate-induced tipping~\cite{wieczorek2011excitability} has been proposed as an alternative mechanism for a critical transition, with respect to the more classical autonomous bifurcations and noise-induced tipping~\cite{ashwin2017parameter}. Rate-induced tipping can be seen as a special type of nonautonomous bifurcation for the differential equation 
 \begin{equation}\label{eq:init-probA}
\dot x =f\big(x,\gamma(rt)\big),\qquad x\in\R^N,t\in\R,r>0
\end{equation}
obtained from~\eqref{eq:parametric-ode} through a time-dependent variation of parameters $\gamma:\R\to\R^M$   which is assumed to be bounded, continuous and asymptotically constant: $\lim_{t\to-\infty}\gamma(t)=\lambda_0,\  \lim_{t\to\infty}\gamma(t)=\lambda_\infty$.  In particular, the parameter $r>0$ represents the rate at which the time-dependent variation takes place. We shall maintain a concise presentation and leave the detailed assumptions and technical aspects for 
Appendix~\ref{appendix:rate}.

Typically the study of rate-induced tipping is carried out under the assumption that the time-dependent variation of the parameters does not cross any point of autonomous bifurcation; i.e.~we  assume that for all $\lambda$ in the image of $\gamma$, the dynamical systems induced by~\eqref{eq:parametric-ode}$_\lambda$, are all topologically equivalent.  One might expect that for any fixed attractor $\A_0$ of the \emph{past limit-problem}, and denoted by $\A_f$ the ``corresponding" attractor  of the \emph{future limit-problem}, for all $r>0$ there is an attractor of~\eqref{eq:init-probA}$_r$ which approaches $\A_0$ as $t\to-\infty$ and $\A_f$ as $t\to\infty$. However, this is not always the case and, depending on $r$, a local attractor  of~\eqref{eq:init-probA}$_r$ limiting at $\A_0$ as $t\to-\infty$ may fail to track $\A_f$ as $t\to\infty$. This phenomenon is called \emph{rate-induced tipping}, and the specific value $r^*_{\gamma}(\A_0)$ of the parameter where it occurs is called \emph{critical rate} (see Definition~\ref{def:R-tipping}). It seems therefore natural that the critical rate $r^*_{\gamma}(\A_0)$
determines a threshold of resilience for~\eqref{eq:parametric-ode}$_{\lambda_0}$ against the time-dependent change of parameters $\gamma$, in the sense that for all $r<r_\gamma^{*}$, where $r_\gamma^{*}$ is the inferior over all the attractors of~\eqref{eq:parametric-ode}$_{\lambda_0}$, the tracking of attractors from the past to the future is always verified.

This statement assumes further significance if we look at rate-induced tipping as of a \emph{nonautonomous bifurcation}. Then, the bifurcation point $r_\gamma^{*}$ can be used to construct a distance to bifurcation analogous to the one in Definition~\ref{dist_to_bif_def}.
\begin{defn}[Distance to rate-induced tipping]\label{dist_to_bif_rate}
Let $\A_0$ be an attractor of the past limit-problem, and $\A_f$ the ``corresponding" attractor  of the future limit-problem. The \textit{distance to rate-induced tipping} for \eqref{eq:init-probA} and for $\A_0$ is defined as
\[
D_{\rm{rate}}(\A_0)=|r-r^*_\gamma(\A_0)|. 
\]
\end{defn}
It should be noted, that in contrast to autonomous bifurcations, estimating the value of a critical rate can be hard if not impossible~\cite{kuehn2022estimating}.
There is, furthermore, a substantial difference from Definition~\ref{dist_to_bif_def}. The tipping point $r_\gamma^{*}$ depends on the function $\gamma$ that realizes the variation of parameters. Therefore, it  is  beneficial to investigate the continuous dependence of $r^*_\gamma$ with respect to the variation of $\gamma$. We hereby propose a first result in this direction, when all the attractors of ~\eqref{eq:parametric-ode}$_\lambda$ are hyperbolic equilibria. The result guarantees the lower semi-continuity (and if it applies the continuity) of $r^*_\gamma$ with respect to $\gamma$ in a specific set of functions: fixed $t_0\in\R$, we consider the set $\Omega(\gamma, t_0)$ of twice continuously differentiable functions $\omega:\R\to\Lambda$ such that $\omega(t)=\gamma(t)$ for all $t\le t_0$ and $\lim_{t\to\infty}\omega(t)=\lambda\in\Lambda$. For a detailed explanation of the symbols appearing in the statement and for the proof, please consult Appendix~\ref{appendix:rate}.

\begin{thm}\label{thm:res_rate_induced_tipping}
Assume that $f$ is bounded and uniformly continuous, and that there are $\ep>0$ and $\overline \delta>0$ such that, for every $\overline r>0$, a $T(\gamma,\overline r)>0$ exists for which if $\omega\in\Omega(\gamma, t_0)$ and $\|\gamma-\omega\|_{\mathcal{C}(\R,\Lambda)}<\overline\delta$, and if $\widetilde x_{\omega,\overline r}$ is a solution of $\dot x= f(x, \omega(\overline r t))$ defined on $[T(\gamma,\overline r),\infty)$ and  $y'={\textnormal{D}}f\big(\widetilde x_{\omega,\overline r}(t),\omega(\overline rt)\big)y$ has an exponential dichotomy with projector the identity on $[T(\gamma,\overline r), \infty)$, then the distance between $\widetilde x_{\omega,\overline r}$ and any solution of $\dot x= f(x, \omega(\overline r t))$ starting at time $s\ge T(\gamma,\overline r)$ at $x_0\in\R^N$ with $|\widetilde x_{\omega,\overline r}(s)-x_0|<\ep$, converges to zero as $t\to\infty$.
Moreover, assume that if $\widetilde x_{\gamma,r}$  has a rate-induced tipping point $r^*_\gamma$, such tipping point is unique and transversal. The following statements hold true.
\begin{itemize}[leftmargin=25pt]
    \item[\textbf{1}.] For all $r>r^*_{\gamma}$ there is $\delta_1>0$ such that if $\omega\in\Omega(\gamma, t_0)$ and $\|\gamma-\omega\|_{\mathcal{C}(\R,\Lambda)}<\delta_1$ then $r\ge r^*_{\omega}$.
    \item[\textbf{2}.]  Assume additionally that for all $\omega\in\Omega(\gamma, t_0)$, if $\widetilde x_{\omega,r}$ undergoes a rate-induced tipping, the tipping point  $r^*_{\omega}$ is unique and transversal. Then, for all $r<r^*_{\gamma}$ there is $\delta_2>0$ such that if $\|\gamma-\omega\|_{\mathcal{C}(\R,\Lambda)}<\delta_2$ then $r\le r^*_{\omega}$.
\end{itemize}
\end{thm}

\subsubsection*{Invariance with respect to change of coordinates} Except for specific classes of problems an explicit formula to identify rate-induced tipping points is not available. Nevertheless, it is immediate to check that any change of time scale proportionally affects such critical rates. In general, diffeomorphisms of $\R^N$ will also affect both the past and future limit problems and the intermediate dynamics.

\section{Example}\label{sec:examples}
In this section, we consider a classic one-dimensional population model with an Allee effect, i.e.~with declining individual fitness for low population densities (see Courchamp et al.~\cite{courchamp2008}).
Firstly, we derive the values of the indicators presented in Sections~\ref{local_lI_subsection} to~\ref{sec:parameters}. Then, we perform a parametric study of their behaviour, and finally compare the resilience of chosen populations with concrete parameter settings. 

\subsection*{The model}
We shall study the parametric differential problem 
\begin{equation}\label{allee_model}
\dot{x}=f(x,r,L)=rx\left(1-\frac{x}{K}\right)\left(\frac{x}{L}-1\right),
\end{equation}
where $x\in \R^+$ represents the population size, $r>0$ the intrinsic growth rate, $K>0$ the carrying capacity of the environment, and $0<L\le K$ an Allee threshold. For the sake of convenience, we will set the carrying capacity to $K=1$. For $0<L<1$ the problem has three equilibria, at $x_0=0$, $x_1=1,$ and $x_L=L$. The first two are stable and the latter unstable. Therefore, $L$ determines a threshold below which the population can not persist. 

\subsection{Resilience of the attractor $x_1=1$}\label{resil_of_allee_subsec}
In this subsection, we apply the presented indicators to the attractor $x_1=1$ of the considered model $\dot x=f(x,r,L)$. Note that when $L=1$, the system undergoes a transcritical bifurcation and the equilibria $x_L$ and $x_1$ collide. When the calculations can not be performed explicitly, the software Matlab is used for numerical computation and simulation. In particular, we use the function {\tt ode45} with the options on the relative and absolute tolerance respectively set to {\tt RelTol=1e-12} and {\tt AbsTol=1e-12}.


\subsubsection*{Local indicators}
The characteristic return time of $x_1=1$ (see Definition~\ref{local_char_time} and Remark~\ref{rmk:char_return_equil}) is easily obtained as the opposite of the reciprocal of the derivative of $f$ at $x_1$: $T_R(1)=-1/{\textnormal{D}}f(1)=L/(r-rL).$ On the other hand, notice that for scalar systems, the amplification envelope is trivially determined---in our case $\rho(1,t)=e^{{\textnormal{D}}f(1)t}$---and does not provide further information.  Concerning other local indicators, thanks to Proposition~\ref{prop:linear_ind_chain}, we immediately have the following equivalence (up to the sign) of the reactivity, stochastic invariability, deterministic invariability and reciprocal of the characteristic return time  ${-R_0=I_\mathcal{S}=I_\mathcal{D}=-{\textnormal{D}}f(1)}$.

\subsubsection*{Basin shape indicators}
The distance to threshold (Definition~\ref{def_DT}) for the attractor is $DT(x_1)=1-L$. Note also that considering the set $C=[0,1],$  the latitude in volume (Definition~\ref{latitude_volume}) and the basin stability (Definition~\ref{basin_stab_def}) (with a uniform density $\rho(x)=\chi_{[0,1]}$) with respect to $C$  both yield the same value as the distance to threshold. Furthermore, the latitude in width (Definition~\ref{def_lw}) has an infinite value since the basin is unbounded on the right-hand side. 

\subsubsection*{Nonlinear transient dynamics}
The return time (Definition~\ref{return_time_general}) does not admit an explicit closed form. Therefore, we employ a routine to estimate it numerically. In Figure~\ref{table_result} we show an approximation of the mean return time from the set $C=(L,1)$ estimated via a Monte Carlo simulation (uniform sampling on the set $(L+10^{-7},1)$ with $10000$ points) depending on the parameters $r\in[0.01,0.5]$ and $L\in[0.05,0.95]$. For all points, the return time is approximated via a finite-time integration ending when the considered trajectory achieves $1-10^{-10}$.
 Note that the considered problem can also be interpreted as a gradient system. Thus, we are able to derive its resistance $W$ (Definition~\ref{potential_def}) as 
\[
W(x_1)=\int_L^{1} f(x,r,L)\, dx=-\frac{(L-1)^3(L+1)r}{12L}.
\]
The intensity of attraction $\mathcal{I}$ (Definition~\ref{intensity_def}) is obtained by calculating the maximum of $f$ on the interval $[L,1].$ We show the behaviour of $\mathcal{I}(x_1)$ depending on $r\in[0.01,0.5]$ and  $L\in[0.05,0.95]$ in Figure~\ref{table_result_cont}.
\subsubsection*{Variation of parameters}
The distance to bifurcation (Definition~\ref{dist_to_bif_def}) for this problem is $D_{\rm bif}(x_1)=1-L$  and therefore it coincides with the distance to threshold. In order to study the resistance and elasticity of $x_1=\{1\}$, we will consider a perturbation of the system's environmental capacity. Particularly, we aim to record the system's response to a reduction of the carrying capacity to $K_p=0.9,$ for an interval of time $[0,T]$, with $T=10$. If $K_p<L,$ the perturbed model is not ecologically meaningful. Hence, we will restrict the analysis only to $L<0.9.$ Since the problem at hand is scalar, the resistance (Definition~\ref{hison_res}) can be calculated simply as $R(x_1, K_p,T)=1-\phi_{K_p}(T,x_1)$. On the other hand, the elasticity (Definition~\ref{def:elasticity}) corresponds to the supremum of the weighted vector field $-f(x)/(1-x)$ over the interval $[\phi_{K_p}(T,x_1),1),$ which is attained at $\phi_{K_p}(T,x_1).$ 
In regards to persistence (Definition~\ref{def:persistence}), fixed $K_p>L,$ the persistence of $x_1$ with respect to the stress intensity $K_p$ is $P_{K_p}=+\infty,$ since the perturbed trajectories can not be driven out of the basin of attraction $(K_p,+\infty) \subset (L,+\infty).$ The same applies to the persistence with respect to a stress duration under the constraint that $K_p>L$, that is $P_T(x_1)=+\infty$. Alternatively, one can consider a simultaneous and consistent change of carrying capacity $K_p$ and Allee threshold, in the sense that if $K_p$ is reduced of a certain percentage, then also the Allee threshold is reduced of the same percentage. In this case, the value of $P_T(x_1)$ will depend on $T>0$.

\subsubsection{Parametric study}\label{subsub_parametric}
We perform a parametric analysis of~\eqref{allee_model}, with the aim of portraying the behaviour of some of the considered indicators as the transcritical bifurcation point $L=1$ is approached (results are summarized in Figures~\ref{table_result} and~\ref{table_result_cont}).

We focus on the parameter subspace given by $r\in[0.01,0.5]$, $L\in[0.05,0.95]$ and calculate seven indicators (see previous Subsection~\ref{resil_of_allee_subsec} for details): reciprocal of characteristic return time $EV=\frac{1}{T_R}$, distance to threshold $DT$, mean return time $T_R^{mean}(1,(L+10^{-7},1)),$
resistance $W,$ intensity of attraction $\mathcal{I}$, resistance $R$ and elasticity $E$. The last two indicators are calculated for the perturbation of the environmental capacity $K_p=0.9$ for a stress period $T=10.$  

\begin{figure}
	\centering
	\includegraphics[trim={0.5cm, 7.5cm, 0.5cm, 7cm},clip,scale=0.68]{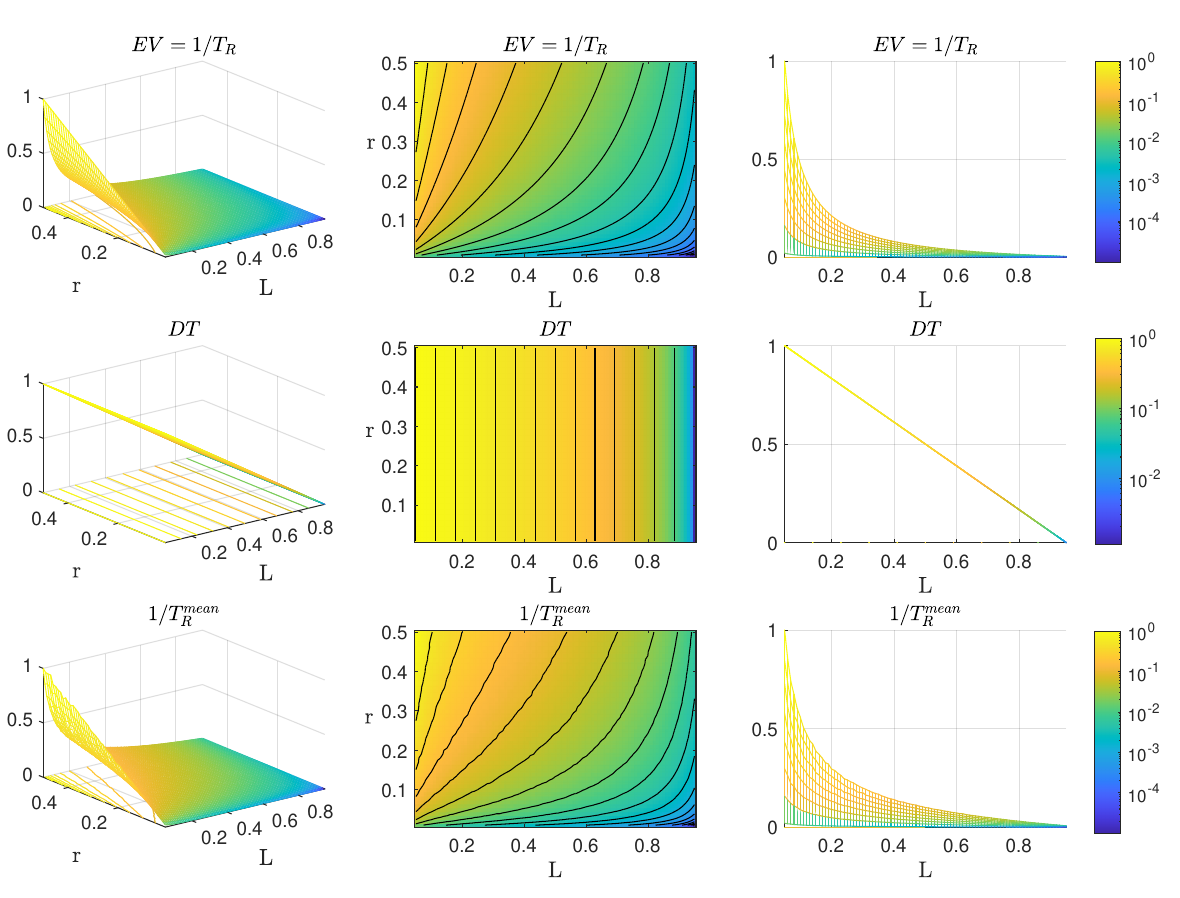}
	\caption{Comparison of the indicators for the attractor ${x_1=1}$ of the model~\eqref{allee_model} upon the variation of the parameters ${r\in[0.01,0.5]}$ and $L\in[0.05,0.95]$. From top to bottom: opposite of the real part of the dominant eigenvalue $EV=\frac{1}{T_R}$, distance to threshold $DT$, and the reciprocal of the mean return time $T_R^{mean}(1,(L+10^{-7},1))$. On the left, we see the scaled indicator values for the given parametric subspace. In the middle a heat map for the parametric subspace, where yellow indicates the highest and dark blue the lowest values of the indicators (see the logarithmic color scale on the right). Black lines correspond to the contour lines, which mark the parameter combinations with the same value of the indicators. If the environmental changes drive the parameters along the contour curves, all indicators, apart from $DT$ are not able to capture the approaching bifurcation independently. On the right, we see the dependence of the indicator values on the parameter $L.$ The three pictures on each row showcase the same surface viewed by different viewpoints.}
	\label{table_result}
\end{figure}

\begin{figure}
	\centering
	\includegraphics[trim={0.5cm, 5.5cm, 0.5cm, 5cm},clip,scale=0.68]{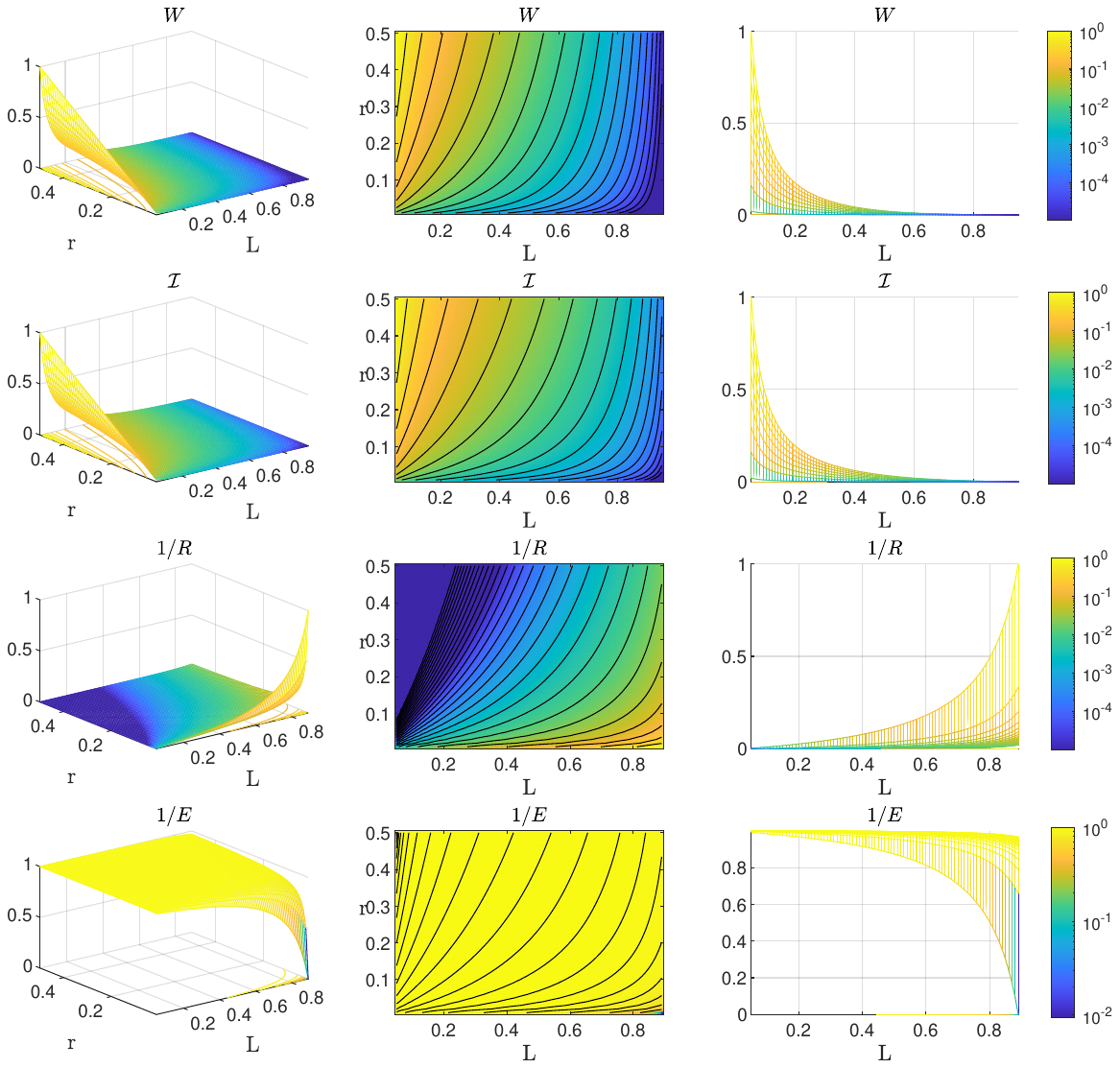}
	\caption{Continuation of the comparison of the indicators for different parameter choices of the model~\eqref{allee_model}: resistance (potential) $W,$ intensity of attraction $\mathcal{I},$ resistance $R$ and elasticity $E$ (please note that in order to have higher resilience corresponding to a higher numerical value, we worked with the reciprocal of $R$ and $E$). We consider perturbation of the environmental capacity to $K_p=0.9,$ for time $t=[0,10].$ On the left, we see the scaled indicator values for a given parametric subspace: $[0.01,0.5]= r \times L=[0.05,0.95]$ for resistance (potential) and intensity and  $[0.01,0.5]= r \times L=[0.05,0.89]$ for resistance and elasticity.} In the middle there is a heat map for the parametric subspace, where yellow indicates the highest and blue the lowest values of the indicators (see the logarithmic color scale on the right). Black lines correspond to the contour lines, which mark the parameter combinations with the same values of the indicator. On the right, we see the dependence of the indicator values on the parameter $L.$ The three pictures on each row showcase the same surface viewed by different viewpoints. We see that resistance and elasticity behave in an opposite manner, elasticity being higher for lower values of $\{r,L\}$ while resistance being lower. 
	\label{table_result_cont}
\end{figure}

   For a better comparison of the indicators, the estimates shown in  Figures~\ref{table_result} and~\ref{table_result_cont} are scaled to lie in $[0,1]$ by $(x-x_{\min})/(x_{\max}-x_{\min}),$ where $x$ represents the indicator, and minimum $x_{\min}$ and maximum $x_{\max}$ corresponds to the local minimum and maximum of the indicator on the considered subset of parameters.
    For the same reason, instead of $T_R^{mean},R,E$, we represent their reciprocal so that the value $1$ always corresponds to the highest estimate of resilience. Additionally, resistance and elasticity we further restrict the parameter space so to guarantee that these indicators are well-defined. 

It is possible to appreciate that upon fixing a value $r>0$ and increasing $L$, all the considered indicators capture a loss of resilience for the system as the bifurcation point approaches. The only exception is given by the resistance which increases since the rate of exponential decay towards the attractor decreases in the nearby of the bifurcation point. This phenomenon is also directly related to the so-called critical slowing-down effect~\cite{dakos2015resilience, strogatz2018nonlinear,scheffer2009}.
On the other hand, there are also parameter combinations, indicated by the black contour lines, where the indicator stays constant.
\subsubsection{Comparison of five species}\label{subsub_comparison}
Finally, we aim to showcase the behaviour of the indicators with respect to five different species growth strategies for the Allee effect model~\eqref{allee_model}. The different values of the parameters as well as the respective graphs of $f(x,r,L)$ for $x\in[0,1]$ are shown in Figure~\ref{system_allee-params}. Given the previous considerations, we shall focus on the following indicators of resilience: the opposite of the dominant eigenvalue $EV=1/T_R,$ distance to threshold $DT,$  mean return time $T_R^{mean}(1,(L+10^{-7},1)),$ resistance (potential) of a gradient system $W,$ intensity of attraction $\mathcal{I},$ resistance $R$ and elasticity $E.$ 
\begin{figure}[]
	\begin{subfigure}{0.35\textwidth}
		\centering
		\begin{tabular}{@{} *3l @{}}    \toprule
			\emph{Species} & \emph{r}& \emph{L}  \\\midrule
			Species $1$ &  $0.5$ & $0.2$  \\ 
			Species $2$ & $1.3$ & $0.3$ \\ 
			Species $3$ &$ 2.5$& $0.4$\\
			Species $4$ &$5$ &$0.6$\\
			Species $5$ &$10$ &$0.7$\\
			\bottomrule
			\hline
		\end{tabular}
	\end{subfigure}
	\begin{subfigure}{0.5\textwidth}
		\includegraphics[scale=0.3,trim={3cm 9.54cm 3cm 9.5cm},clip,height=4.6cm]{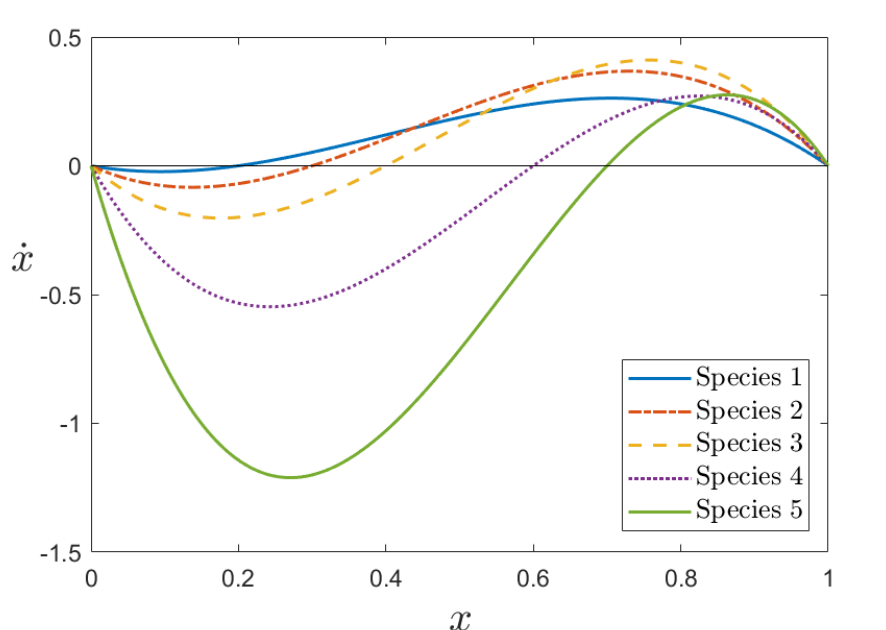}
	\end{subfigure}
	\caption{Five population strategies with different choices of parameters in equation~\eqref{allee_model}. The table on the left-hand side contains the parameters for each species. The plot on the right-hand side showcases the graphs of $f(x,r,L)$ for $x\in[0,1]$ depending on the chosen values of parameters $r$ and $L$.}\label{system_allee-params}
\end{figure}

The numerical values for these indicators are shown in Figure~\ref{table_results_5species}. In all cases, $x_1=1$ is the reference attractor. Each row corresponds to a single indicator, each column to a species. The chromatic scale applies row-wise as the values of the indicators have not been normalized. Darker tones of blue correspond to higher values of resilience. It is apparent how answering the question ``Which species is the most resilient?" is not trivial. The ranking of the resilience of each species varies depending on the considered indicator.

In particular, it is interesting to notice that, not only there is no overall agreement on the resilience ranking, but some indicators ($EV$ and $DT$ for example) provide antipodal responses. It is possible to single out three main clusters: a) $EV$ and $1/T_R^{mean}$, b) $1/E$, $I$ and $W$, c) $1/R$ and  $DT$. With the exception of species ``4" all the other species are considered as the most resilient for at least one of the represented indicators.

\begin{figure}
	\centering
	\includegraphics[scale=0.35,trim={0cm 0cm 0cm 1cm},clip]{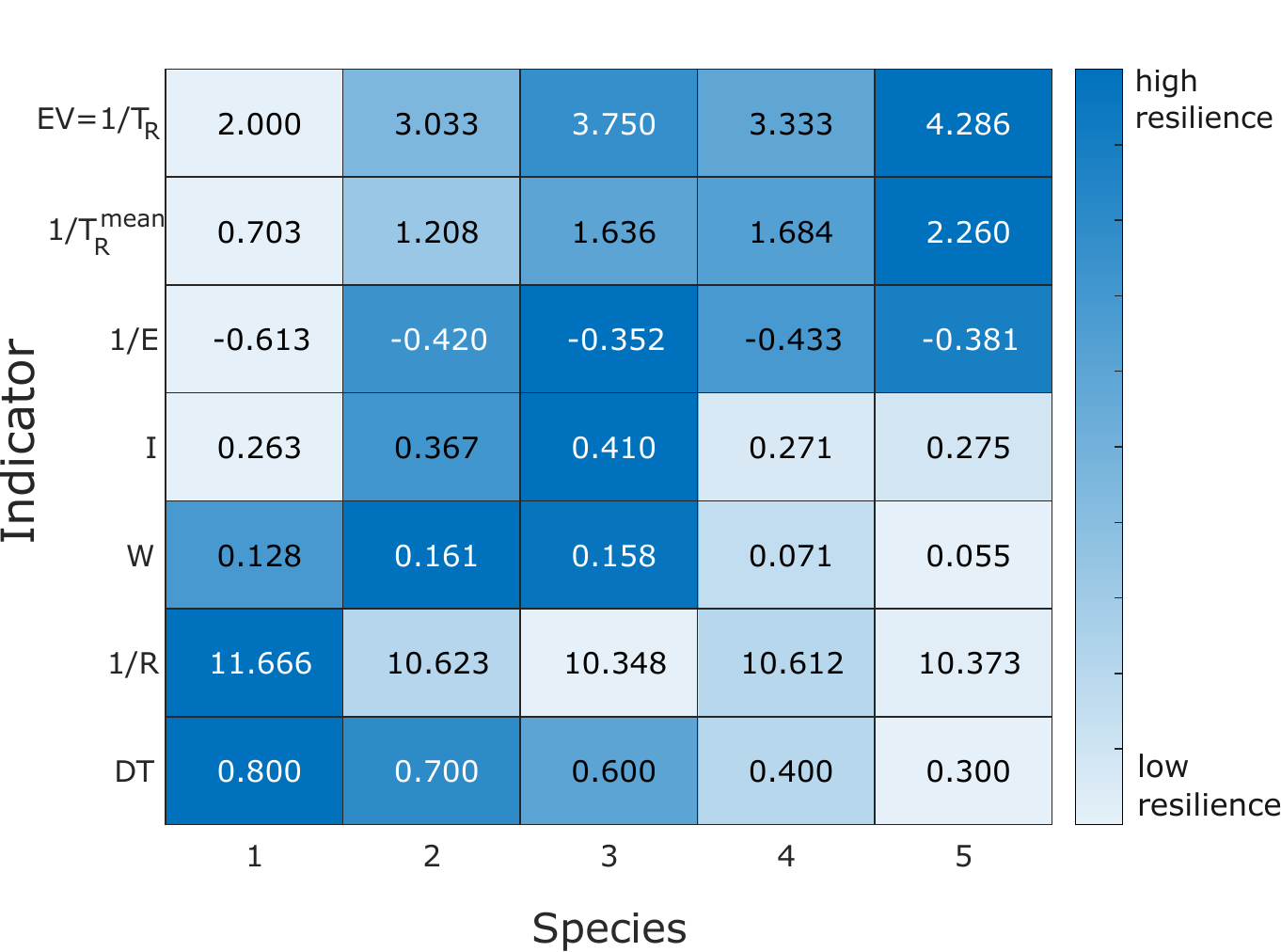}
	\caption{Each row represents one indicator and each column one species.  The chromatic scale applies row-wise as the values of the indicators have not been normalized. Darker tones of blue correspond to higher values of resilience. With the exception of species 4 all the other species are considered as the most resilient for at least one of the represented indicators.}\label{table_results_5species}
 \end{figure}
 
 \begin{figure}[]
	\begin{subfigure}{0.45\textwidth}
		\includegraphics[scale=0.45,trim={4.1cm 9.2cm 0 9cm},clip]{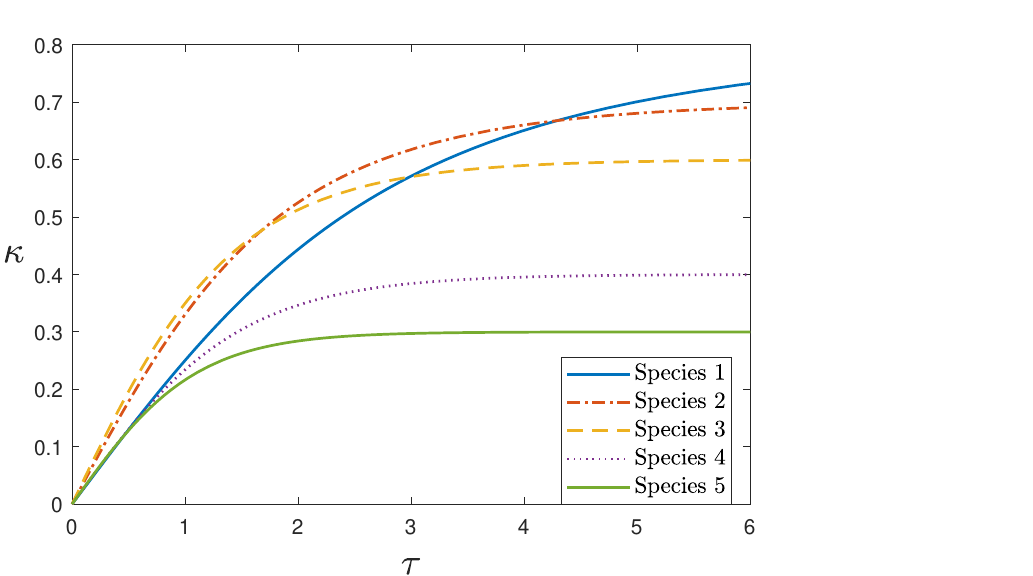}
	\end{subfigure}
	\begin{subfigure}{0.52\textwidth}
		\includegraphics[scale=0.44,trim={0cm 0cm 0 0cm},clip]{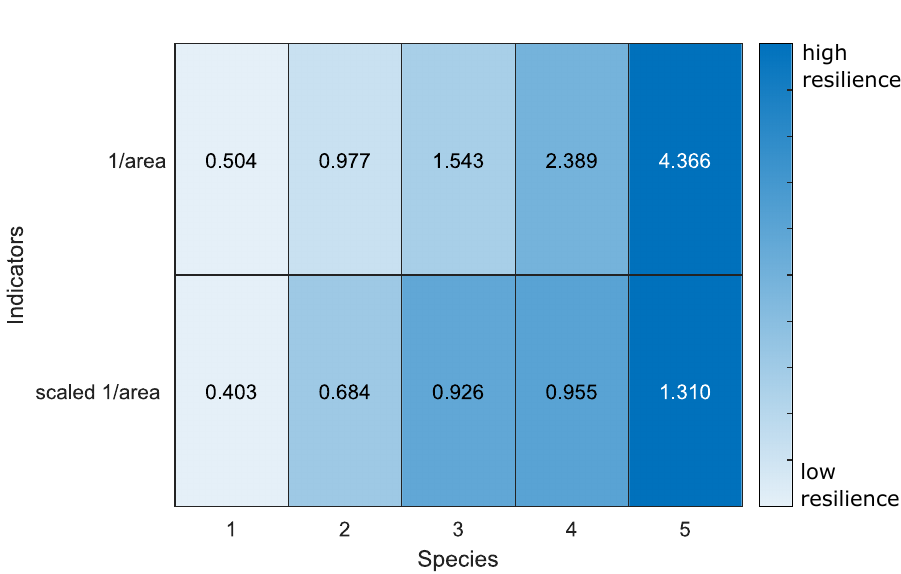}
	\end{subfigure}
	\caption{On the left-hand side, the resilience boundaries of five species (see Figure~\ref{system_allee-params}) modelled through~\eqref{allee_model}. On the right-hand side, first row, the numerical values of the areas between the resilience boundary and the respective distance to threshold calculated for each species. The second row contains the same values now divided by the respective distances to threshold   $DT(x_1)=1-L$. The color gradient should be intended row-wise. Darker shades of blue correspond to a higher resilience. }
	\label{resil_bound_results}
\end{figure}

 Finally, in Figure~\ref{resil_bound_results} we can see the resilience boundaries for the considered five species. As studied in~\cite{meyer2018}, it is possible to extrapolate a measure of resilience for the attractor by calculating the area between the resilience boundary and the line $\kappa= DT(x_1)=1-L$ (see Subsection~\ref{resil_boundary_subsec}). However, since the distances to threshold for the five species are different it seems natural to ask if the obtained values are representative and comparable. The right panel of 
Figure~\ref{resil_bound_results} shows the numerical values obtained for the indicator before and after dividing them by the respective distances to threshold.

\subsection{Resilience of the attractor $x_0=0$.}\label{subsec:allee_attr_0}

In this subsection, we perform a (shorter) analysis of resilience for the attractor $x_0=0$. In order to solve the singularity arising from posing $L=0$ in \eqref{allee_model}, we shall study the problem 
\begin{equation}\label{eq:allee_model_2}
\dot{x}=g(x,r,L)=rx\left(1-x\right)\left(x-L\right), \quad x\ge0,
\end{equation}
which is equivalent to \eqref{allee_model} when $L>0$ (and $K=1$), via the variation of time-scale $t\mapsto Lt$. When $L$ tends to zero, the continuous variation of the solutions can be deduced using Tikhonov's theorem~\cite{tikhonov1952systems}. Note that when $L=0$, \eqref{eq:allee_model_2} undergoes a transcritical bifurcation and the equilibria $x_L$ and $x_0$ collide. 


\subsubsection*{Local indicators}
Reasoning as in the previous section we obtain the characteristic return time of $x_0=0$ (see Definition~\ref{local_char_time} and Remark~\ref{rmk:char_return_equil}) as the reciprocal of the derivative of $g$ at $x_0$: $T_R(0)=-1/{\textnormal{D}}g(0,r,L)=(rL)^{-1}.$ Here again,  the amplification envelope  $\rho(0,t)=e^{{\textnormal{D}}g(0,r,L)t}$ does not provide further information, while the same equivalence (up to the sign) of the reactivity, stochastic invariability, deterministic invariability and reciprocal of the characteristic return time  holds thanks to  Proposition~\ref{prop:linear_ind_chain}.

\subsubsection*{Basin shape indicators}
The distance to threshold (Definition~\ref{def_DT}) for the attractor is $DT(x_0)=L$. Note also that considering the set $C=[0,1],$  the latitude in volume (Definition~\ref{latitude_volume}) and the basin stability (Definition~\ref{basin_stab_def}) (with a uniform density $\rho(x)=\chi_{[0,1]}$) with respect to $C$  both yield the same value as the distance to threshold. Furthermore, since the model has no biological meaning for $x<0$, the latitude in width (Definition~\ref{def_lw}) must be restricted to the positive cone, and therefore it would coincide with the distance to threshold. 

\subsubsection*{Nonlinear transient dynamics}
Same as for the attractor $x_1=1$, we use a numerical routine to approximate the return time (Definition~\ref{return_time_general}). The results are now shown in  Figure~\ref{table_result_2}. We have used the same setting and stopping criteria for the Monte Carlo simulation as before (uniform sampling on the set $(0, L-10^{-7})$ with $10000$ points) depending on the parameters $r\in[0.01,0.5]$ and $L\in[0.05,0.95]$. 
The resistance $W$ (Definition~\ref{potential_def}) is obtained as 
\[
W(x_0)=\int_L^{0} g(x,r,L)\, dx=\frac{L^3(2-L)r}{12}.
\]
The intensity of attraction $\mathcal{I}$ (Definition~\ref{intensity_def}) is obtained by calculating the maximum of $g$ on the interval $[0,L].$ We show the behaviour of $\mathcal{I}(x_0)$ depending on $r\in[0.01,0.5]$ and  $L\in[0.05,0.95]$ in Figure~\ref{table_result_2}.
\subsubsection*{Variation of parameters}\label{subsub_v_parameters_0}
The distance to bifurcation (Definition~\ref{dist_to_bif_def}) is ${D_{\rm bif}(x_0)=L}$  and therefore it coincides with the distance to threshold. We also note that a change of the carrying capacity $0<K_p$ for \eqref{eq:allee_model_2} does not affect the equilibrium at $x_0=0$, in the sense that $x_0=0$ keeps being an equilibrium for the new system. Consequently, the resistance of $x_0$ for such variations is always zero while the elasticity is not even well-defined---the quantity $\Phi_{K_p}(t,a,1)$ in Definition~\ref{def:elasticity} is equal to zero. 
Concerning the persistence, a reasoning analogous to the one for the attractor $x_1=1$ holds, in the sense that no variation of the carrying capacity is able to change the equilibrium at $0$ nor to drive it outside its basin of attraction (unless the carrying capacity itself reaches zero, which, however, would make the model inconsistent).

\subsubsection{Parametric study}

We perform a parametric analysis of~\eqref{eq:allee_model_2}, with the aim of portraying the behaviour of some of the considered indicators as the transcritical bifurcation point $L=0$ is approached (results are summarized in Figure~\ref{table_result_2}). The same remarks about the setup, reparametrization and comparability of the indicators as in~\ref{subsub_parametric} hold. However, the indicators of resistance and elasticity are not calculated, since a variation of the carrying capacity does not alter their values as explained in subsection~\ref{subsub_v_parameters_0}. Therefore, the calculated indicators are: reciprocal of characteristic return time $EV=\frac{1}{T_R}$, distance to threshold $DT$, mean return time $T_R^{mean}(0,(0,L-10^{-7})),$
resistance $W$ and intensity of attraction $\mathcal{I}$. Interestingly enough, the indicator of mean return time captures a phenomenon that remained elusive to all the other indicators: as $L$ approaches the values $1$, where a transcritical bifurcation for the equilibria $x_1$ and $x_L$ occurs, the leading Lyapunov exponent for $x_L$ tends to zero and a ``slowing-down" of nearby trajectories appear. Consequently, solutions starting nearby the ``weakly" unstable equilibrium $x_L$ linger close to it for longer intervals of time. This increases the mean return time to $x_0=0$.

\begin{figure}
	\centering
	\includegraphics[trim={1.5cm, 4.2cm, 0.2cm, 3.5cm},clip,scale=0.68]{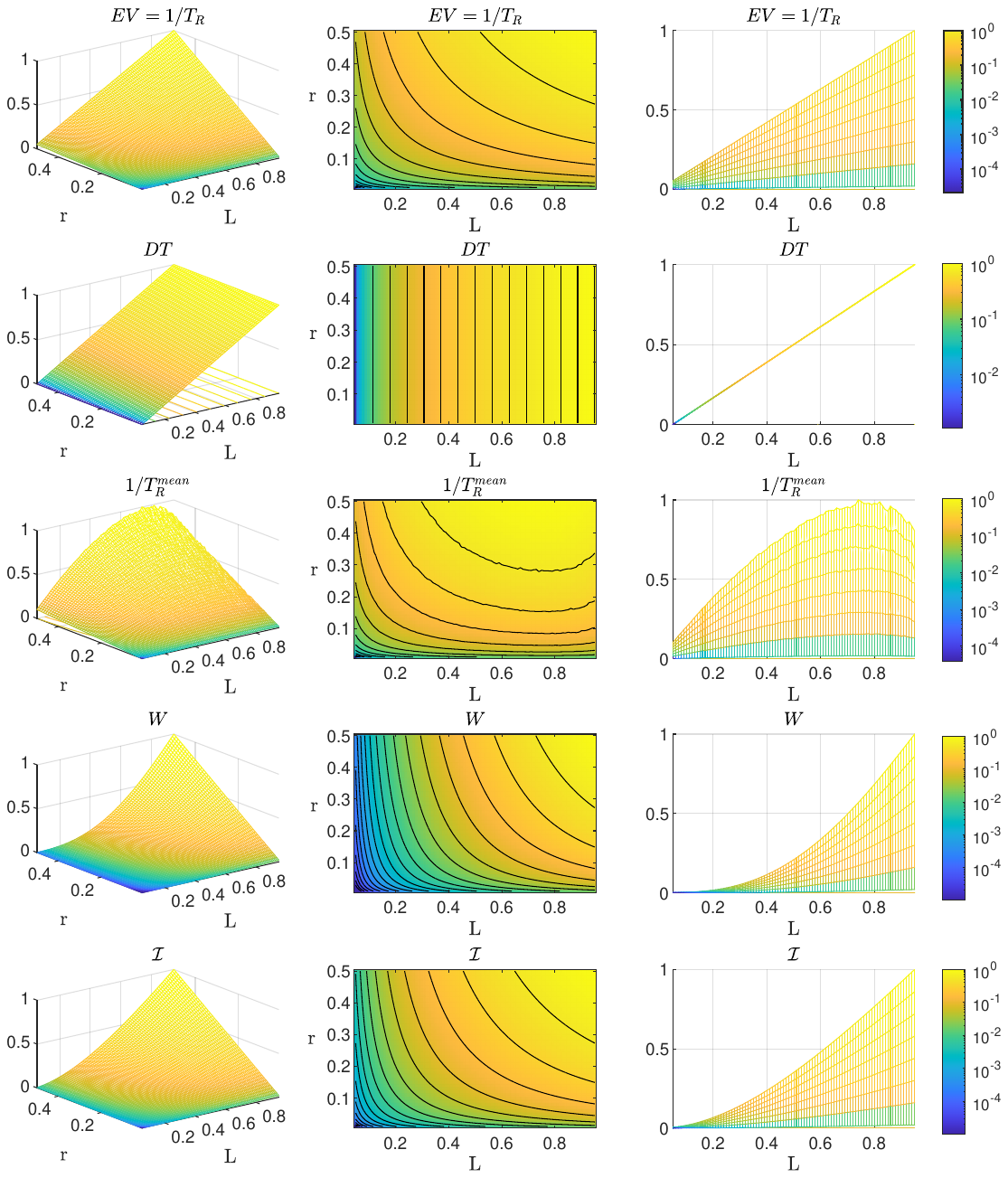}
	\caption{Comparison of the indicators for the attractor ${x_0=0}$ of the model~\eqref{eq:allee_model_2} upon the variation of the parameters ${r\in[0.01,0.5]}$ and $L\in[0.05,0.95]$. From top to bottom:  $EV=\frac{1}{T_R},$  $DT,$  $T_R^{mean}(0,(0, L-10^{-7})),$ $W(0)$ and $I(0)$. On the left, we see the scaled indicator values for the given parametric subspace. In the middle a heat map for the parametric subspace, where yellow indicates the highest and dark blue the lowest values of the indicators (see the logarithmic color scale on the right). Black lines correspond to the contour lines, which mark the parameter combinations with the same value of the indicators. If the environmental changes drive the parameters along the contour curves, all indicators, apart from $DT$ are not able to capture the approaching bifurcation independently. On the right, we see the dependence of the indicator values on the parameter $L.$ The three pictures on each row showcase the same surface viewed by different viewpoints.}
	\label{table_result_2}
\end{figure}

\subsection{Synthesis, comparisons and remarks}In this section, we carried out the analysis of a scalar population model with Allee effect~\eqref{allee_model} in terms of the resilience of its attractors at $x_1=1$ and $x_0=0$. The different outcomes given by some of the indicators of resilience presented in this paper have been tested against the variation of the growth rate and of the Allee threshold (both treated as parameters). For the attractor $x_1$, the indicators were also compared across five different species with fixed pairs of parameters (see part~\ref{subsub_parametric} and~\ref{subsub_comparison}). The resulting absence of a non-ambiguous answer to the question \emph{``which species is more resilient?"}, may seem contradictory at first, but it is a known feature of resilience. The measured resilience depends on both, the specific dynamic feature under analysis, and the class of acting perturbations. As pointed out by Carpenter et al.~\cite{carpenter2001} and Tamberg et al.~\cite{tamberg2022} one needs to carefully define the context and questions before coming to conclusions about resilience. 

For instance, if we are interested in the survival of a species, that is, the resilience of the equilibrium $x_1$, the following conclusions may be draft. If we are solely interested in the largest one-time shock that the species can withstand (for example a sudden ecological catastrophe), but we are not concerned with the dynamics of the recovery, the most suitable indicator is distance to threshold $DT$ (or other basin shape indicators--- for their comparison see Proposition~\ref{prop:DT_width_volume}). As explained in subsection~\ref{resil_of_allee_subsec} and graphically shown in Figure~\ref {table_result}, $DT$ is exclusively dependent on the value of the Allee threshold $L$ and is not changed by the magnitude of the growth rate $r$. From a biological point of view, this might be interpreted as the number of individuals in the species that have to suddenly disappear for the species to go extinct, independently of its growth rate. If the Allee threshold is crossed, the growth rate will not help. The lower $L$ is, the weaker the Allee effect, which means the species population can sustain bigger shocks from which they can eventually recover. Thus, we observe the highest resilience for values of $L$ close to $0$ and the lowest for high Allee thresholds around $1$ (see Figure~\ref{table_result}).

Conversely, when the focus is on the speed of recovery after a single ``small" perturbation of the number of individuals in the species, a local indicator such as $EV$ may better fulfill the task. Indeed, $EV$ corresponds to the exponential rate of convergence towards a hyperbolic attracting solution. In our model, this indicator depends on both the intrinsic growth rate $r$ and threshold $L$ (see subsection~\ref{resil_of_allee_subsec} and Figure~\ref{table_result}). A higher growth rate $r$ results in a faster recovery, whereas the higher the Allee threshold $L$, the slower the recovery. As extensively explained in Section~\ref{local_lI_subsection}, all the local indicators are reliable only when ``small" perturbations are taken into account. Whenever a larger perturbation is considered the local indicators such as $EV$ might not be informative enough and we have to turn to other transient dynamics indicators. One straightforward choice is the global version of the local indicator $EV:$ the mean return time $T_R^{mean}.$  In our model, $T_R^{mean}$ of the attractor $x_1=1$ shows overall similar behaviour as the local indicator $EV$. This is 
mostly due to the extreme simplicity of the model at hand. In general, however, these indicators may provide substantially different results---examples of this can be found in part~\ref{subsub_comparison}.

Next, we focus on continuous and repeated perturbations. In biological terms, we can consider a perturbation caused by a permanent outflow of species (imagine for example fishing, hunting, or illness). This can be modeled as $\dot{x}=f(x)+g(t),$ where $f(x)$ is the unperturbed dynamics and $g(t)$ is the perturbation term (e.g.~fishing). Then, the indicator intensity $\mathcal{I}$ will find the critical magnitude of the outflow term $g(t)$ after which the species can still recover. The overall trend seen in Figure~\ref{table_result_cont} is similar to that of $EV$---the higher the reproductive speed in terms of growth rate $r$, the better can the species compensate for the outflow, although there are still differences with other indicators (see part~\ref{subsub_comparison}).

Another type of perturbation, different from those already mentioned, is recurring stress on the species numbers. This is where periods of no disturbance are followed by strong and rapid stress on the species population. Think, for instance, of a seasonal illness or periodic fishing. This can be modeled in terms of periodic, discrete shocks of some strength. In this case, the resilience boundary is the most obvious choice of indicator. It captures both the magnitude of shocks that a species can sustain to not fall below the critical population numbers given by the Allee threshold, as well as, the recovery capabilities of the species during the periods without perturbation. High values of growth rate $r$ help the recovery rate, and low values of Allee threshold $L$ permit severe shocks to the species number to still lie in the recoverable numbers (see part~\ref{subsub_comparison} for comparison).

A different interpretation for the transient dynamics is given by the indicator of resistance (potential) which quantify how much work needs to be put into eradicating the species at hand (see Figure~\ref{table_result_cont}). In our context this is possible due to the specific structure of the equation that can be written as a gradient problem. It is not surprising that this indicator provides estimates in accordance with the other indicators of resilience for transient dynamics---the higher the growth rate, the harder it is to eliminate the species. Likewise, the weaker the Allee effect, the harder it is to eradicate species below the threshold $L.$

On the other hand, the changing environment is another type of stress species factor. In the previous subsection, we discussed the effect of lowering the species' carrying capacity. This could be caused, for example, by deteriorating environmental conditions. In terms of resilience, we might ask the following questions: once the environmental stress factor has ended, what is the impact on the species' population numbers? If the species can recover (which is not always possible in general) how fast the recovery is? The former can be answered by measuring the resistance (Harrison) $R$, the latter by elasticity $E$ (see Figure~\ref{table_result_cont}). Resistance (Harrison) quantifies how the population numbers decrease after the stress ends. What we can notice is that populations with higher growth rates are more sensitive to the decline of carrying capacity. The same is true for low values of the Allee threshold. These species decline faster due to the way the model is defined where the overall tendency of the vector field to have a higher magnitude also leads to a faster decrease in the population numbers.
Elasticity, on the other hand, determines the speed of recovery after the stress ends by measuring the species population's ability to flourish again. Notice that for each species it finds the slowest recovery speed, thus, the worst-case scenario. The starting position of recovery is described by resistance, therefore, they are related. We notice that species with higher growth rates $r$, are better equipped to recover, by quickly reproducing their numbers, also for low Allee thresholds. Apart from resistance, this trend is consistent with other indicators mentioned. We can see that the species can have different strategies to fight temporal environmental stresses: typically,  resistance and elasticity are opposite aspects of resilience. Finally, we like to stress that similar comments can be made for the equibrium $x_0=0$, which represents the extinction state. The interest of the extinctions state lies in the opportunity to control infesting species or diseases. Through the parametric study carried out in Subsection~\ref{subsec:allee_attr_0}, and the numerical evidence showcased in Figure~\ref{table_result_2}, we illustrated, once again, the different outcomes of resilience indicators. Particular emphasis is posed on the mean return time which displays a new phenomenon: the long persistence of an infesting species (which is eventually destined to extinction) when the Alee threshold approaches the carrying capacity. In practical terms, this case seems quite rare in reality where the Alee threshold is usually considered in $(0,0.5)$.

\section{Summary and conclusions}\label{sec:conclusions}
The notion of dynamical resilience has prompted a huge scientific interest ever since its first introduction by Holling~\cite{holling1973}. Given the variety of indicators that arose under this headline, a unitary or preferred definition seems unattainable at this point. Nevertheless, a systematic classification of the available indicators using rigorous mathematical formalism was long needed. In this paper, we intentionally try to give a broad perspective beyond ecological models---where resilience appeared in the first place---and provide a structural view of the matter from the abstract perspective of dynamical systems theory. For the same reason, we also include those indicators that have a rigorous formalization but can often be difficult to calculate analytically for real systems.

We divide the indicators of resilience into four groups and present them through rigorous mathematical formalism: local linear indicators (Section~\ref{local_lI_subsection}), indicators that describe the shape of
the basin (Section~\ref{secbasinshape}), indicators that characterize the transient dynamics in the basin (Section~\ref{sec:transient}), and indicators tailored to parameter changes (Section~\ref{sec:parameters}). Advantages and disadvantages of each class and single indicators are discussed in detail in each section.

Our approach not only allows to generalize the available indicators to local attractors beyond equilibria and periodic orbits, but also to set up a common framework to compare them and from which further research can proceed in an organized fashion. Let us provide concrete examples supporting these statements.

In Section~\ref{secbasinshape} the most classical indicators of ecological resilience are presented and contrasted. Thanks to a rigorous formalization, it is possible to prove simple---and yet so far missing---results that relate such indicators and use the distance to threshold as a common benchmark for all the others. Moreover, this helps in creating further relations between this class of indicators and more recent ones (e.g.~between precariousness and the resilience boundary). On the other hand, we have for example seen that distance to threshold can be also used to interpret the simplest indicator of resilience in parameter space, distance to bifurcation (see subsection~\ref{dist_to_bif_subsec}). This fact gave us the opportunity to naturally extend the concept of resilience to rate-induced tipping, a nonautonomous bifurcation which can appear when a adiabatic change of a parameter is substituted by  a time-dependent one (see Subsection~\ref{subsec:R-tipping}).

To keep the presentation as broadly accessible to a variety of different disciplines, we only touched upon some very interesting ideas whose presentation, however, requires a more technical treatment. For example, we decided to maintain the core of this article autonomous and use nonautonomous dynamical systems only where strictly necessary. Yet, the inherent nature of resilience requires a better understanding of the role that time plays in the perturbation of a system, either deterministic or random. Therefore, we foresee that nonautonomous dynamical systems theory will play an important role in a further understanding of resilience, as also the most recent contributions in the area seem to suggest (see Sections~\ref{sec:transient} and~\ref{sec:parameters}). 
For example, we used the notions of exponential dichotomy and hyperbolic solutions to reinforce the local indicators in Section~\ref{local_lI_subsection}. In fact, the classical indicators using the variational equation to infer asymptotic rates of convergence of hyperbolic equilibria or periodic orbits remain reliable under sufficiently small perturbations due to persistence of hyperbolicity in this generalized sense. In analogy, the study of more general exponentially stable attractors requires the notion of persistence of normally hyperbolic invariant manifolds. We have intentionally avoided treating this topic explicitly in order to provide a shorter and more streamlined presentation of the basic structural aspects of defining different notions of resilience. We are also certain that the interested reader with knowledge in normally hyperbolic invariant manifolds can carry out the relative extensions. 

We also like to stress the fact that this paper does not cover the (equally interesting) interpretations of resilience, which do not fit into the setting established in Section~\ref{sec:prelimin}. Among them, there is for example statistical model-based indicators (see Adamson et al.~\cite{adamson2020}), the analysis of pattern-formation in ecological systems (e.g.~\cite{bastiaansen2020effect,reyer2015forest}), certain resilience concepts in economics (e.g.~\cite{briguglio2009economic}) and in the social sciences (e.g.~\cite{iianalysing}). It would be very interesting---and highly nontrivial---to integrate these approaches within a coherent mathematical framework and we hope that a dedicated research effort will go in this direction.

In summary, the formalization and classification carried out in this work is important to ensure a more reliable, quantitatively comparable and reproducible study of resilience in dynamical systems which can stimulate further research in the area and facilitate quantitative application-based comparisons.

\appendix
\section{Rate induced tipping and proof of Theorem~\ref{thm:res_rate_induced_tipping}}\label{appendix:rate}

Hereby, we shall give a more precise presentation of the assumptions and results on resilience against rate-induced tipping that were briefly presented in Subsection~\ref{subsec:R-tipping}. We deal with differential equations of the type, 
\[
\dot x =f\big(x,\gamma(rt)\big),\qquad x\in\R^N,t\in\R,r>0,
\]
where 
$f\in C^2(\R^N\times\R^M,\R^N)$ and $\gamma\in C(\R,\Lambda)$, with $\Lambda\subset\R^M$ compact and connected, is such that for some $\lambda_0,\lambda_\infty\in\Lambda$,
\[
\lim_{t\to-\infty}\gamma(t)=\lambda_0,\quad \lim_{t\to\infty}\gamma(t)=\lambda_\infty\quad\text{and}\quad\lim_{t\to\pm\infty}\gamma'(t)=0.
\]
Moreover, we assume that for all $\lambda\in\Lambda$ the dynamical systems induced by~\eqref{eq:parametric-ode} are topologically equivalent and their dynamics is completely determined by hyperbolic fixed points (for attractors which are not fixed points see for example~\cite{alkhayuon2018rate,kaszas2019tipping,longo2021critical,longo2021rate}).

Let $x_0^s\in\R^N$ be any fixed stable hyperbolic equilibrium  for~\eqref{eq:parametric-ode}$_{\lambda_0}$ and $x_\lambda^s\in\R^N$ be the ``corresponding'' stable hyperbolic equilibrium for ~\eqref{eq:parametric-ode}$_{\lambda}$. This means that if $\{h_{\lambda}:\R^N\to\R^N\}$ is a $\lambda$-continuous family of time-preserving homeomorphisms guaranteeing the topological equivalence, then $x_\lambda^s=h_{\lambda}(x_0^s)$. 
Thanks to Theorem 2.2 in~\cite{ashwin2017parameter}, for every $r>0$, there is a unique maximal solution of~\eqref{eq:init-probA}$_r$ $\widetilde x_{\gamma,r}:(-\infty,\beta_{\gamma, r})\to\R^N$, with $\beta_{\gamma, r}\in\R\cup\{\infty\}$, such that $\lim_{t\to-\infty}\widetilde x_{\gamma,r}(t)=x_0^s$. 
Additionally, this solution is \emph{locally pullback attractive} (a weak form of attractivity which holds only for the past and is specific of nonautonomous systems; see~\cite{kloeden2011} for more information). Notice also that, thanks to Lemma 2.3 in~\cite{ashwin2017parameter}, there is $r_0>0$ such that for all $0<r<r_0$, $\beta_{\gamma, r}=\infty$ and $\lim_{t\to\infty}\widetilde x_{\gamma,r}(t)=x_{\lambda_\infty}^s$. We can now give the definition of rate-induced tipping.

\begin{defn}\label{def:R-tipping}
Under the introduced notation and assumptions, we say that the locally pullback attracting solution $\widetilde x_{\gamma,r}$  \emph{undergoes a rate-induced tipping at $r_\gamma^{*}(x_0^s)>0$} if  $\beta_{\gamma, r}=\infty$ and $\lim_{t\to\infty}\widetilde x_{\gamma,r}(t)=x_{\lambda_\infty}^s$ for all $r\in(0,r^*_\gamma(x_0^s))$, and either $\beta_{\gamma, r_\gamma^{\scaleto{*}{3pt}}(x_0^s)}<\infty$ or $\lim_{t\to\infty}\widetilde x_{\gamma,r_\gamma^{\scaleto{*}{3pt}}(x_0^s)}(t)\neq x_{\lambda_\infty}^s$. Moreover, we will call a rate-induced tipping \emph{transversal} if there is $\delta>0$ such that for all $r\in[r_\gamma^{*}(x_0^s), r_\gamma^{*}(x_0^s)+\delta]$, $\beta_{\gamma, r}<\infty$ or $\lim_{t\to\infty}\widetilde x_{\gamma,r}(t)\neq x_{\lambda_\infty}^s$.
\end{defn}
The critical rate correspond to the value 
\[
r_\gamma^{*}(x_0^s)=\sup\left\{\rho>0 \ \bigg|\  \beta_{\gamma, \rho}=\infty \text{ and }\lim_{t\to\infty}\widetilde x_{\gamma,\rho}(t)=x_{\lambda_\infty}^s\text{ for all }\rho\le r\right\}\in\R\cup\{\infty\}.
\]

We can now prove Theorem~\ref{thm:res_rate_induced_tipping} which is a first result guaranteeing the lower semi-continuity (and if it applies the continuity) of $r^*_\gamma$ with respect to $\gamma$ in a specific set of functions: fixed $t_0\in\R$, we consider the set $\Omega(\gamma, t_0)$ of twice continuously differentiable functions $\omega:\R\to\Lambda$ such that $\omega(t)=\gamma(t)$ for all $t\le t_0$ and $\lim_{t\to\infty}\omega(t)=\lambda\in\Lambda$.\par\bigskip

\noindent\emph{Proof of Theorem~\ref{thm:res_rate_induced_tipping}}
In order to prove \textbf{1} let us reason by contradiction. Assume that there is $\overline r >r^*_{\gamma}$ and a sequence $(\omega_n)_\nin$ in $\Omega(\gamma, t_0)$ such that $\|\gamma-\omega_n\|_{\mathcal{C}(\R,\Lambda)}<1/n$ and $\overline r <r^*_{\omega_n}$ for all $n\in\N$. Now, note that the sequence of functions $f\big(x,\omega_n(t)\big)$ converges to $f\big(x,\gamma(t)\big)$ in the compact-open topology. Therefore, thanks to Kamke's lemma (see Sell~\cite{sell1971topological}) the sequence of solutions $(\widetilde x_{\omega_n,\overline r})_{n\in\N}$ converges, up to a sub-sequence, to $\widetilde x_{\gamma,\overline r}$  uniformly on compact intervals contained in the maximal interval of definition of the latter. In fact, due to construction, $\widetilde x_{\omega_n,\overline r}(t)=\widetilde x_{\gamma,\overline r}(t)$ for all $\nin$ and $t\le t_0$, and thus the uniform convergence holds on all intervals of the form $(-\infty,T]$ with $T$ in the maximal interval of definition of $\widetilde x_{\gamma,\overline r}$. 

Now, for all $\nin$ let $\lambda_n\in\Lambda$ be such that $\lim_{t\to\infty} \omega_n(t)=\lambda_n$. Consequently, $\lim_{n\to\infty}\lambda_n=\lambda_\infty$. 
Moreover, denoted by $x^s_{\lambda_n}$ the stable hyperbolic equilibrium of ~\eqref{eq:parametric-ode}$_{\lambda_n}$ such that $\lim_{t\to\infty} \widetilde x_{\omega_n,\overline r}(t)= x^s_{\lambda_n}$, since $\overline r <r^*_{\omega_n}$, we have that 
\[
\lim_{n\to\infty}x^s_{\lambda_n}=\lim_{n\to\infty}h_{\lambda_n}(x^s_0)=h_{\lambda_\infty}(x^s_0)=x_{\lambda_\infty}^s.
\]
Therefore, the sequence $(\widetilde x_{\omega_n,\overline r})_{n\in\N}$ is in fact uniformly bounded on $\R$ and thus it must be $\beta_{\gamma, \overline r}=\infty$, and furthermore,  $\widetilde x_{\gamma,\overline r}$ is bounded. 
Now, since the $\omega$-limit set of $\widetilde x_{\gamma,\overline r}$ is nonempty, thanks to the cocycle property, the continuity of the skew-product flow induced by~\eqref{eq:init-probA}$_{\overline r}$ (see~\cite{sell1971topological}), and the fact that the dynamics of ~\eqref{eq:parametric-ode}$_{\lambda_\infty}$ is completely determined by hyperbolic fixed points, there is a fixed point $\overline x_{\lambda_\infty}$ of ~\eqref{eq:parametric-ode}$_{\lambda_\infty}$ such that $\lim_{t\to\infty}\widetilde x_{\gamma,\overline r}(t)=\overline x_{\lambda_\infty}$.
Moreover, by assumption, it must be $|\overline x_{\lambda_\infty}- x_{\lambda_\infty}^s|>\ep>0$. On the other hand, notice that, for the $\ep>0$ given by the assumptions, since $| \gamma(\overline r t)-\lambda_\infty|$ tends to zero as $t\to\infty$, there are $t_1=t_1(\gamma, \overline r, \ep)>0$, and due to the persistence of hyperbolic solutions (see Theorem 3.8 in~\cite{potzsche2011nonautonomous}),  a solution  $ \overline\varphi$ of~\eqref{eq:init-probA}$_{\overline r}$ such that $ \overline \varphi$ is defined at least in $[t_1,\infty)$, and $\|\overline x_{\lambda_\infty} -  \overline \varphi\|_{\mathcal{C}([t_1,\infty),\R^N)}<\ep/5$. 
Furthermore, since $\lim_{t\to\infty}\widetilde x_{\gamma,\overline r}(t)=\overline x_{\lambda_\infty}$, there is $t_2=t_2(\gamma, \overline r, \ep)\ge t_1>0$ such that $\|\widetilde x_{\gamma,\overline r}- \overline \varphi\|_{\mathcal{C}([t_2,\infty),\R^N)}<\ep/5$. Now, consider $s>\max\{t_2, T(\overline r)\}$. Thanks to the continuous variation of the solutions, $|\widetilde x_{\omega_n,\overline r}(s)-\widetilde x_{\gamma,\overline r}(s)|<\ep/5$ for all $n\in\N$ sufficiently big. However, again for the persistence of the hyperbolic solutions, there is $n_0\in\N$ such that for all $n>n_0$ a solution $\overline \psi_n$ of $\dot x= f(x, \omega_n(\overline r t))$ exists which is defined at least in $[t_2,\infty)$ and $\|\overline \psi_n- \overline \varphi\|_{\mathcal{C}([t_2,\infty),\R^N)}<\ep/5$ and $\lim_{t\to\infty}|\widetilde x_{\omega_n,\overline r}(t)- \overline \psi_n(t)|=0$. However, at least for one $n>n_0$ this is in contradiction with the fact that $\lim_{t\to\infty}\widetilde x_{\omega_n,\overline r}(t)=x^s_{\lambda_n}$ because for all $t>\max\{t_2, T(\overline r)\}$
\[
\begin{split}
    0<\ep<|\overline x_{\lambda_\infty}-x^s_{\lambda_\infty}|<&
    |\overline x_{\lambda_\infty}-\overline \phi(t)|+|\overline \phi(t)-\overline\psi_n(t)|+\\
    +&|\overline\psi_n(t)-\widetilde x_{\omega_n,\overline r}(t)|+|\widetilde x_{\omega_n,\overline r}(t)-x^s_{\lambda_n}|+|x^s_{\lambda_n}-x^s_{\lambda_\infty}|<\ep.
\end{split}
\]
Therefore, \textbf{1} must be true.
\par\smallskip

In order to prove \textbf{2}, we shall reason similarly; once more by contradiction. Assume that there is $\overline r <r^*_{\gamma}$ and a sequence $(\omega_n)_\nin$ in $\Omega(\gamma, t_0)$ such that $\|\gamma-\omega_n\|_{\mathcal{C}(\R,\Lambda)}<1/n$ and $\overline r >r^*_{\omega_n}$ for all $n\in\N$. 
Therefore, thanks to Kamke's lemma (see Sell~\cite{sell1971topological}) the sequence of solutions $(\widetilde x_{\omega_n,\overline r})_{n\in\N}$ converges, up to a sub-sequence, to $\widetilde x_{\gamma,\overline r}$  uniformly on compact intervals contained in the maximal interval of definition of the latter.
Additionally, as for \textbf{1}, for the $\ep>0$ given by the assumptions, since $| \gamma(\overline r t)-\lambda_\infty|$ tends to zero as $t\to\infty$, there are $t_1=t_1(\gamma, \overline r, \ep)>0$ and a solution  $ \overline\varphi$ of~\eqref{eq:init-probA}$_{\overline r}$ such that $ \overline \varphi$ is defined at least in $[t_1,\infty)$, and $\| x_{\lambda_\infty}^s -  \overline \varphi\|_{\mathcal{C}([t_1,\infty),\R^N)}<\ep/3$. Moreover, since $\overline r <r^*_{\gamma}$, there is also $t_2\ge t_1$ such that $\|\widetilde x_{\gamma,\overline r}-\overline \varphi\|_{\mathcal{C}([t_2,\infty),\R^N)}<\ep/3$. 
On the other hand, the persistence of hyperbolic solutions guarantees also that there is $n_0\in\N$ such that for all $n>n_0$ a solution $\overline \psi_n$ of $\dot x= f(x, \omega_n(\overline r t))$ exists which is defined at least in $[t_2,\infty)$ and $\|\overline \psi_n- \overline \varphi\|_{\mathcal{C}([t_2,\infty),\R^N)}<\ep/3$. In particular, thanks to the assumptions, for all $n>n_0$, it must be $\lim_{t\to\infty}|\overline \psi_n(t)- x_{\lambda_n}^s|=0$, where $x_{\lambda_n}^s=h_{\lambda_n}(x_{\lambda_\infty}^s)$. Now, consider $t>t_2$. From the continuous variation of the solutions, we obtain the contradiction. Indeed, there must be at least one $n>n_0$ such that 
\[
0\le |\widetilde x_{\omega_n,\overline r}(t)-\overline \psi_n(t)|\le |\widetilde x_{\omega_n,\overline r}(t)- \widetilde x_{\gamma,\overline r}(t)|+|\widetilde x_{\gamma,\overline r}(t)-\overline \varphi(t)|+|\overline \varphi(t)-\overline \psi_n(t)|,
\]
and the left-hand side can be made smaller than $\ep$ by choosing $n$ sufficiently large. In this case, however, the assumptions guarantee that $\lim_{t\to\infty}|\widetilde x_{\omega_n,\overline r}(t)- x_{\lambda_n}^s|=0$ which is against the fact that  $\overline r >r^*_{\omega_n}$ for all $n\in\N$. Therefore,we obtain the thesis. 
\qed

\subsection*{Acknowledgments}
We thank an anonymous referee for her/his comments and suggestions, which have been included in the current version of the paper.

\bibliographystyle{siam}
\addcontentsline{toc}{chapter}{Bibliography}
\bibliography{main}

\end{document}